\documentclass[12pt]{amsart}

\usepackage[utf8]{inputenc}
\usepackage[english]{babel}
\usepackage[a4paper,twoside,top=1.2in, bottom=1.1in, left=0.8in, right=0.8in]{geometry}

\usepackage{graphicx}
\usepackage{caption} % Better handling of empty figure captions.

\usepackage{amsmath}
\usepackage{amsthm}
\usepackage{amssymb}
\usepackage{esint} % For the average integral
\usepackage{mathrsfs} % Just for cool sigma-algebra
\usepackage{faktor} % For quotients
\usepackage{dsfont} % For identity
\usepackage{xcolor}
\usepackage{array}
\usepackage{hhline}
\usepackage{enumerate}
\usepackage{enumitem} % Customize enumerate and itemize
\usepackage{comment, verbatim} % For comment environment
\usepackage{imakeidx} % Analytical index
\usepackage{xparse} % In order to use ExplSyntax (in older latex versions)
\usepackage{mathtools}
\usepackage{fancyhdr} % Header and footer of pages
\usepackage{ifthen} % Ifthen construct
\usepackage{forloop} % forloop construct
\usepackage{xstring}
\usepackage{emptypage} % Leave completely empty the empty pages
\usepackage{minibox}

\usepackage{galois}

\usepackage[initials]{amsrefs} % For alphanumeric labels: alphabetic

\usepackage{tikz}
\usetikzlibrary{arrows,shapes,patterns,calc,fadings,decorations.pathreplacing,cd}

\usepackage{hyperref} % References become hyperlinks.
\hypersetup{
	colorlinks = true,
	linkcolor = {blue},
	urlcolor = {red},
	citecolor = {blue},
}

\newcounter{results}[section] % Uniform counters for lemmas, theorems, propositions etc

\newcounter{steps}[section] % Uniform counters for lemmas, theorems, propositions etc

\theoremstyle{plain}
\newtheorem{theorem}[results]{Theorem}

\newtheorem{lemma}[results]{Lemma}
\newtheorem{proposition}[results]{Proposition}
\newtheorem{corollary}[results]{Corollary}

\newtheorem*{theorem*}{Theorem}
\newtheorem*{lemma*}{Lemma}
\newtheorem*{proposition*}{Proposition}
\newtheorem*{corollary*}{Corollary}
\newtheorem*{exercise*}{Exercise}
\newtheorem*{fact*}{Fact}

\theoremstyle{remark}
\newtheorem{remark}[results]{Remark}

\theoremstyle{definition}
\newtheorem{definition}[results]{Definition}
\newtheorem{example}[results]{Example}

\newtheorem*{definition*}{Definition}
\newtheorem*{example*}{Example}

\newcommand{\N}{\ensuremath{\mathbb N}}%Natural numbers
\newcommand{\R}{\ensuremath{\mathbb R}}%Real numbers
\newcommand{\C}{\ensuremath{\mathbb C}}%Complex numbers

\newcommand \eps{\ensuremath{\varepsilon}}
\newcommand{\st}{\ensuremath{\ :\ }} % Such that in formulas.
\newcommand{\id}{\ensuremath{\mathds{1}}}% Identity

\DeclareMathOperator{\tr}{tr}

\DeclareMathOperator{\sym}{Sym}

\let\div\undefined
\newcommand{\div}{\ensuremath{\mathrm{div}}} % Divergence
\DeclareMathOperator{\II}{I\!I} % Second fundamental form
\DeclareMathOperator{\Riem}{Riem} % Ricci tensor
\DeclareMathOperator{\dist}{dist} % Distance
\let\oldL\L
\renewcommand{\L}{\ensuremath{\mathcal L}} % Jacobi operator
\DeclareMathOperator{\ind}{index} % Index

\newcommand{\be}{\begin{equation}}
\newcommand{\ee}{\end{equation}}

\newcommand{\AAa}{\mathbb{A}}
\newcommand{\BB}{\mathbb{B}}
\newcommand{\CC}{\mathbb{C}}

\newcommand{\NN}{\mathbb{N}}

\newcommand{\RR}{\mathbb{R}}
\newcommand{\SSp}{\mathbb{S}}

\newcommand{\cA}{\mathcal{A}}

\newcommand{\cC}{\mathcal{C}}

\newcommand{\cE}{\mathcal{E}}
\newcommand{\cF}{\mathcal{F}}
\newcommand{\cG}{\mathcal{G}}

\newcommand{\cI}{\mathcal{I}}

\newcommand{\cL}{\mathcal{L}}
\newcommand{\cM}{\mathcal{M}}

\newcommand{\cR}{\mathcal{R}}
\newcommand{\cS}{\mathcal{S}}
\newcommand{\cT}{\mathcal{T}}

\newcommand{\cZ}{\mathcal{Z}}

\newcommand{\mbfA}{\mathbf{A}}
\newcommand{\mbfB}{\mathbf{B}}
\newcommand{\mbfC}{\mathbf{C}}
\newcommand{\mbfD}{\mathbf{D}}

\newcommand{\mbfE}{\mathbf{E}}
\newcommand{\mbfF}{\mathbf{F}}  
\newcommand{\mbfG}{\mathbf{G}}

\newcommand{\mbfP}{\mathbf{P}}

\newcommand{\bfr}{\mathbf{r}}

\newcommand{\mbfT}{\mathbf{T}}

\newcommand{\mbfV}{\mathbf{V}}

\newcommand{\mbfx}{\mathbf{x}}
\newcommand{\bfx}{\mathbf{x}}

\newcommand{\rmI}{\mathrm{I}}

\newcommand{\rmP}{\mathrm{P}}

\newcommand{\rmW}{\mathrm{W}}

\newcommand{\mfr}{\mathfrak{r}}

\newcommand{\scG}{\mathscr{G}}
\newcommand{\scH}{\mathscr{H}}

\newcommand{\scL}{\mathscr{L}}
\newcommand{\scM}{\mathscr{M}}
\newcommand{\scN}{\mathscr{N}}

\newcommand{\orig}{\mathbf{0}}
\newcommand{\End}{\operatorname{End}}

\newcommand{\Ker}{\operatorname{Ker}}
\newcommand{\Coker}{\operatorname{Coker}}
\newcommand{\spt}{\operatorname{spt}}

\newcommand{\graph}{\operatorname{graph}}
\newcommand{\injrad}{\operatorname{injrad}}
\newcommand{\Reg}{\operatorname{Reg}}
\newcommand{\Sing}{\operatorname{Sing}}

\newcommand{\MSI}{\textbf{MSI}}
\newcommand{\conetimes}{\times\!\!\!\!\!\times}
\newcommand{\eucl}{\mathrm{euc}}

\newcommand{\set}[1]{\left\{#1\right\}}

     \title[Non-persistence of strongly isolated singularities] 
     {Non-persistence of strongly isolated singularities, and geometric applications}
     \author{Alessandro Carlotto, Yangyang Li, Zhihan Wang}
     \address{
     \newline \indent Alessandro Carlotto: 
     \newline Universit\`a di Trento, Dipartimento di Matematica,
via Sommarive 14, 38123 Trento, Italy
\newline\textit{E-mail address: alessandro.carlotto@unitn.it}
 \newline \newline \indent Yangyang Li: 
\newline Department of Mathematics, University of Chicago, 5734 S. University Avenue,
Chicago, IL, 60637, United States of America
       \newline \textit{E-mail address: 
       yangyangli@uchicago.edu} 
     	   	 \newline \newline \indent Zhihan Wang: 
\newline Department of Mathematics, Cornell University, 310 Malott Hall,
Ithaca, NY, 14853, United States of America
       \newline \textit{E-mail address: zw782@cornell.edu} 
       }

\begin{document}

\begin{abstract}
    We obtain a generic regularity result for stationary integral $n$-varifolds with only strongly isolated singularities inside $N$-dimensional Riemannian manifolds, in absence of any restriction on the dimension ($n\geq 2$) and codimension. As a special case, we prove that for any $n\geq 2$ and any compact $(n+1)$-dimensional manifold $M$ the following holds: for a generic choice of the background metric $g$ all stationary integral $n$-varifolds in $(M,g)$ will either be entirely smooth or have at least one singular point that is not strongly isolated.
    In other words, only ``more complicated'' singularities may possibly persist. This implies, for instance, a generic finiteness result for the class of all closed minimal hypersurfaces of area at most $4\pi^2-\varepsilon$ (for any $\varepsilon>0$) in nearly round four-spheres: we can thus give precise answers, in the negative, to the questions of persistence of the Clifford football and of Hsiang's hyperspheres in nearly round metrics. 
    The aforementioned main regularity result is achieved as a consequence of the fine analysis of the Fredholm index of the Jacobi operator for such varifolds: we prove on the one hand an exact formula relating that number to the Morse indices of the conical links at the singular points, while on the other hand we show that the same number is non-negative for all such varifolds if the ambient metric is generic.
\end{abstract}

\maketitle
    
\tableofcontents

\section{Introduction}\label{sec:Intro}

\subsection{Context and geometric motivations}\label{subs:ContextGeoMotiv} In 1983, Hsiang \cite{Hsi83} disproved Chern's spherical Bernstein conjecture by constructing a sequence of infinitely many, pairwise distinct, embedded minimal hyperspheres in the round four-dimensional sphere. This result is to be compared with what happens in one dimension less, as determined by the earlier rigidity theorems obtained by Almgren \cite{Alm66} and Calabi \cite{Cal67}. Such hyperspheres are constructed by means of an equivariant reduction and the analysis of the resulting singular ODE system in the quotient space. It turns out that their (unit multiplicity) varifold limit is the spherical suspension of a Clifford torus inside an equatorial three-sphere, thus a singular minimal subvariety (henceforth refereed to as Clifford football), which is smooth at all points except the north and south pole.

It is an interesting question, brought to the attention of the first-named author about a decade ago by Neves, whether such a picture persists when considering nearly round metrics on $S^4$. In particular, one may ask whether for \emph{any} metric $g$ sufficiently close to the round one --- say in the smooth topology --- the Riemannian manifold $(S^4,g)$ still contains infinitely many (embedded) minimal hyperspheres. This problem stems from the well-known link between rigidity phenomena characterizing special submanifolds in round spheres, and the resulting scarcity phenomena for slightly deformed metrics.
Indeed, as a reflex of the well-known characterization of simple closed geodesics on the round two-sphere as equatorial circles it was proven by Morse that there are nearly round metrics on $S^2$ that have \emph{only three} simple closed geodesics, and similarly (as a reflex of the aforementioned rigidity theorems by Almgren and Calabi) White \cite{Whi91} proved that there are nearly round metrics on $S^3$ that have \emph{only four} minimal hyperspheres. Aiming for an understanding of the situation in ambient dimension more naturally leads to the question above.

In principle, one may attack this problem by means of a perturbative approach of essentially PDE-theoretic nature, by first showing (ideally) that any nearly round metric allows for a --- suitably defined --- singular minimal subvariety modelled on the aforementioned Clifford football, and then desingularizing the football in question to obtain the desired minimal hyperspheres. Such an approach turns out to be at the very least challenging, because a direct application of the implicit function theorem is (unsurprisingly) obstructed by the large kernel of the Jacobi operator of the football; in suitable weighted Sobolev spaces such a kernel has actually dimension 18 (cf. \cite{Car19}, see later discussion in Section \ref{sec:Index}) and it is unclear how to possibly handle it by means of a Lyapunov-Schmidt reduction; in fact, it is a consequence of the present work (see, specifically, the statements of Corollary \ref{cor:Smooth} and Corollary \ref{cor:MainHsiang} below), that such a program is inevitably doomed to fail.

\subsection{Main results}\label{subs:MainRes} For indeed, in this article we approach Neves' question from a totally different perspective, and answer it in strong negative terms. Such a conclusion ultimately descends from our main theorem, which can be stated as follows:
\begin{theorem} \label{Thm:GenReg,IsolatedSing}
Given a closed manifold $M$ of dimension $N\geq 3$, there exists a generic subset $\scG_0$ of the space of smooth metrics on $M$ with the following property:
 for every $g\in \scG_0$, any $g$-stationary integral $n$-varifold in $(M,g)$, $2\leq n<N$, will:
\begin{enumerate}[label=(\roman*)]
\item{either be entirely smooth, or}
\item{have at least one singular point that is not strongly isolated, or}
\item{have only strongly isolated singular points all having links with Morse index equal to N.}
\end{enumerate}
 \end{theorem}

The precise notion of ``strongly isolated singularity'' is recalled in Definition \ref{def:MSI}. In particular, in the codimension one case (that is to say: when $N=n+1$) the third alternative may not possibly happen (see Remark \ref{rem:BasicIndexEst}) and so we conclude an unconditional generic regularity result.

\begin{corollary}\label{cor:MainSpecializedCodim1}
 Given a closed manifold $M$ of dimension $n+1\geq 3$, there exists a generic subset $\scG_0$ of the space of smooth metrics on $M$ with the following property:
for every $g\in \scG_0$, any $g$-stationary integral varifold 
will either be entirely smooth or have at least one singular point that is not strongly isolated.
    \end{corollary}

\begin{remark}\label{rem:BaireWhite}
In the statements above and throughout this article the notion of ``genericity'' is understood in the sense of Baire, i.\,e. a generic set is a countable intersection of open dense sets (hence a dense $G_{\delta}$).

A well-known result by White (see \cites{Whi91, Whi17}) ensures that the set of smooth ($C^\infty$) Riemannian metrics such that \emph{all} closed, (smooth) embedded minimal hypersurfaces are non-degenerate (namely: have no non-trivial Jacobi fields) is indeed generic. We can thus assume to have fixed, once and for all, a subset $\mathscr{G}$ of the class of smooth Riemannian metrics enjoying such a property, and when we write --- for instance --- that $g\in\mathscr{N}$ is a generic metric we may convene that $\mathscr{N}\subset\mathscr{G}$. In particular, in the previous statement we shall tacitly agree that $\scG_0\subset \scG$.
\end{remark}

Now, the connection with the geometric question we started with, and more specifically with the problem of perturbing the Clifford football to nearby stationary varifolds lies in the fact that, if we impose an area bound --- that is to say: a suitable upper bound on the mass of our stationary integral varifolds --- then all such varifolds, if mod 2 cyclic (cf. \cite{Whi09}), necessarily have a very special singular set (namely: all singularities are automatically strongly isolated) and thus have the structure of what we call an $\MSI$ (see Definition \ref{def:MSI} and Proposition \ref{pro:SmallMass=>Isol}). Informally speaking, one can say that, under such a bound, the Clifford football represents, in terms of regularity, the worst that can possibly happen. 
As a result, we have the following geometric consequence:

\begin{corollary}\label{cor:Smooth}
Given any $\varepsilon>0$, there is a neighborhood $\mathscr{N}(\varepsilon)$ of the round metric in the space of smooth metrics on $S^4$ such that for a generic metric $g \in \mathscr{N}(\varepsilon)$, the following is true:
Every mod 2 cyclic $g$-stationary integral varifold with mass less that $4\pi^2 - \varepsilon$ is entirely smooth.
\end{corollary}

This statement does not, in itself, give a direct response to Neves' question, although it indicates a definite obstruction to the possible ``perturbation-desingularization'' approach described above. But, in fact, Theorem \ref{Thm:GenReg,IsolatedSing} also implies a generic finiteness result under a pure area bound:

\begin{corollary}\label{cor:MainArea}
Given any $\varepsilon>0$ there exists a neighborhood $\mathscr{N}(\varepsilon)$ of the round metric on $S^4$ such that for a generic choice of $g\in \mathscr{N}(\varepsilon)$ the Riemannian manifold $(S^4, g)$ shall contain only finitely many closed, embedded minimal hypersurfaces of area less that $4\pi^2-\varepsilon$. 
\end{corollary}

We recall that $2\pi^2$ equals the area of an equatorial $S^3$ in round $S^4$; thus note that the conclusion of this corollary cannot possibly hold if we remove the genericity assumption (one counterexample being indeed the round metric).
Furthermore, we expect the threshold $4\pi^2$ to be sharp, based on the following heuristic argument. There clearly exist nearly round metrics on $S^4$ that contain (at least) five minimal hyperspheres, any pair of which shall intersect by virtue of the Frankel property: now, one expects that the desingularization of (any) such pair would give rise - by virtue of the implicit function theorem - to infinitely many closed minimal hypersurfaces with area arbitrarily close to $4\pi^2$.

A statement describing the implications at the level of the ``perturbation'' problem we posed above requires some notation and a brief digression. One can ``parametrize'' the Hsiang hyperspheres by an integer $k\in\mathbb{N}$, agreeing that $M_k$ is the Hsiang hypersphere intersecting the Clifford football transversely along exactly $2k+1$ tori; in particular for $k=0$ one recovers the equatorial three-dimensional hypersphere. Of course, one can act (both on the the Clifford football, and on any of Hsiang's hypersurfaces) via isometries of round $S^4$; considering such actions is already necessary if one attempts to perturb such hypersurfaces ``one (minimal) hypersphere at a time''.

For an open set $\mathscr{N}$ of Riemannian metrics on $S^4$ (containing the unit round metric $g_0$) let us assume to have a continuous map
\begin{equation}\label{eq:NorGraph}
\Psi_k:\mathscr{N} \to \ O(5)\times C^1(M_k;\R)
\end{equation}
such that $\Psi^{(2)}_k(g)$ defines a normal graph over the image through $T^{(k)}=:\Psi^{(1)}_k(g)$ of $M_k$ --- a hypersurface henceforth simply denoted $M_k(g)$ --- that is minimal in metric $g$; in particular, if this is the case we note that the ``area'' map $\mathscr{N}\ni g\mapsto \| M_k(g)\|$ is itself continuous. 

That being said, and recalling that the area of the Clifford football equals $\pi^3\simeq 31.00063$, it follows from the aforementioned varifold convergence that there exist at most finitely many (conjecturally none) Hsiang hyperspheres whose area exceeds the threshold value $4\pi^2=39.47842$. With all of this notation in place and keeping in mind all these remarks, the preceding statement implies this one:

\begin{corollary}\label{cor:MainHsiang}
    For any $\eps>0$ the following holds.

Given any neighborhood $\mathscr{N}\ni g_0$ of smooth Riemannian metrics on $S^4$ there are only finitely many integers $k_1<k_2<\ldots< k_{\ell}$ and, for each such integer, only a finite set $I(k_j)$ such that a continuous map $\Psi^{i_j}_{k_j}, (i_j\in I(k_j))$  as per \eqref{eq:NorGraph} and satisfying
\[
\sup_{g\in\mathscr{N}}\|M^{i_j}_{k_j}(g)\|<4\pi^2-\eps
\]
can possibly exist. 
\end{corollary}

This means that, no matter how small the neighborhood of nearly round metrics we consider (and no matter how cleverly we design the map $\Psi$), only finitely many Hsiang hyperspheres shall possibly persist.

\begin{remark}\label{rem:FiniteBound}
We expect that the finiteness result given in Corollary \ref{cor:MainArea} cannot possibly be upgraded to a uniform bound. For indeed, a standard application of the Lyapunov-Schmidt reduction (applied to ``one hypersphere at a time'') should allow to conclude, for any sufficiently large $k\in\mathbb{N}$ (say $k\geq \underline{k}$), the existence of a neighborhood $\mathscr{N}_k\ni g_0$ such that for any $g\in\mathscr{N}_k$ the Riemannian manifold $(S^4,g)$ shall contain at least one minimal hypersphere defined by a continuous deformation --- understood in the sense above --- of an isometric copy of $M_k$. In particular for any metric $g\in \bigcap_{\underline{k}\leq k\leq \overline{k}} \mathscr{N}_k$ the manifold $(S^4,g)$ shall contain a finite but arbitrarily large number of closed embedded minimal hypersurfaces, provided one takes $\overline{k}$ big enough. 
\end{remark}

Let us add some comments on our ``generic regularity'' theorem, that is Theorem \ref{Thm:GenReg,IsolatedSing}, without restricting to the codimension one case. When considering the class of $n$-dimensional stationary varifolds with only strongly isolated singularities (cf. Definition \ref{def:MSI}) in an $N$-dimensional Riemannian manifold, we can rule out, \emph{generically}, all singularities whose link has Morse index strictly larger than $N$ (note that the weak inequality is, in fact, always satisfied, see Remark \ref{rem:BasicIndexEst}). It is appropriate for us to remark that, when the link in question is disconnected, such a condition is, in fact, certainly fulfilled with the sole (well justified) exception of links consisting of the disjoint union of exactly two half-dimensional equatorial spheres (that is to say: $2n=N$). The fact that there may (in fact: there should) be persistent isolated singularities in the case of connected link is less obvious, although e.\,g. we can note that the Veronese embedding of $\C\mathbb{P}^2$ inside the 7-dimensional round sphere does not bound any smooth five-dimensional manifold (so that the cone over such $\C\mathbb{P}^2$ cannot possibly be smoothed out by minimal submanifolds). This clearly provides partial albeit compelling evidence in that
direction; we refer the reader to the striking recent work by Liu \cite{Liu22} for a thorough discussion of these aspects and the construction of a  number of different examples of homologically area-minimizing submanifolds with non-smoothable singularities.

\begin{remark}\label{rem:Sharpness}
More generally, one can study (after Cartan) the following three examples of focal submanifolds of isoparametric minimal hypersurface in spheres. For $\mathbb{F}=\R, \C, \mathbb{H}$, set $m=\dim_{\R}(\mathbb{F})$ and consider the embedding map 
  $\mathbb{F}\mathbb{P}^2 \to \SSp^{3m+1} \subset \RR^{3m+2} = \mathbb{F}^3\times \RR^2$ defined by
  \[
[u: v: w]\mapsto (|u|^2+|v|^2+|w|^2)^{-1}\cdot \left(\sqrt3 v\bar w, \sqrt3 w\bar u, \sqrt3 u\bar v, \frac{3}{2}(|u|^2 - |v|^2),\frac{1}{2}(2|w|^2 - |u|^2 - |v|^2)\right)
  \]
  This gives the classical Veronese minimal $\mathbb{R}\mathbb{P}^2$ in $\SSp^4$, a minimal $\C\mathbb{P}^2$ in $\SSp^7$ and a minimal $\mathbb{H}\mathbb{P}^2$ in $\SSp^{13}$. It is by now standard to prove that none of the three projective planes in question bounds a smooth manifold (of real dimension, respectively, $3,5,9$) since they have odd Euler characteristic. We are led to believe that all of the corresponding minimal submanifolds have Morse index equal to $3m+2$ (i.\,e. $N$ in the notation of Theorem \ref{Thm:GenReg,IsolatedSing}) --- which at the moment is only known for $\mathbb{F}=\mathbb{R}$, see Remark \ref{rem:BasicIndexEst} --- and provide models of persistent singularities.
  \end{remark}

\subsection{Previous work}\label{subs:Previous}
In recent years we have witnessed impressive advances on the theme of ``generic regularity'' of minimal hypersurfaces \cites{LW20, CLS22, CMS23a, CMS23b, Edelen2021, Li23, LW22}, not to mention the related (equally striking) advances on the study of generic mean curvature flows, which can simply not properly be accounted for here. (For various significant results on generic properties of geodesics and geodesic nets the reader is instead referred to \cites{ChoMan23, MKSS24} as well as references therein.) Besides being constrained to the codimension one case, many such results typically concern area-minimizing hypersurfaces, either with respect to a fixed boundary (i.\,e. in the framework of Plateau's problem) or with respect to a fixed homology class (cf. \cite[Problem 108]{Yau82} and \cite[Problem 5.16]{Bro86}); this is the case, in particular, for \cite{CMS23a, CMS23b}, which substantially refined the pioneering theorem by Smale \cite{Sma93}. Instead \cite{CLS22, Wang20, LW20, Edelen2021, Li23, LW22} rather deal (in various forms) with the somewhat more general case of minimal hypersurfaces arising from min-max techniques (in the spirit of Almgren-Pitts) thus not necessarily (and not typically) stable ones; however such works are constrained to ambient dimension 8 and only allow for singularities that are modelled on stable minimal hypercones (as reflected by the assumption of local stability of the minimal hypersurfaces in question). (We wish to stress that, on the contrary, \cite{CMS23a, CMS23b} allow for more general singular sets, which in particular do not need to consist of isolated points.) 

It is to be remarked that many of such contributions build, in various ways, upon the pioneering work by Hardt and Simon \cite{HarSim85}: going beyond the results in Bombieri-De Giorgi-Giusti \cite{BDGG69} (constructing minimal smoothings of the Simons cone in $\R^{2n}$, $2n\geq 8$, by an ODE approach) they proved that, in fact, every regular minimizing hypercone can be uniquely (up to scaling) perturbed to one side to produce a smooth minimizing hypersurface asymptotic to it near infinity. Making using of this fact, they succeeded in showing that minimizing hypersurfaces with a generic boundary in $\R^8$ are actually smooth. Such results about the minimal smoothing of a minimizing hypercone were generalized by the third-named author without assuming the hypercone in question being regular (see \cite{Wan24}), while the strong uniqueness (up to a scaling) of such one-sided smoothing was then established by Edelen-Szekelyhidi \cite{EdeSze24} for ``cylindrical'' hypercones under a pure density assumption.

All that said, there are two aspects of substantial novelty in the present work: on the one hand we deal here with isolated singularities that are modelled on regular yet not necessarily stable cones (which is in fact \emph{necessary} to ultimately resolve the geometric questions we posed at the very beginning of this introduction, as already displayed by the aforementioned Clifford football) while on the other hand our analysis applies to minimal $n$-dimensional submanifolds in an $N$-dimensional ambient manifold, without restricting to the codimension one case. The case of surfaces ($n=2$ thus $N\geq 4$) is somewhat special and can be tackled with tools that are quite different than the ones employed here; the reader should in particular compare our results with those of White \cite{Whi85,Whi19} (obtaining, respectively, a full ``generic regularity result'' for surfaces minimizing area with respect to a fixed boundary, or in a fixed homology class), and the ones by Moore \cite{Moo06, Moo21} (much related to earlier work by B\"ohme and Tromba, \cite{BohTro81}) which are cast in the language of (prime) parametric minimal surfaces.

\subsection{Approach}\label{subs:Approach}

Our main results, while crucially building on previous work of the second and third author (see, in particular, \cite{LW22}), and to some extent on
certain deep contributions by Edelen \cite{Edelen2021}, are obtained by a blend of different techniques and ideas. In a nutshell, our approach consists in studying the Fredholm index of the Jacobi (stability) operator of minimal submanifolds with only strongly isolated singularities (see Definition \ref{def:MSI} of $\MSI$), acting on suitably defined weighted Sobolev spaces. Employing powerful tools in linear Analysis, we first obtain an exact formula giving the value of such an index in terms of the normalized Morse index of links of the cones at the singular points; we refer the reader to Theorem \ref{thm:Count} for a precise statement. Roughly speaking, the presence of singularities perturbs the natural Fredholm index zero Schr\"odinger operator, in a way that is solely encoded by the structure of the singularities. Such a counting formula, in spite of its striking simplicity, appears to be new to the best of our knowledge. Since, by virtue of general lower bounds for the Morse index of minimal submanifolds in round spheres, the local contributions are all non-positive (i.\,e. $\leq 0$) they force the Fredholm index of any $\MSI$ to be itself less or equal than zero, with strict inequality in many case of interest (for instance: unconditionally in the codimension one case). We then complement such a result by showing that, on the other hand, generically (in the sense of Baire) such an index must be non-negative (i.\,e. $\geq 0$); the approach that we present (see Section \ref{sec:GenReg}) essentially follows --- at least at a \emph{conceptual} level --- ideas going back to the fundamental work of White \cite{Whi91}, where one obtained suitable local and global Sard's lemmata for the natural projector of metric-$\MSI$ \ pairs onto the first factor; while in that context one could exploit a Banach manifold structure (and thus ultimately invoke Smale's Sard-type theorem) here the use of (pseudo-)canonical neighborhoods, which are significantly wilder than metric balls, does not allow for the employment of such pre-existing methodologies, and we rather need to work harder to unwind such tools and recast/adapt them to our own setting. 
It is in that respect that we crucially exploit the tools of \cite{LW22} and \cite{Wang20}; we warn the reader, however, that due to our necessity of dealing with unstable (infinite Morse index) cones, and without any codimensional restriction, several changes (and corresponding adaptations) had to be performed, thereby not allowing for a direct quote of technical lemmata (which would have shortened and simplified this manuscript).

\subsection{Organization of the paper}\label{subs:Organization}
Besides the present introduction, this article consists of four sections plus five appendices, that contain technical yet \emph{essential} material that we decided to separate from the main body of the paper with the sole scope of improving readability, thereby leading to a more direct path towards the main theorems. In Section \ref{sec:Prelim} we present the general setup of this work and collect some of the basic definitions we employ in the sequel. Then, Section \ref{sec:Index} combines some preliminary facts about analysis on regular minimal cones with the proof of our ``counting formula'' for the Fredholm index of the Jacobi operator of an $\MSI$. Hence, we move to the detailed study of the generic behaviour of such an operator, which we carry out in Section \ref{sec:GenReg}, crucially building upon various results proven in the appendices. We then capitalize such efforts in Section \ref{sec:ProofMain}, devoted to the proofs of the main theorems, also getting back to the motivating geometric applications we presented above.

Concerning the ``technical material'', we start in Appendix \ref{app:MSE} by collecting certain key features of the minimal surface equation (system) in arbitrary codimension and then 
proving several useful facts about the ``transfer of normal sections'' between nearby subvarieties. It is already at this level that one can get an appreciation of the difference in complexity between the codimension one case and the case of higher codimension: while in the former setting the transfer of normal section is --- at least at a local level --- essentially trivial (since one can just pre-compose with the parametrization map), in the latter one needs to project back onto the normal bundle of the base submanifold, which poses the problem of discussing how natural geometric PDE transform under this operation, and how one may handle the error terms in the resulting equations. 
We then move on to Appendix \ref{sec:3circles}, devoted to certain three-circle inequalities, which allow to compare the behaviour of an $\MSI$ at different scales, thereby deriving $L^2$-decay (more generally: $L^2$-non concentration) estimates that are key to analyzing the limit behaviour of sequences of tame Jacobi fields (cf. Definition \ref{def:SlowerGrowth}). In Appendix \ref{sec:QuantUniq} we discuss the problem of quantitative (uniform) uniqueness of tangent cones at strongly isolated singularities, which in particular allows to obtain Corollary \ref{Cor_Converg in all Scales} concerning the ``uniform convergence at all scales'' of a stationary integral varifold on approach to a strongly isolated singularity (crucially building on both \cite{Sim83} and, more closely, on the recent advances in \cite{Edelen2021}). In turn, such a statement serves as input for Appendix \ref{app:GraphicalPar}, devoted to a fine analysis of the mutual parametrization of two (possibly three) nearby $\MSI$, that comes into play in the core of the paper when we prove the fundamental limit equation (see Lemma \ref{Lem_Tangent vectors to scL^k}). We expect that some of these tools will also be helpful for other related studies in the near future.
Lastly, in Appendix \ref{Sec:Count_Decomp} we discuss (following Edelen's study of the codimension one case) how to effectively describe the space of all $\MSI$ inside a Riemannian manifold, and thus how to obtain a controlled covering theorem (namely: Theorem \ref{Thm_Countable Decomp}) for the space of metric-$\MSI$ pairs by means of countably many (carefully designed) canonical pseudo-neighborhoods. Note that, in comparison to \cite{Edelen2021}, for submanifolds in codimension other than one one loses the natural ``order structure'' coming from the trivialization of the normal bundle, thereby requiring the introduction of new techniques, e.\,g. to handle the combinatorics of ``cascades'' of isolated conical singularities.
The covering theorem in question is then one of the three crucial ingredients in the proof of the global Sard theorem, Theorem \ref{thm:generic_perturbation}, that we alluded to before.

\section{Notation, setup and preliminary results}\label{sec:Prelim}

Throughout this article, we will suppose $M$ be a smooth (i.\,e. $C^\infty$) manifold of dimension $N\geq 3$; it is understood that the regularity of maps between or tensors over smooth manifolds are measured with respect to smooth local coordinates. Although our main theorems, as anticipated in the introduction, are stated for smooth Riemannian metrics, we will need to also work with metrics having a finite degree of regularity.

So, we shall assume --- unless otherwise explicitly stated (like we will do at the very end of Section \ref{sec:GenReg}, and in Section \ref{sec:ProofMain}) --- that $M$ is endowed with metrics of class $C^{k,\alpha}$ where $k\geq 4$ is an integer that (due to technical reasons, cf. Appendix \ref{app:MSE} and Appendix \ref{app:GraphicalPar}) will be chosen larger than a convenient threshold $k_0$, and $0<\alpha<1$.

\subsection{MSI and general setup}\label{subs:GenSetp} Let us start with some basic definitions; it is appropriate, to avoid ambiguity, to state what we mean by regular set of an integral varifold.

\begin{definition}\label{def:RegSet}
Let $(M,g)$ be a Riemannian manifold of dimension $N\geq 3$. Given an integral varifold $V$, of dimension $2\leq n<N $, we say that $p\in\spt(\|V\|)\subset M$ belongs to its regular set if there exists $r>0$ such that $\spt(\|V\|)\cap \overline{B^g(x,r)}$ is a smooth, embedded, compact, connected $n$-dimensional submanifold of class $C^2$ with boundary contained in $\partial B^g(x,r)$. The set of regular points of $V$ shall be denoted by $\Reg(V)$; its complement in $\spt(\|V\|)$ will be referred to as singular set of $V$, i.\,e. we let $\Sing{V}=\spt(\|V\|)\setminus \Reg(V)$.
\end{definition}

Note that, in the setting above, when we say that $\Sigma$ is an embedded $C^{\ell\geq 2}$ submanifold of dimension $2\leq n<N$ we mean that for every $p\in \Sigma$, there exist $C^\infty$ coordinate charts $M\supset U_p\to \RR^N$ on a neighborhood of $p$, mapping $\Sigma$ to a graph of some $\RR^{N-n}$-valued function, of class $C^{\ell}$, under the Euclidean metric $g_\eucl$.

\begin{remark}\label{rem:OptRegMinSubm}
Of course, $\Reg(V)$ (respectively: $\Sing{V}$) is open (respectively: closed) in $\spt(\|V\|)$. Furthermore, by standard elliptic regularity (Schauder theory applied to the minimal surface system) if --- in the setting of the preceding definition --- $p\in\Reg(V)$ and the ambient metric $g$ is $C^{k,\alpha}$ then, in fact, $\spt(\|V\|)$ is an  $n$-dimensional submanifold of class $C^{k+1,\alpha}$ in an open neighborhood of the point in question (see e.\,g. \cite[Theorem 1.1]{Whi91}); in particular if the ambient metric is smooth (by which we shall mean $C^{\infty}$) then so will be $\spt(\|V\|)$, at least away from the singular set of $V$.
\end{remark}

Let us proceed and introduce the objects that we wish to study throughout this article:

\begin{definition}\label{def:MSI}
Let $(M,g)$ be a Riemannian manifold of dimension $N\geq 3$. For every stationary integral varifold $V$, of dimension $2\leq n<N $, we call a point $p\in \Sing(V)$ \textbf{strongly isolated}, if \emph{some} tangent cone of $V$ at $p$ is \textbf{regular}, i.\,e. it is of multiplicity one and has smooth link (equivalently: if the singular set of such a cone coincides with the origin). We will then say that a stationary integral varifold $V$ \textbf{has only strongly isolated singularities} (and refer to it as $\MSI$) if either $\Sing(V) = \emptyset$, or every $p\in \Sing(V)$ is strongly isolated. 
\end{definition}

\begin{remark}\label{rem:MSI}
   We remark that, in the setting of the preceding definition, at any singular point $p\in\Sing(V)$ there holds uniqueness of the tangent cone in question, thanks to a deep result of Simon \cite{Sim83} (see also \cite{Sim85}). Furthermore, the corollary stated after Theorem 5 therein implies that every such singular point has a neighborhood where the only singular point is the one in question: hence, since $\Sing(V)$ is closed and thus compact whenever the ambient manifold is, the finite covering property tells us that there are only finitely many such singular points (thereby justifying the terminology we employ).  
 \end{remark}

 \begin{remark}\label{rem:MSInotation}
   Relying on the content of the previous remark, we will switch to a somewhat more natural (or, possibly, more convenient) notation: given one such varifold, we will rather employ the letter $\Sigma=\Sigma^n$ to denote it, meaning that $V=|\Sigma|$ (that is to say: $V$ is the \emph{multiplicity 1} integral varifold associated to the submanifold $\Sigma$)
   and in our work we will practically identify $\Sigma$ with the regular part of $\overline{\Sigma}$. Thus $\Sigma$ is treated as a regular (cf. Remark \ref{rem:OptRegMinSubm}) but not necessarily closed submanifold. 
   The (unique) tangent cone of $\Sigma$ at a singular point $p$ is denoted by $\mbfC_p(\Sigma)$, or simply $\mbfC_p$ if there is no risk of confusion. 
\end{remark}

Based on White's natural homomorphism one can associate to an integral varifold a mod 2 flat chain (see \cite{Whi09}); if such object has zero boundary (in the sense of the standard boundary operator in the latter setting) we will simply refer to it as mod 2 cyclic varifold.  

\begin{remark}\label{MSIvsMOD2cyclic}
Let us note the following fact: an $\MSI$, say $\Sigma$, can always be regarded as a mod 2 cyclic stationary varifold. This can be justified as follows.
By Sard’s theorem and the slicing theorem, we can choose a sequence $r_i \to 0$ such that $\Sigma_i := \Sigma \setminus\cup_{p\in\Sing(\Sigma)} B_{r_i}(p)$ is a mod 2 flat chain with boundary $\partial \Sigma_i$, and, \emph{as soon as $n\geq 2$}, the mass of $\partial \Sigma_i$ converges to $0$ as $i\to\infty$. By the Federer-Fleming compactness theorem (\cite{FF60}, see also \cite{Fed69}), $\Sigma = \lim_i \Sigma_i$ is a mod 2 flat chain and $\partial \Sigma = \lim_i \partial \Sigma_i = 0$. Hence, $\Sigma$ is a cycle.
    \end{remark}

  To fix the notation, we then let $\Sigma$ have only strongly isolated singularities $p_1,\ldots, p_Q$ with (respectively) regular cones $\mbfC_1,\ldots, \mbfC_Q$ and associated links $S_1,\ldots, S_{Q}$ (that is to say: $S_i:=\mbfC_i\cap\mathbb{S}^{N-1}$, where $\mathbb{S}^d$ shall henceforth denote the round unit sphere in Euclidean $\R^{d+1}$). When we wish to stress the role of the basepoint (rather than the label) we shall write $\mbfC_p$ and $S_p$ instead. In this setting and under such assumptions, one can find a compact set $\Sigma_0\subset \Sigma$ such that $\Sigma\setminus \Sigma_0=\bigsqcup_{i=1}^{Q}E_i$ and for each value of the index $i$ we have that $E_i$ is diffeomorphic to the product of the corresponding link with the interval $(1,\infty)$.
  
\subsection{First and second variation}\label{subs:FirstSecVar}

For a regular submanifold $\Sigma=\Sigma^n$ (not necessarily closed, cf. Remark \ref{rem:MSInotation}) in a Riemannian manifold $(M,g)$ we shall denote by $T\Sigma$ its tangent bundle and we let $\mbfV:=\mbfV(\Sigma,M)$ denote its normal bundle instead. Here we recall what follows. If $g$ is a $C^{k,\alpha}$ metric on $M$, and $\Sigma$ is an embedded $C^{\ell\geq 2}$ submanifold, then: when $k\geq \ell-1$, the normal bundle $\mbfV$ of $\Sigma$ is of class $C^{\ell-1}$, meaning that for every $p\in \Sigma$, there is a neighborhood $\Omega_p\subset \Sigma$ of $p$ and a $C^{\ell-1}$ bundle map $\Upsilon: \BB^n\times\RR^{N-n}\to TM$ which is a bundle isomorphism onto $\mbfV|_{\Omega_p}$;
   when $k\geq \ell$, the normal exponential map $\mbfE: \mbfV\to M$, $(x, z)\mapsto \exp_x^g(z)$, is of class $C^{\ell-1}$, and thus the pull-back metric $\mbfE^*g$ is of class $C^{\ell-2}$ (under the bundle-isomorphic parametrization above). For all claims pertaining to the regularity of the exponential map we refer the reader (here and below), for instance, to the recent article by Lange \cite{Lan24} and reference therein.

For sections $X,Y$ of $T\Sigma$ the second fundamental form of $\Sigma$ inside $M$ is defined by
\[
\II_{\Sigma,g}(X,Y)=(\nabla_X Y)^{\perp}
\]
where $\nabla$ is the Levi-Civita connection in metric $g$ (see Section \ref{subs:Conv}), and ${}^{\perp}={}^{\perp,g}$ stands for the orthogonal projection onto $\mbfV$; the trace of the second fundamental form is then the (vector-valued) mean curvature:
\[
H_{\Sigma,g}=\tr_{\Sigma}(\II_{\Sigma,g}).
\]
In the setting of Definition \ref{def:MSI}, it is a standard consequence of the first variation formula that an $\MSI$ has mean curvature identically equal to zero on its regular part. 

We will routinely work with sections of the bundle $\mbfV$ and corresponding differential operators; in particular we let $\nabla^\perp:=\nabla_{\Sigma,g}^{\perp}$ denote the metric connection on $\mbfV$ that is determined by $\nabla$, and shall further denote by $\Delta_{\Sigma,g}^{\perp}$ the Laplace operator on the normal bundle of $\Sigma$, namely $\mbfV$. If we then consider the second variation of the $n$-dimensional area functional at the $g$-critical point $\Sigma$ we will find, for any section (here assumed, at least initially,  $C^2$ and compactly supported) $u$ of $\mbfV$
\be\label{eq:2ndVar}
\delta^2\Sigma [u]=-\int_{\Sigma}g(u,{L_{\Sigma,g}}u)\,d\|\Sigma\|
\ee
where $L_{\Sigma,g}$ is the ``Jacobi operator'' of $\Sigma$, which takes the form (see e.\,g.  \cite{Sim68})
\be\label{eq:Jacobi}
L_{\Sigma,g}u=\Delta_{\Sigma,g}^{\perp}u+g(\II_{\Sigma,g}, u)\II_{\Sigma,g}+\tr_{\Sigma} R_g(u,\cdot,\cdot)
\ee
with $R_g$ denoting the $(1,3)$-curvature tensor of the ambient manifold $(M,g)$; if $\left\{X_1,\ldots, X_n\right\}$ is any local orthonormal frame of the tangent sub-bundle $T\Sigma$ then we mean
\[
g(\II_{\Sigma,g}, u)\II_{\Sigma,g}=\sum_{i,j=1}^n g(\II_{\Sigma,g}(X_i,X_j), u)\II_{\Sigma,g}(X_i,X_j), \ \text{and} \ \tr_{\Sigma} R_g(u,\cdot,\cdot)=\sum_{i=1}^n R_g(u,X_i,X_i).
\]

\begin{remark}\label{rem:GradientsCompare}
If $u$ is a section of $\mbfV$ of class $C^1$ and $X, Y$ are any (otherwise unspecified) sections of $T\Sigma$, in the setting above, then the identity $g(Y,u)=0$ implies, by covariant differentiation in the direction of $X$, that
\[
g(Y,\nabla_X u)=- g(\II_{\Sigma,g}(X,Y), u).
\]
By the arbitrariness of $X, Y$ this identity allows to identify the ``tangential" component of $\nabla u$ as $-g(\II_{\Sigma,g}, u)$ and thereby, to relate bounds for $\nabla u$ to bounds for $\nabla^\perp u$ and the second fundamental form of the submanifold in question.
\end{remark}

There is a special case that warrants further discussion (and partly special notation): for a regular minimal cone $\mbfC$, of dimension $n$, in $\R^N$, we define the link $\Sigma:=\mbfC\cap\mathbb{S}^{N-1}$; here we are not assuming $\Sigma$ to be orientable, however recall that in the codimension one case $(N=n+1)$ a simple topological argument ensures that any such $\Sigma$ is two-sided hence orientable. At the link $\Sigma$ we thus associate a Jacobi operator $L_\Sigma$ that has the form
     \[
     L_\Sigma u:=\Delta^{\perp}_\Sigma u +g(\II_\Sigma,u)\II_\Sigma+(n-1)u
     \]
where the background unit round metric is understood throughout, and whose spectrum is a discrete sequence
     \[
\lambda_1\leq\lambda_2\leq\lambda_3\leq  \ldots \leq \lambda_k\ldots \to+\infty,   \]
under the sign convention that $L_\Sigma u_j=-\lambda_j u_j$. 
We shall now introduce the following --- a posteriori very convenient --- notion of ``effective Morse index''.

\begin{definition}\label{def:EffMorsIndex}
Let $\Sigma$ denote a closed, embedded minimal submanifold in the round sphere of dimension $N-1$, i.\,e. $\SSp^{N-1}\subset \R^N$, and let $\mbfC$ denote the cone over $\Sigma$ having vertex at the origin $\orig\in\R^N$ (that is: $\mbfC:=\mathbf{0} \conetimes \Sigma$). Then we define the \textbf{effective Morse index} of $\Sigma$ as
\be\label{Equ_GenReg:EffectiveInd_Cone}
    \rmI(\mbfC) := (\,\text{index of the Jacobi operator of } \mbfC\cap \mathbb{S}^{N-1}\,) - N. \ee
(Here it is to be stressed that we are not assuming $\Sigma$ to be connected.)    
\end{definition}

The following recollections will be repeatedly referred to in the sequel of this article.

\begin{remark}\label{rem:BasicIndexEst}
\begin{enumerate}[label={(\Alph*)}]
\item\label{item:IndexCodim1} In the codimension one case it is well-known that the Morse index of any non-equatorial $\Sigma$ is at least $N+1$: there always holds  
$\rmI(\mbfC)\geq 1$. 

\item\label{item:IndexHigherCodim} In the case of higher (in fact: arbitrary) codimension Simons proved in \cite{Sim68} that  
the Morse index of the Jacobi operator of  $\mbfC\cap \mathbb{S}^{N-1}$ is at least $N$ for any connected $n$-dimensional minimal submanifold that is not a great sphere (and any equatorial $n$-sphere has of course index equal to $N-n$). Hence, there \emph{always} holds $\rmI(\mbfC)\geq 0$ since if the link is disconnected then necessarily $n\leq N/2$ and so the total Morse index of the link is at least $2(N-n)\geq 2(N/2)=N$. It is an open question to determine all $(n-1)$-dimensional minimal submanifolds $\Sigma$ in $\SSp^{N-1}$ whose Morse index equals $N$ (thereby saturating the preceding bound): as noted by Kusner-Wang \cite[Theorem 4.8]{KusWan24} recent work by Karpukhin \cite{Kar21} implies that the aforementioned Veronese embedding $\R\mathbb{P}^2\to\SSp^4$ does provide such an example, whether by contrast all minimal 2-tori in $\SSp^4$ have index at least 6. 
\end{enumerate}
\end{remark}

\subsection{Functional spaces and augmentation by translation-like sections}\label{subs:Functional}

Let again $M$ denote a smooth, compact boundaryless manifold of dimension $N\geq 3$, endowed with a metric of class $C^{k,\alpha}$, and let $\Sigma$ be an
 open submanifold of dimension $2\leq n<N$ and class $C^{\ell}$ for $2\leq \ell\leq k+1$; as stipulated above 
$\mbfV=\mbfV(\Sigma,M)$ shall denote its normal bundle, which --- by a standard extension procedure (employing the ambient charts for $\Sigma$) --- one may regard as a subbundle of the smooth tangent bundle $TM$ of the ambient manifold. 

For any $0\leq m\leq \ell-1$, $0\leq \beta\leq \alpha$ and finite $p\geq 1$ we will routinely work with the functional spaces $C_{loc}^{m,\beta}(\Sigma;\mbfV)$,  $L^p_{loc}(\Sigma;\mbfV)$ and $W^{m,p}_{loc}(\Sigma;\mbfV)$ of $\mbfV$-valued maps. Furthermore, we will deal with vector spaces of compactly supported sections, and most often with $C^{m}_c(\Sigma;\mbfV)$.  In the very same setting described above, but specified to the case when $\Sigma$ is an $\MSI$ in $(M,g)$, we will further define weighted Sobolev spaces, as follows. 

Given a multi-index $\boldsymbol{\beta}=(\beta_1,\ldots,\beta_{Q})\in\mathbb{R}^{Q}$, let us agree to denote by $\rho^{\boldsymbol{\beta}}$ a positive, continuous function that equals $\rho_{\Sigma,g}^{\beta_i}$ along the end $E_{i}\subset \Sigma$ at least for $\dist_g(p_i,\cdot)\leq  3 r_0$ where $r_0=r_0(\Sigma,g)$; in fact, without loss of generality, it is convenient for the purposes of the present paper to assume that $\rho\leq 1$ at all points of $\Sigma$. Any such function, possibly subject to additional requirements, will henceforth be referred to as \textbf{radius function} of $\Sigma$ in $(M,g)$, and $r_0(\Sigma,g)$ shall be called its \textbf{scale}.

    We then let $W^{m,p}_{\boldsymbol{\beta}}(\Sigma;\mbfV)$ to be the Banach space completion of $C^m_c(\Sigma;\mbfV)$ with respect to the norm
				\[
				\left\|u\right\|_{W^{m,p}_{\boldsymbol{\beta}}(\Sigma)}:=\left(\sum_{j=0}^{m}\int_{\Sigma}|\rho^{(-\boldsymbol{\beta}+j)}(\nabla^{\perp})^{(j)}u|^{p}\rho^{-n}\,d\|\Sigma\|\right)^{1/p}.
				\] 

   It is also possible to define weighted Sobolev spaces $W^{-m,p}_{\boldsymbol{\beta}}(\Sigma;\mbfV)$ using the language of distributions; for manifolds with asymptotically cylindrical ends this is done e.\,g. in reference \cite{Mel93} and it is standard to adapt the treatment to manifolds with asymptotically conical or conically singular (CS) ends (see e.\,g. \cite{Pac13}), which in particular allows to cover the case under consideration. A posteriori, there holds a Banach space isomorphism
$(W^{m,p}_{\boldsymbol{\beta}}(\Sigma;\mbfV))^*\simeq W^{-m,p'}_{-\boldsymbol{\beta}-n}(\Sigma;\mbfV)$.  We wish to single out a case that warrants special notation: when for some $\tau\in\R$ we take $\beta_i=\tau$ (for all $i=1,\ldots, Q$) we shall simply write $W^{m,p}_{\tau}(\Sigma;\mbfV)$ in place of $W^{m,p}_{\boldsymbol{\beta}}(\Sigma;\mbfV)$.
    
	Some fundamental facts about Analysis on manifolds with conical singularities or, more generally, on \emph{conifolds} have been collected, for instance, in \cite{Pac13} (see also references therein); this includes in particular the Sobolev embedding theorem for the weighted spaces defined above. 
    
    Lastly, we will need to possibly ``augment'' such spaces in the terms that follow. 
  
  \begin{definition}\label{def:TranslLikeSection}
  In the setting described above, specified to the case when $\Sigma$ is an $\MSI$ in $(M,g)$ and $m\leq k-2$, we define a section $\phi\in C^{m}_{loc}(\Sigma;\mbfV)$ \textbf{translation-like} if for every $q\in \Sing(\Sigma)$, there exist a neighborhood $U\ni q$ in $M$ and a vector $v\in T_qM$ (viewed as a vector field in $U$ using the normal coordinates) such that under normal coordinates $x^i$ centered at $q$, 
  \[
  \phi|_U(x) = \Pi^\perp_{x}(v),
  \]
  where \[
    \Pi^\perp_x: T_xM \to T_xM
  \]  
  is the $g$-orthogonal projection map onto $T_x^\perp \Sigma (=\mbfV_x)$. 
  \end{definition}

\begin{remark}
  In the setting described at the very beginning of this subsection (Section \ref{subs:Functional}), we recall that for every $x\in M$, the $g$-exponential map $\exp_x^g:T_xM\to M$ is of class $C^{k-1,\alpha}$; hence the $N\times N$-matrix-valued function $(\exp_x^g)^*g\in C^{k-2,\alpha}$ (see \cite{Lan24}). We stress the following consequence: if $\Sigma$ is $g$-minimal, then under $C^\infty$ coordinates of $M$, $\Sigma$ is locally a graph of $C^{k+1, \alpha}$ $\RR^{N-n}$-valued function (as recalled above) but for every $p\in \Sigma$, $(\exp_p^g)^{-1}(\Sigma)$ is only a $C^{k-1,\alpha}$ $\RR^{N-n}$-valued function near $\orig$.
  \end{remark}

  \begin{definition}\label{eq:AugmentedSpaces}
  If $W^{m,p}_{loc}(\Sigma;\mbfV)$ denotes the space of
	 sections that are locally $W^{m,p}$ then one sets  
 \[\hat{W}^{m,p}_\tau(\Sigma;\mbfV) := \{u\in W^{m,p}_{loc}(\Sigma;\mbfV): \|u-\phi\|_{W^{m,p}_{\tau}}<+\infty \text{ for some translation-like function }\phi\}.
  \]
  \end{definition}

  Let us discuss, very concretely, the effects of such an ``augmentation'': in essence, for $\tau>0$ (thus imposing decay at the singular tips of the limit cones) we simply have that 
  \be\label{eq:Augment}
 \hat{W}^{m,p}_\tau(\Sigma;\mbfV)={W}^{m,p}_\tau(\Sigma;\mbfV)\oplus W_{\text{Tr}}(\Sigma;\mbfV)
  \ee
  where the vector space $W_{\text{Tr}}$ has dimension exactly equal to $NQ$ (notation as in Section \ref{subs:GenSetp}), with a basis being given by the collection of sections
\be\label{eq:BasisTranslaLikeFunctions}
\left\{(1-\zeta_{\Sigma,g,r_0}) \ \cdot \  \Pi^{\perp}_x(v) \st i=1,\ldots, Q, \ v\in T_{p_i}M\right\}
\ee
where $\zeta_{\Sigma,g,r_0}$ is the smooth cutoff function defined in the dedicated paragraph of Section \ref{subs:Conv} (so that $1-\zeta_{\Sigma,g,r_0}$ transitions from the value $1$ for $\rho_{\Sigma,g}\leq r_0$ and $0$ for $\rho_{\Sigma,g}\geq  2 r_0$; here $r_0>0$ is the scale of the radius function of $\Sigma$ in $(M,g)$). 

\begin{remark}
Note that the vector space $W_{\text{Tr}}$ comes naturally equipped with a quotient norm; however, being finite-dimensional, such a norm is equivalent to any norm of our liking. For technical reasons that will emerge at a later stage (cf. Corollary \ref{Cor_App_Uniform Noncon for Jac fields} and Lemma \ref{Lem_Cptness for Slower Growth Jac Fields}) we will assume to work with the $L^{\infty}$ norm determined by the ambient Riemannian manifold $(M,g)$.
\end{remark}
    
\subsection{Normal graphs and transfer of sections}\label{subs:NormalSect}

For each of the following definitions, we let $(M,g)$ denote an ambient Riemannian manifold (not necessarily complete) and $\Sigma_0, \Sigma_1$ denote suitably regular ($C^{\ell\geq 2}$) submanifolds (cf. Remark \ref{rem:OptRegMinSubm}), again possibly open. Furthermore, set $\mbfV_0$ (respectively: $\mbfV_1$) denote the normal bundle of $\Sigma_0$ (respectively: $\Sigma_1$) in $(M,g)$. With such a background metric tacitly understood (to lighten notation), thus letting $\exp_x(\cdot)$ denote the exponential map (in the ambient manifold $(M,g)$) with basepoint $x$, we shall further set $\exp_{x,r}(v):=\exp_x(rv)$ wherever this is well-defined.

We will now present two related, yet different, notions of ``graphicality'': the first notion is essentially local and has to do with the description of a (suitable portion of a) submanifold as a graph over its tangent space at a point, while the second rather relates to viewing an $\MSI$ as a normal graph over a nearby reference $\MSI$.

\begin{definition}\label{def:RegularScale}
    In the setting described above, assuming $k\geq k_0:=7$, if the ambient metric $g$ is of class $C^{k,\alpha}$, we define the \textbf{regularity scale} of $\Sigma_0$ to be the function
    \[
        \mfr_{\Sigma_0, g}: \Sigma_0 \to \RR^+
    \]
    determined by letting $\mfr_{\Sigma_0, g}(x)$ be the supremum (least upper bound) of all $r>0$ such that the following conditions hold:
    \begin{enumerate}[label=(\roman*)]
        \item{$r< \injrad_{M,g}(x)$, and $\|r^{-2}\exp_{x,r}^*g - g_{\eucl}\|_{C^5(\overline{\BB}(1))}\leq 1$;} 
        \item{there exists $\phi:\textrm{dom}(\phi)\subset T_x\Sigma_0\to T^{\perp}_x\Sigma_0$, of class $C^2$, satisfying the inequality \newline $r^{-1} |\phi| + |\mathring{\nabla} \phi| + r|\mathring{\nabla}^2 \phi| \leq 1$  and
        \[(\exp_x)^{-1}(\Sigma_0 \cap B^g(x,r)) = \graph_{T_x \Sigma_0}(\phi)\]
         where we identify the pair $(T_x \Sigma_0, T_x M)$ with $(\mathbb{R}^n, \mathbb{R}^N)$ and set \[
         \graph_{T_x \Sigma_0}(\phi)= \left\{(x,
         \phi(x)) \st x\in \textrm{dom}(\phi))\right\}.
         \]}
     \end{enumerate}
   \end{definition}

   Recall that in the preceding definition $\graph_{T_x \Sigma_0}(\phi)$ is automatically of class $C^{\ell}$ so long as  $\ell\leq k-1$; in particular, when $\Sigma_0$ is an $\MSI$ then such a graph is that of a $C^{k-1,\alpha}$ $\R^{N-n}$-valued function near the origin (namely $\phi$ itself).

\begin{remark}\label{rem:RegCurvBound}
    Note that, by the very definition above ``a bound on the regularity scale implies a curvature bound, as well as an area bound'', i.\,e. in the setting above 
    \be\label{eq:CurvBound}
    |\II_{\Sigma_0,g}(y)|_g\leq C/\mfr_{\Sigma_0, g}(x), \  \ \forall y\in B^g(x,\mfr_{\Sigma_0, g}(x)/C)
    \ee
    and
    \be\label{eq:AreaBound}
    \|\Sigma_0\| (B^g(x,\mfr_{\Sigma_0, g}(x)/C))\leq C\mfr^{n}_{\Sigma_0, g}(x)
    \ee
    for a constant $C\geq 2$ only depending on the dimension $N$ (and on $n\leq N-1)$. Hence, if the ambient manifold is compact we can --- by means of a standard covering argument --- derive a global, uniform area bound as soon as we are given a positive, uniform lower bound on the regularity scale of $\Sigma$. 
    Therefore,
    it is by now standard to show (cf.~e.\,g.~\cite{Sha17}) that a positive, uniform lower bound on the regularity scale implies a compactness result with respect to graphical, smooth convergence (with multiplicity one); as a special case, we will invoke this result for a sequence $\left\{\Sigma_j\right\}_{j\geq 1}$ of smooth, closed minimal submanifolds of dimension $n-1$ in the unit round sphere of dimension $N-1$.
    
\end{remark}

\begin{remark}\label{rem:FailMultOne}
We wish to explicitly stress that, in absence of restrictions on the range of $n$ (with respect to $N$), assuming a uniform bound on the second fundamental form and area \emph{does not} imply smooth graphical convergence with multiplicity one. For instance, one may consider $\Sigma=\Sigma(d)$ to be the disjoint union of the two  spheres in $\mathbb{S}^5$ obtained by intersecting with two three-dimensional subspaces only meeting at the origin, whose distance is equal to $d$ in $\mathbb{G}(3,6)$: as $d\to 0^+$ clearly $\Sigma(d)$ will converge to an equatorial two-sphere with multiplicity two. This phenomenon is ruled out by a positive lower bound on the regularity scale.
\end{remark}

\begin{definition}\label{def:KappaGraph}
    In the setting above 
    we shall say that $\Sigma_1$ is a \textbf{$\kappa$-$C^{m}$ graph} over $\Sigma_0$ in an open subset $U\subset M$ if $\Sing(\Sigma_1) \cap U = \emptyset$, there exist $U\cap \Sigma_0\subset \Omega_0\subset \Sigma_0$ and $\phi\in C^{m}(\Omega_0;\mbfV_0)$ such that
    \[
        \Sigma_1\cap U\subset \graph_{\Sigma_0}(\phi):=\left\{\exp_x(\phi(x)) \st x\in\Omega_0\right\} \subset \Sigma_1
    \]
    implicitly assuming such a definition is well-posed (depending on the injectivity radius of the ambient manifold) and moreover
    \[
        \sum_{j=0\ldots, m}\sup_{x\in \Omega_0} \left(\mfr_{\Sigma_0, g}^{j - 1}(x) \cdot |(\nabla^{\perp})^{j}\phi(x)|_g \right) \leq \kappa.
    \]

\end{definition}

    \begin{remark}\label{Rem_Choice of smallness in graph rad}
      Let $\delta_0\in (1/4)$ be the dimensional constant determined in Lemma \ref{Lem_App_GraphParam_Compare diff graph func}. If, for $i=1, 2$, $\Sigma^i$ is a $\delta_0^2$-$C^2$ graph over $\Sigma^0$ in $U$, then for every $\check x\in \Sigma^0$ such that --- recalling the notion of regularity scale introduced in Definition \ref{def:RegularScale} --- $B^g(\check x, \mfr_{\Sigma^0, g}(\check x))\subset U$, the assumptions of Lemma \ref{Lem_App_GraphParam_Compare diff graph func} are satisfied by the data $\check\eta^{-1}(\Sigma^j), \check r^{-2}\check\eta^*g$ in place of $\Sigma^j, g$ therein, $j=0,1,2$, where $\check r:= \delta_0 \mu_0\mfr_{\Sigma^0, g}(\check x)$ for $\mu_0(N)\in (0,1/2)$ also a dimensional constant, and we have denoted $\check\eta: x\mapsto \exp^g_{\check x}(\check r x)=\exp^g_{\check x,\check r}(x)$. Therefore, Lemma \ref{Lem_App_GraphParam_Compare diff graph func} allows to quantitatively compare the graphical section of $\Sigma^2$ over $\Sigma^1$ with the difference of the graphical sections of $\Sigma^i$ over $\Sigma^0$, $i=1,2$ as well as of slightly translated graphs. 
\end{remark}

\begin{definition}\label{def:GraphRadius}
    In the context of the preceding definition, under the specification that $\Sigma_0$ be an $\MSI$ in $(M,g)$,
   we define the \textbf{graphical radius} $\bfr^{\Sigma_1}_{\Sigma_0, g}(x)$ of $\Sigma_1$ over $\Sigma_0$ under metric $g$ as the infimum over all nonnegative functions $\bfr$ on $\Sing(\Sigma_0)$ so that $\Sigma_1$ is a $\delta_0^2$-$C^2$ graph  over $\Sigma_0$ in the open set  \[
      M\setminus \bigcup_{x\in \Sing(\Sigma_0)}\overline{B^g(x, \bfr(x))}\,.
    \]
    
    We use $\mbfG^{\Sigma_1}_{\Sigma_0, g}$ to denote the (uniquely defined) \textbf{graphical section} over $\Sigma_0$ of $\Sigma_1$ under metric $g$ wherever defined, which we then extend by zero to an $L^\infty$ section on the whole $\Sigma_0$, so that $\mbfG^{\Sigma_1}_{\Sigma_0, g}\in L^{\infty}(\Sigma_0;\mbfV_0)$.
\end{definition}

\begin{definition}\label{def:Convergence}
Let $(M,g)$ be a Riemannian manifold of dimension $N\geq 3$, and let $\left\{g_j\right\}_{j\geq 1}$ be a sequence of $C^{k,\alpha}$ Riemannian metrics converging, in that Banach space, to $g$ as $j\to\infty$. 
Assume, further, to be given an $\MSI$ $\Sigma$ in $(M,g)$ as well as, 
for every $j\geq 1$, an $\MSI$ $\Sigma_j$ in $(M,g_j)$. We shall say that the sequence $\left\{\Sigma_j\right\}_{j\geq 1}$ \textbf{converges} to $\Sigma$ in $C^{k',\alpha'}_{loc}$ (for $0\leq k'\leq k$ and $0\leq\alpha'\leq\alpha$) if, as one lets $j\to\infty$, the following two conditions hold:
\[
\bfr^{\Sigma_j}_{\Sigma, g}(x)\to 0 \  \text{for all} \ x \in \Sing(\Sigma)
\]
and for any $r>0$ 
\[
\|\mbfG^{\Sigma_j}_{\Sigma, g}\|_{C^{k',\alpha'}(\Sigma\setminus B^g(\Sing(\Sigma),r))}\to 0,
\]
where it is understood, throughout this article, that $B^g(\Sing(\Sigma),r)=\cup_{p\in \Sing(\Sigma)}B^g(p,r)$.
\end{definition}

That said, whenever we have ``graphical convergence''  we can in fact ``transfer sections'' employing the exponential map. In order to clarify this point, we give the following:

\begin{definition}\label{def:Transfer}

In the setting of Definition \ref{def:KappaGraph}, set $\Sigma_1=\graph_{\Sigma_0}(\phi)$. Given any section $u:\Sigma_1\to\mbfV_1$ we let $\mbfT^{\Sigma_1}_{\Sigma_0,g}(u):\Sigma_0\to \mbfV_0$ be defined as follows
\[
\mbfT^{\Sigma_1}_{\Sigma_0,g}(u)(x)=[(\exp_{x}(\phi(x)))^{-1}_{\ast}(u(\exp_x(\phi(x))))]^{\perp_g}.
\]
\end{definition}

In practice, in this paper we will deal with ``transferring sections'' in two special cases:
\begin{itemize}
\item{a sequence of $\MSI$ $\left\{\Sigma_j\right\}_{j \geq 1}$} converging (in $\mbfF$-metric, with multiplicity one, and) in $C^{k',\alpha'}_{loc}$ as per Definition \ref{def:Convergence} to a reference ``base'' $\MSI$ $\Sigma$;
\item{an $\MSI$ that -- based on \cite{Sim83} -- can be written, locally around $q\in\Sing(\Sigma)$ as a normal graph over its unique tangent cone $\mbfC_q=\mbfC_q(\Sigma)$ in $T_q M$.}
\end{itemize}
Note that the second instance can also be phrased in terms of ``convergence'' on approach to $q$ by virtue of the decay properties of the normal section defining $\Sigma$ over $\mbfC_q$.

The key analytic features of the transfer operators are collected in Proposition \ref{Prop_MSE_Section Transp}, in a way that is equally well applicable to both circumstances.
The second special case (normal graph over tangent cone) will also feature prominently in the second part of Appendix \ref{sec:3circles}, whose key results will in turn be employed in Section \ref{sec:GenReg}.

\subsection{An ancillary computation about translation-like sections}\label{subs:Ancill} The following result will be profitably used throughout the paper (for instance in the proof of Lemma \ref{Lem_Anal for slower growth func}).

    \begin{lemma}\label{lem:JacobiTranslFunct}
    Let $(M,g)$ be a closed Riemannian manifold, and let $\Sigma$ be a $g$-stationary varifold with only strongly isolated singularities ($\MSI$). For any translation-like section $\phi: \Sigma\to \mbfV:=\mbfV(\Sigma,M)$ belonging to the space $W_{\text{Tr}}$ there holds $L_{\Sigma,g}\phi\in L^{\infty}(\Sigma;\mbfV)$ and there exists a constant $C=C(g,\Sigma)$ such that, if $\phi=\mbfx^\perp$ around $q\in\Sing(\Sigma)$  then (in that same neighborhood)
    \[
    \|L_{\Sigma,g}\mbfx^\perp\|_{L^{\infty}}\leq C|\mbfx|,
    \]
    where $|\mbfx|$ denotes the norm of $\mbfx\in T_q M$ with respect to the metric $g$.
    In particular, there holds  $L_{\Sigma,g}\phi\in W^{0,p}_{-\varepsilon}(\Sigma;\mbfV)$ for any $p\geq 1$ and any $\varepsilon>0$.
    \end{lemma}

    \begin{proof}

We work in the neighborhood of $q\in\Sing(\Sigma)$ endowed with geodesic normal coordinates centered at the point in question. Thus, with slight notational abuse, we will assume $\Sigma^n \subset (\R^N,g)$  and $\phi=\mbfx^\perp$; we shall employ the well-known asymptotic expansion 
\be\label{eq:MetrucNormalExpansion}
g_{ij}(x)=\delta_{ij}-\frac{1}{3}R_{ijk\ell}x^k x^\ell+O(|x|^3), \ (x\to 0).
\ee
Hence, by letting as usual $r:=|x|$, there holds in particular for the Christhoffel symbols and their first derivatives
\be\label{eq:ChristNormalExpansion}
\Gamma^k_{ij}=O(r), \ \text{and} \ \Gamma^k_{ij,\ell}=O(1), \ (r\to 0^+).
\ee
Now, it suffices in fact to consider the case when $g(\mbfx,\mbfx)=1$ at $q$ and prove the desired local $L^\infty$ bound, for the claim then comes by a standard scaling argument.

That said, for $\mbfx=\sum_{i=1}^N \mbfx_i\partial^i$ we need to compute at any $x\in\Sigma$ the quantity
\be\label{eq:JacobiMean}
(L_{\Sigma,g}(\mbfx^\perp))(x)=\left[\frac{d}{ds}\right]_{s=0}H_{\Sigma+s\mbfx,g}(x_s)
\ee
where $x_s$ is the unique element in the intersection $\cup_{t\in [0,2s]}\exp_x(t\mbfx^\perp) \cap (\Sigma+s\mbfx)$, which is well-defined for $s=s(x)$ sufficiently small. In fact, due to the assumed minimality of $\Sigma$ we have
\[
\left[\frac{d}{ds}\right]_{s=0}H_{\Sigma+s\mbfx,g}(x_s-(x+s\mbfx))=g(H_{\Sigma,g},\mbfx^\perp-\mbfx)=0
\] 
and so we are actually left with computing
\[
\left[\frac{d}{ds}\right]_{s=0}H_{\Sigma+s\mbfx,g}(x+s\mbfx)=\left[\frac{d}{ds}\right]_{s=0}H_{\Sigma,\tau^*_{s\mbfx}g}(x)
\]
where the right-hand side is now the variation of mean curvature of a fixed submanifold with respect to a deformation of the ambient metric, and we have set $\tau_{W}(x)=x+W$. 
To move further let us recall that, if we denote $h:=[\frac{d}{dt}]|_{t=0}g(t)$ then there holds in general, again under the minimality assumption
\[
\left[\frac{d}{dt}\right]_{t=0}H_{\Sigma, g(t)}=-g(\II_{\Sigma,g},h)+((\div_{\Sigma}h)^{\#})^{\perp_{g}}-\frac{1}{2}\tr_{\Sigma}\nabla^{\perp_g}h;
\]
working in local coordinates $\left\{x^1,\ldots x^n,x^{n+1},\ldots, x^N\right\}$ where $\left\{\partial_1,\ldots \partial_n\right\}$ (and, respectively: $\left\{\partial_{n+1},\ldots,\partial_N\right\})$ form a local basis for the the tangent (resp. normal) space to $\Sigma$
the three terms above need to be understood as follows :
\[
-g(\II_{\Sigma,g},h)=-g^{ip}g^{jq}h_{pq}\Gamma^a_{ij}\partial_a, \
((\div_{\Sigma}h)^{\#})^{\perp_{g}}=g^{ij}g^{ab}h_{ib,j}\partial_a, \
\tr_{\Sigma}\nabla^{\perp_g}h=-g^{ij}g^{ab}h_{ij,b}\partial_a
\]
for indices $1\leq i,j,p,q\leq n, n+1\leq a,b\leq N$ and summation over repeated indices.

As a result, since in our case
\[
\left[\frac{d}{ds}\right]_{s=0}\tau^*_{s\mbfx}g=\scL_{\mbfx}g=\sym (\nabla \mbfx)
\]
(that is: the Lie derivative of the metric $g$ in the direction $\mbfx$, which equals the ``symmetrized covariant derivative'' of $\mbfx$ with respect to the metric $g$) we will have, in the end
\be\label{eq:FinalVarMeanC}
L_{\Sigma,g}(\mbfx^{\perp})=-g(\II_{\Sigma,g},\sym (\nabla \mbfx))+((\div_{\Sigma}\sym (\nabla \mbfx))^{\#})^{\perp_{g}}-\frac{1}{2}\tr_{\Sigma}\nabla^{\perp_g}\sym (\nabla \mbfx).
\ee
At this stage we simply need to note that the tensor $h:=\sym (\nabla \mbfx)$ satisfies, in view of \eqref{eq:ChristNormalExpansion} and the constancy of $\mbfx$ the bound $\|h\|_g=O(r)$ as $r\to 0^+$, and so, by inspecting all terms of the right-hand side of \eqref{eq:FinalVarMeanC} we conclude that each of them is uniformly bounded as $r\to 0^+$ and so will the same conclusion hold for $L_{\Sigma,g}(\mbfx^\perp)$, as claimed.
    \end{proof}

    \begin{remark}\label{rem:LinftyCanNeighb}
    By inspecting the previous proof it is straightforward to see that, in fact, the constant $C=C(g,\Sigma)$ can be chosen uniformly for all pairs (as per Definition \ref{def:Pairs} below)
    $(g', \Sigma')\in \cM_n^{k, \alpha}(M)$ satisfying the bounds
        \[
\|g'\|_{C^{k,\alpha}}\leq \Lambda, \ \ \mfr_{\Sigma', g'} \geq \Lambda^{-1}\rho_{\Sigma',g'}
         \]
            where we recall that $\rho_{\Sigma',g'}$ denotes the radius function of $\Sing(\Sigma')$ in $(M, g')$.
   That is to say: the claim of Lemma \ref{lem:JacobiTranslFunct} holds true for all $\phi\in W_{\text{Tr}}$ for a constant $C=C(\Lambda)$.
    \end{remark}

\subsection{Conventions}\label{subs:Conv} We shall collect here some more notation and conventions that are implicitly assumed throughout the paper.

\vspace{3mm}

\textbf{Cutoff functions.} First of all, let $\zeta\in C^\infty(\RR; [0,1])$ be a cutoff function (fixed once and for all) such that $\zeta \equiv 0$ on $(-\infty, 1]$ and $\zeta \equiv 1$ on $[2, +\infty)$. Hence, given as above an ambient Riemannian manifold $(M,g)$, an $\MSI$ $\Sigma$ and $r_0>0$, we let
            \be\label{Equ_Def zeta_{Sigma, g, r_0}}
              \zeta_{\Sigma, g, r_0}(x):= \zeta(\rho_{\Sigma,g}(x)/r_0),  
            \ee
            where $\rho_{\Sigma,g}$ is a radius function for $\Sigma$ in $(M,g)$ and $r_0(\Sigma,g)$ denotes its scale (see Section \ref{subs:Functional}).
            When convenient we will simplify this notation; for instance given a sequence of data of the type above we may write $\zeta_j$ in place of $\zeta_{\Sigma_j, g_j, r_j}$. Specific conventions of this sort will however always be declared. 

            \vspace{3mm}

\textbf{Metric notions.} In $\R^N$ we shall denote by $g_{\eucl}$ the standard Euclidean metric, by $\dist_{g_{\eucl}}$ the associated distance and by $\mathring{\nabla}$ its Levi-Civita connection. Furthermore, we let $\BB(x,r)$ denote the open Euclidean ball of center $x\in \R^N$ and radius $r>0$; in the special case of balls centered at the origin we shall simply write $\BB(r)$. For Euclidean annuli we write $\AAa(x,r,s)=\BB(x,s)\setminus \overline{\BB(x,r)}$; when $x$ is the origin we convene to write $\AAa(r,s)$.
In a Riemannian manifold $(M,g)$ we let $\dist_g$ denote the corresponding distance and by $\nabla$ the Levi-Civita connection. 
It is convenient to explicitly indicate the background metric when talking about balls and annuli, so we will write $B^g(x,r)$ (respectively: $A^g(x,r,s)$) for the ball of center $x\in M$ and radius $r>0$ (respectively: for the annulus of center $x$ and radii $0<r<s$). At certain, relatively rare, points of the paper we will need to deal simultaneously with multiple metrics, in which case we will rather indicate the selected metric explicitly, like e.\,g. for $\nabla^g$ in lieu of $\nabla$ and for $\exp_x^g$ in lieu of $\exp_x$ when employing the exponential map; in particular, this will be the case in Appendix \ref{app:MSE}.

            \vspace{3mm}

\textbf{Use of constants.} Throughout this article we shall employ the letter $C$ to denote a constant that is allowed to vary from line to line (or even within the same line); we shall stress the functional dependence of any such constant on geometric quantities by including them in brackets, writings expressions like $C=C(\Lambda,\sigma)$ or similar. The dimension of the ambient manifold (that is: $N$) and of the subvarieties we work with (namely: $n$) are fixed (i.\,e. we do not need to ever vary them in our arguments) and so, for notational simplicity, we agree not to indicate them among the parameters our constants depend upon; the only (explicitly stated) exceptions to such a rule are in Appendix \ref{app:MSE} and Appendix \ref{app:GraphicalPar}, where we determine the constant $\delta_0=\delta_0(N)$ that was mentioned above after the definition of $\kappa$-$C^k$ graph (Definition \ref{def:KappaGraph}). Lastly, in the rare cases when -- within a certain proof -- it is appropriate to keep track of different constants (for instance because they display different functional dependence, or need to be chosen in a certain order) we avoid ambiguities by employing different labels or numbers to indicate such constants. This will always be explicitly remarked if appropriate.

\section{An index-theoretic perspective on generic regularity}\label{sec:Index}

As anticipated in the introduction, 
we shall prove in this section an ``index-counting formula'' of independent interest.  To get there, it is convenient to open a short digression pertaining to the analysis of the Jacobi operator of regular minimal cones in Euclidean space, which will anyways come into play along the course of this work (primarily in Appendix \ref{sec:3circles}).

\subsection{Analysis on a regular minimal cone}\label{subs:AnalysisCone1}
Given $n\geq 2$ and $N>n$ let $\mathcal{C}_{N,n}$ denote the collection of \emph{non-trivial} regular $n$-dimensional minimal cones $\mathbf{C}$ in Euclidean $\mathbb{R}^N$, always understood as \emph{multiplicity one} varifolds.

Thus, for $\mbfC\in \cC_{N,n}$ there always holds that $S:= \mbfC\cap \SSp^{N-1}$ is a smooth minimal submanifold in $\SSp^{N-1}$ other than an $(n-1)$-dimensional equator; as a result, we stress in particular that $\mathcal{C}_{N,n}$ is actually \emph{empty} when $N=3$ and $n=2$. We can parametrize $\mbfC$ by 
    \begin{align}
        (0, +\infty)\times S\to \mbfC,\ \ \ (r, \omega)\mapsto x = r\omega.  \label{Equ_Pre_Parametrize minimal hypercone}
    \end{align}
At the cone $\mbfC\subset \R^{N}$ one deals with the Jacobi operator ($r:=|x|$)
\[
L_{\mbfC}=\partial^2_r +\frac{n-1}{r}\partial_r+\frac{1}{r^2}(L_S-(n-1)).
\]
(Note that, compared to the general notational conventions stipulated in Section \ref{sec:Prelim} here we omit the explicit indications of the background metric for $L_{\mbfC}$ and $L_S$.)

If one looks for homogeneous Jacobi fields, arguing by separation of variables (cf. \cite{CHS84}) the only possible rates of growth/decay correspond to the real part of the complex numbers 
solving the algebraic equation (one for each choice of the label $j$ parametrizing the eigenfunctions on the link, with the sign conventions declared above)
\[
\gamma^2+(n-2)\gamma-(\lambda_j+(n-1))=0
\]
which are 
\be\label{eq:Roots}
\gamma^{\pm}_j(\mbfC)=-\left(\frac{n-2}{2}\right)\pm\sqrt{\left(\frac{n-2}{2}\right)^2+(\lambda_j+(n-1))} \ \in\mathbb{C}.
\ee
Therefore, we collect the real part of these numbers in the set (henceforth referred to as \textbf{asymptotic spectrum} of $\mbfC$)
\be\label{eq:AsymptSpectr}
\Gamma(\mbfC) := \{\deg(w): w \text{ is a homogeneous Jacobi field on }\mbfC\} \equiv \{\Re(\gamma^\pm_j): j\geq 1\}\subset \R
\ee
and conveniently define
\be\label{Equ_GenReg:AsympSpec_Cone}
    \gamma_-(\mbfC) := \sup (\Gamma(\mbfC)\cap \R_{<1})<1\,. 
    \ee

\begin{remark}\label{rem:IndexContrib}
Note for later reference that the inequality $\lambda_j<0$ (that corresponds to the ``index contributions'' of the link) is true if and only if 
\be\label{eq:MagicThreshold}
\gamma_j^{+}=-\left(\frac{n-2}{2}\right)+\sqrt{\left(\frac{n-2}{2}\right)^2+(\lambda_j+(n-1))}<1.
\ee
With this in mind, it is convenient for us to set
$\gamma_*=-(n-1)$ and 
$\gamma^*=1$ (that are the roots associated with $\lambda_j=0$). Furthermore, due to the presence of translations in Euclidean $\R^N$, it is always the case that $0\in\Gamma(\mbfC)$  and so it follows that $\gamma_{-}(\mbfC)\geq 0$ for any $\mbfC\in\cC_{N,n}$.
\end{remark}

  That said, and conveniently set $\mu_j=\lambda_j+(n-1)$ for any $j\geq 1$, a general (real-valued) Jacobi field $v\in C^\infty_{loc}(\mbfC;\mbfV)$ is thus given by a linear combination of homogeneous Jacobi fields,
    \begin{align}
        v(r, \omega) = \sum_{j\geq 1} (v_j^+(r) + v_j^-(r))\varphi_j(\omega), \label{Equ_Pre_Decomp Jac into homogen Jac field}
    \end{align}
    where 
    \begin{align}
    \begin{split}
        v_j^+(r) & = c_j^+\cdot r^{\gamma_j^+}; \\
        v_j^-(r) & = \begin{cases}
            c_j^- \cdot r^{\gamma_j^-},       & \ \text{ if }\mu_j \neq -\frac{(n-2)^2}{4}; \\
            c_j^- \cdot r^{\gamma_j^-}\log r, & \ \text{ if }\mu_j = -\frac{(n-2)^2}{4}.
        \end{cases}
    \end{split} \label{Equ_Pre_Homogen Jac field}
    \end{align}
    for some $c_j^\pm\in \RR$ if $\mu_j\geq -(n-2)^2/4$, while instead $c_j^\pm\in \CC$, $c_j^-=\overline{c_j^+}$ when $\mu_j<-(n-2)^2/4$.

\subsection{An index-counting formula}\label{subs:IndexCount} Here is the main result of this section:

 \begin{theorem}\label{thm:Count}
    Let $(M, g)$ be an $N$-dimensional closed Riemannian manifold, where the metric $g$ belongs to $C^{k,\alpha}$ for some $k\geq 4$, 
    and let $\Sigma$ be a stationary integral $n$-varifold in $(M,g)$ with only strongly isolated singularities (as per Definition \ref{def:MSI} of $\MSI$). 
    Let $2\leq m\leq k-2$ and 
    \be\label{eq:RangeTau}
      \tau \in \left(\sup_{p\in \Sing(\Sigma)} \gamma_-(\mbfC_p), 1 \right) \,.
    \ee
    Then the Jacobi operator \[
      L_{\Sigma,g} : \hat{W}^{m,2}_\tau(\Sigma;\mbfV) \to W^{m-2,2}_{\tau-2}(\Sigma;\mbfV)
    \]
    is a Fredholm operator of index \[
      \widehat{\ind}_{\tau}(L_{\Sigma,g}) = -\sum_{p\in \Sing(\Sigma)} \rmI(\mbfC_p)\,.
    \]
  \end{theorem}

Towards the proof of this result, let us first handle the effect of the aforementioned ``augmentation''. We will employ the following functional-analytic statement, whose proof is straightforward:

\begin{lemma}\label{lem:augment}
Let $X,Y$ be Banach spaces and let $T:X\to Y$ be a bounded linear operator. Assume $X_0\subset X$ be a closed subspace of finite codimension equal to $d\in\N$ and let $T_0$ denote the restriction of $T$ to $X_0$. Then $T$ is Fredholm if and only if $T_0$ is, and the corresponding indices are related by the equation
\[
\ind(T)=\ind(T_0)+d.
\]
\end{lemma}

\begin{proof}[Proof of Theorem \ref{thm:Count}]
We claim that the operator
      $L_{\Sigma,g} : W^{m,2}_\tau(\Sigma;\mbfV) \to W^{m-2,2}_{\tau-2}(\Sigma;\mbfV)$ is Fredholm and that its index satisfies
\be\label{eq:IndCount1}
\ind_{\tau}(L_{\Sigma,g}) =-\sum_{p\in \Sing(\Sigma)} (\rmI(\mbfC_p)+N)=-NQ-\sum_{p\in \Sing(\Sigma)} \rmI(\mbfC_p)
\ee
where, we recall, $Q$ is the notation we employ to denote the number of singular points of $\Sigma$. If that is the case, then the desired conclusion comes at once from Lemma \ref{lem:augment} since --- like we discussed in Section \ref{subs:Functional} --- as soon as $\tau>0$ (which is indeed the case, by virtue of the last assertion in Remark \ref{rem:IndexContrib}) $W^{m,p}_\tau(\Sigma;\mbfV)$ is a closed subspace of $\hat{W}^{m,p}_\tau(\Sigma;\mbfV)$ of codimension equal to $NQ$.

That said, in order to prove \eqref{eq:IndCount1} we will appeal to Lockhart-McOwen theory (specifically to Theorem 6.2 and Section 7 in \cite{LocMcO85}; see also Section 9 of \cite{Pac13}, to be specified to the ``conically singular'' (CS) case only). For this reason we shall consider the set $\Gamma(\mbfC_i)$ (as per \eqref{eq:AsymptSpectr})
and further define
\be\label{eq:IndicialProduct}
\Gamma(\Sigma):=\left\{\boldsymbol{\gamma}=(\gamma_1,\ldots,\gamma_{Q})\in\R^{Q} \st \gamma_i\in\Gamma(\mbfC_i) \ \text{for some} \ i\in\left\{1,\ldots,Q\right\}\right\}.
\ee
Let then $L_{\Sigma,g}^*$ denote the adjoint of $L_{\Sigma,g} : W^{m,2}_\tau(\Sigma;\mbfV) \to W^{m-2,2}_{\tau-2}(\Sigma;\mbfV)$, so (by the duality result recalled in Section \ref{subs:Functional})
we will then have
\be\label{eq:Adjoint}
L_{\Sigma,g}^*: W^{2-m,2}_{2-\tau-n}(\Sigma;\mbfV)\to W^{-m,2}_{-\tau-n}(\Sigma;\mbfV);
\ee
let $\ind_{\tau}(L^{*}_{\Sigma,g})$ denote the Fredholm index of such an adjoint.
Since, however, by elliptic regularity the  Fredholm index of $L_{\Sigma,g}$ is independent of the choice of $m$ within any connected component of $\R^{Q}\setminus \Gamma(\Sigma)$, by self-adjointness 
we have that
$
\ind_{\tau}(L_{\Sigma,g}^*)=\ind_{2-\tau-n}(L_{\Sigma,g}).
$
On the other hand, we can rely on the usual ``orthogonality relations'' for the adjoint, which (in the context under consideration) give
$
\ind_{\tau}(L_{\Sigma,g}^*)=-\ind_{\tau}(L_{\Sigma,g}).
$
Hence combining these two equations we finally get
\be\label{eq:AdjFinal}
\ind_{2-\tau-n}(L_{\Sigma,g})=-\ind_{\tau}(L_{\Sigma,g}).
\ee
The next step is then to compute the \emph{difference}
$
\ind_{\tau}(L_{\Sigma,g})-\ind_{2-\tau-n}(L_{\Sigma,g})
$
at least for suitably chosen values of $\tau$; it is here that the indicial roots (and their multiplicities) come into play. 

In order to relate this number to the Morse index of the links, keeping in mind Remark \ref{rem:IndexContrib} we note that picking 
\[
\tau\in \left(\sup_{p\in \Sing(\Sigma)} \gamma_-(\mbfC_p), 1 \right)
\]
there holds in fact $-n+1<2-\tau-n<\tau$ and the interval $(2-\tau-n,\tau)$ intercepts all (and only) those indicial roots associated to negative eigenvalues of the links, namely to index contributions. 
That said, we further note that each such eigenvalue ``generates'' two indicial roots (possibly coinciding, i.\,e. one root with multiplicity two). 

Therefore, for any such fixed value of $\tau$ the weight-crossing formula in \cite{LocMcO85} gives
\[
\ind_{\tau}(L_{\Sigma,g})-\ind_{2-\tau-n}(L_{\Sigma,g})=-2\sum_{j=1\ldots, Q}(\rmI(\mbfC_j)+N)
\]
and so, by \eqref{eq:AdjFinal} we then get the claimed conclusion.
\end{proof}

\subsection{The case of the Clifford football in the four-sphere}\label{subs:Clifford}

Let us now revisit the previous discussion by specializing to the case of our primary interest, that is the application of this result to the analysis of the Clifford football in $\mathbb{S}^4$.

\begin{example}
In this case we have $N=4$, $n=3$ and $Q=2$, each of the two singularities being modelled on the ``Clifford cone'', the cone over the Clifford torus in $\mathbb{S}^3$. The eigenvalues of the Jacobi operator of the link are computed --- by separation of variables --- to be
\[
\lambda^{p,q}=2[(p^2+q^2)-2], \ \ \ (p,q\in\N)
\]
so that, in particular (properly accounting for the multiplicity of the associated eigenspaces)
\[
\lambda_1=-4, \ \lambda_2=\lambda_3=\lambda_4=\lambda_5=-2, \ \lambda_6=\lambda_7=\lambda_8=\lambda_9=0
\]
and $\lambda_j>0$ for any $j\geq 10$.
We then get the corresponding values of $\gamma$:
\begin{multline*}
\gamma^{\pm}_1=-\frac{1}{2}\pm\frac{\sqrt{-7}}{2}, \
\gamma^{+}_2=\gamma^{+}_3=\gamma^{+}_4=\gamma^{+}_5=0, \
\gamma^{-}_2=\gamma^{-}_3=\gamma^{-}_4=\gamma^{-}_5=-1 \\
\gamma^{+}_6=\gamma^{+}_7=\gamma^{+}_8=\gamma^{+}_9=1, \
\gamma^{-}_6=\gamma^{-}_7=\gamma^{-}_8=\gamma^{-}_9=-2
\end{multline*}
and for $j\geq 10$ we will have instead $\gamma^+_j>1$ and $\gamma^-_j<-2$.

Hence
$
\gamma_{-}(\mbfC_1)=\gamma_{-}(\mbfC_2)=0$ and so Theorem \ref{thm:Count} is a statement about the Fredholm index of the Clifford football when each weight $\beta_j$ (equivalently: $\tau$ in that setting) lies in the open interval $(0,1)$. Next, following the previous argument, the counting of the indicial roots refers to the interval $(-2,1)$ and gives for the index of the Clifford football the value -10  for all $\tau\in (0,1)$; note that for $\tau\in (1,2)$ we would get instead the value $-18$.
\end{example}

\section{Meagerness of singularities with positive normalized Morse index}\label{sec:GenReg}

In this section we will complement the previous, general index-counting formula with the following assertion about the generic sign of the Fredholm index of the Jacobi operator. We let, throughout, $n\geq 2, N>n$ and assume $k\geq k_0$ as well as $\alpha\in (0,1)$.

    \begin{theorem}\label{thm:generic_perturbation}
        Let $(M, g)$ be an $N$-dimensional closed Riemannian manifold with a generic choice of metric $g$, understood either in $C^{k,\alpha}$ or in $C^\infty$, and let $\Sigma$ be a stationary integral $n$-varifold with only strongly isolated singularities. Let $\tau$ satisfy  equation \eqref{eq:RangeTau}.
        Then
        \[
            \widehat{\ind}_{\tau}(L_{\Sigma,g}) \geq 0 \,.
        \]
    \end{theorem}

\begin{remark}
As apparent from the preceding statement, there are in fact two versions of this theorem, depending on the (regularity of the) space of metric we work with. For the vast majority of this section we will in fact work with the space of $C^{k,\alpha}$ metrics, and --- at the very end --- we will present the argument that allows to derive the smooth version of the theorem from that in finite regularity.
\end{remark}

\subsection{Compact subclasses of regular minimal cones}\label{sub:Compact}

\begin{definition}\label{def:CcL}
Recalling the definition of $\mathcal{C}_{N,n}$ given at the beginning of Section \ref{sec:Index}, we shall set for any $\Lambda>0$
\[
\mathcal{C}_{N,n}(\Lambda):=\{\mbfC\in \cC_{N,n}: \inf_{\mbfC \cap \SSp^{N-1}} \mathfrak{r}_\mbfC(x) \geq \Lambda^{-1}\},
\]
where, just by specializing Definition \ref{def:RegularScale},
    \[
    \mathfrak{r}_\mathbf{C}(x) := \sup\left\{r > 0 : \mbfC\cap \BB(x,r)=\graph_{T_x \mbfC}(u), \ r^{-1}|u|+|\mathring{\nabla} u|+r|\mathring{\nabla}^2 u|\leq 1\right\}
    \]
for $u:\textrm{dom}(u)\subset T_x\mbfC\to T^{\perp}_x\mbfC$, of class $C^2$.
\end{definition}

\begin{remark}\label{rem:SufficesBundRegScale}
In view of Remark \ref{rem:RegCurvBound} we note that a positive, uniform lower bound on the regularity scale of the link $\mbfC \cap \SSp^{N-1}$ implies an upper bound on its $(n-1)$-dimensional area, hence on the density of the cone $\theta_\mbfC(\orig)$. 
\end{remark}

    \begin{lemma} \label{Lem_cC_(N,n)(Lambda) bF Cpt}
        For every $\Lambda \geq 1$, $\cC_{N,n}(\Lambda)$ is compact in $\mbfF$-metric and in $C^\infty$ topology along the cross-section.
    \end{lemma}
    \begin{proof}
        For every $\mbfC \in \mathcal{C}_{N, n}(\Lambda)$, its cross-section $\Sigma := \mbfC \cap \SSp^{N-1}$ is a minimal submanifold of dimension $n - 1$, which (by definition) satisfies a positive, uniform lower bound on the regularity scale.
    
        Therefore, in view of Remark \ref{rem:RegCurvBound}, every sequence $\{\mbfC_i\}_{i\geq 1} \subset \mathcal{C}_{N, n}(\Lambda)$ has a subsequence, still denoted by $\{\mbfC_i\}$, such that $\Sigma_i := \mbfC_i \cap \SSp^{N-1}$ converges to a minimal submanifold $\Sigma_\infty$ of $\SSp^{N-1}$ in $C^\infty$. Here, we stress once again that the convergence occurs with multiplicity one.
        
        In particular, $\mbfC_i \to \mathbf{C}_\infty := 0 \conetimes \Sigma_\infty$ (the minimal cone over $\Sigma_\infty$) under the $\mbfF$-metric. Note that $\Sigma_\infty$ cannot be a standard $\mathbb{S}^n$, because if it were, $\mathbf{C}_\infty$ would be an $n$-dimensional subspace; thus, by Allard's regularity theorem \cite{All72}, this would imply that $\mathbf{C}_i$ is a smooth cone (for any large enough $i$) and therefore also an $n$-dimensional subspace, which contradicts the assumption that $\mathbf{C}_i \in \mathcal{C}_{N, n}$.
        Lastly, as a direct consequence of the definition of regularity scale, one can verify that $0 \conetimes \Sigma_\infty \in \mathcal{C}_{N, n}(\Lambda)$.
    \end{proof}

 \begin{remark}\label{rem:SpectrumConv}
    If $\{\mbfC_i\}_{i\geq 1}$ is a family of regular cones such that 
    \[
        \mbfF(|\mbfC_i|, |\mbfC_\infty|) \to 0
    \] 
    (with multiplicity one)
    then the asymptotic spectrum also converges in the following sense: for every $j\geq 1$, as $i\to \infty$,
    \[
        \gamma_j^\pm(\mbfC_i) \to \gamma^\pm_j(\mbfC_\infty)\,.
    \]
    So, this conclusion applies in particular whenever we are within the range of applicability of the preceding lemma.
\end{remark}

As a simple application of this compactness result we prove the following inequality, which we will employ in the sequel of this paper.

\begin{lemma}\label{lem:L2perp}
Given $\Lambda>0$ there exists a constant $C=C(\Lambda)$ such that the following holds: for any cone $\mbfC\in\cC_{N,n}(\Lambda)$ and for any $\mbfx\in\R^N$, set $\mbfA(1,2):=\mbfC\cap\AAa(1,2)$, there holds
\be\label{eq:L2perp}
|\mbfx|\leq C(\Lambda)\|\mbfx^{\perp}\|_{L^2(\mbfA(1,2))}.
\ee
\end{lemma}

\begin{proof}
Assume, towards a contradiction, that there exists a sequence $\left\{\mbfC_i\right\}_{i\geq 1}$ in $\cC_{N,n}(\Lambda)$ and corresponding vectors $\mbfx_i \in \R^{N}$ (say normalized to have unit norm) such that
\[
\|\mbfx_i^{\perp}\|_{L^2(\mbfA_i(1,2))}\leq 1/i,
\]
where we have set (like we did above) $\mbfA_i(1,2)=\mbfC_i\cap\AAa(1,2)$.
Now, by Lemma \ref{Lem_cC_(N,n)(Lambda) bF Cpt}, possibly extracting a subsequence (which we do not rename) there holds  $\mbfF(|\mbfC_i|, |\mbfC_\infty|) \to 0$ for some $\mbfC_\infty\in \cC_{N,n}(\Lambda)$. (We stress that, in particular, this implies the non-triviality of the limit cone.) In addition, the sequence $(\mbfx_i)$ has itself a converging subsequence to a limit element $\mbfx_\infty\in\R^N$ with $|\mbfx_\infty|=1$.

Hence, passing to the limit in the previous inequality we would obtain $\|\mbfx_{\infty}^{\perp}\|_{L^2(\mbfA_{\infty}(1,2))}=0$, which is only possible if the flow of translations generated by $\mbfx_\infty$ in $\R^N$ preserves $\mbfC_\infty$. However, this contradicts the fact that $\Sing(\mbfC_{\infty})=\left\{0\right\}$ (which is not kept still by any translation). Thereby the proof is complete.
\end{proof}

We proceed with a (by now well-known) result on the discreteness of the space of densities of minimal cones, that is key (among other things) to proving Theorem \ref{Thm_Countable Decomp}, which we will do in Appendix \ref{Sec:Count_Decomp}.

\

    \begin{lemma} \label{Lem_cC_(N,n)(Lambda) finitely many density}
        $\{\theta_\mbfC(\orig):\mbfC\in \cC_{N,n}(\Lambda)\}$ is a finite subset of $\RR$.
    \end{lemma}
    \begin{proof}
        Suppose by contradiction that there exists a sequence of $\{\mbfC_i\}_{i\geq 1} \subset \mathcal{C}_{N, n}(\Lambda)$ with pairwise distinct densities at the origin.

        By the preceding compactness result (Lemma \ref{Lem_cC_(N,n)(Lambda) bF Cpt}), $\Sigma_i := \mbfC_i \cap \SSp^{N-1}$ converges in $C^\infty$ to a minimal submanifold $\Sigma_\infty$. Let then $\mbfC_\infty := \mathbf{0} \conetimes \Sigma_{\infty}$. Hence, for sufficiently large $i$, $\Sigma_i$ can be expressed as the graph of a smooth normal section
        \[
            u_i: \Sigma_\infty \to \Sigma^\perp_\infty\,.
        \]
        
        It follows from the \oldL{}ojasiewicz-Simon inequality \cite[Theorem~3]{Sim83} that for every $\mu \in (0, 1)$, there exist  $\vartheta(\Sigma_\infty) \in (0, 1/2)$, and $\delta(\Sigma_\infty, \mu)$ such that every $C^{2, \mu}$ normal section $u: \Sigma_\infty \to \Sigma^\perp_\infty$ with $\|u\|_{C^{2,\mu}(\Sigma_{\infty})} \leq \delta$ shall satisfy the bound on the $(n-1)$-dimensional Hausdorff measure
        \[
            |\scH^{n-1}(\graph_{\Sigma_\infty}(u)) - \scH^{n-1}(\graph_{\Sigma_\infty}(0))|^{1 - \vartheta} \leq \|\scM(u)\|_{L^2(\Sigma_\infty)}\,,
        \]
        where $\scM(u) = 0$ corresponds to the minimal surface equation (cf. Appendix \ref{app:MSE}). In particular, for sufficiently large $i$, $\|u_i\|_{C^{2,\mu}(\Sigma_{\infty})} \leq \delta$ and thus,
        \[
            |\scH^{n-1}(\graph_{\Sigma_\infty}(u_i)) - \scH^{n-1}(\graph_{\Sigma_\infty}(0))|^{1 - \vartheta} = 0\,,
        \]
        which implies that $\theta_{\mbfC_i}(\mathbf{0}) = \theta_{\mbfC_\infty}(\mathbf{0})$, a contradiction to our assumption.
    \end{proof}

     \begin{remark}\label{rem:DensityFunct}  
 Following up on the preceding Remark \ref{rem:SufficesBundRegScale}, we note that Lemma \ref{Lem_cC_(N,n)(Lambda) bF Cpt} implies that we can define
\be\label{eq:DensityFunct}
\Theta(\Lambda):=\sup_{\mbfC\in\cC_{N,n}(\Lambda)}\theta_{\mbfC}(\mathbf{0})>0
\ee
which does in fact determine a monotone (non-decreasing) function $\Theta:\R^+\to\R^+$.
\end{remark}

\subsection{Key definitions and tools}

Now, we move on towards the proof of Theorem \ref{thm:generic_perturbation};   
we shall work with the following spaces.

\begin{definition}\label{def:Pairs} Given a smooth, boundaryless compact manifold $M$ of dimension $N\geq 3$, and fixed a reference (smooth) background metric $\hat g$, we let (for $k\geq k_0, \alpha\in (0,1)$):
\begin{itemize}
\item $\cG^{k,\alpha}(M)$ denote the space of $C^{k,\alpha}$ Riemannian metrics on $M$, endowed with its natural Banach manifold structure;
\item $\cM^{k,\alpha}_n(M)$ denote the space of pairs $(g, \Sigma)$, where $g\in \cG^{k,\alpha}(M)$ and $\Sigma$ is a connected $n$-dimensional minimal submanifold with only strongly isolated singularities ($\MSI$) in $(M, g)$ for $2\leq n<N$; 
 \item $\Pi: \cM^{k,\alpha}_n(M) \to \cG^{k,\alpha}(M)$ denote the projection onto the first factor.
    \end{itemize}
\end{definition}

\begin{remark}\label{rem:TopPairs}
The space $\cM^{k,\alpha}_n(M)$ is endowed with the topology induced by $C^{k,\alpha}$ convergence in the first factor, where it is understood that the norms of all tensors are measured with respect to $g_0$, and multiplicity $1$ $\mbfF$-convergence (following \cite{Pit81}) in the second factor.
\end{remark}

\begin{definition} \label{def:SlowerGrowth}  For $(g, \Sigma)\in \cM^{k,\alpha}_n(M)$, we shall call a normal section $u$ on $\Sigma$ \textbf{of $\tau$-growth} if $u\in \hat{W}^{0,2}_\tau(\Sigma;\mbfV)$. When $\tau$ satisfies (\ref{eq:RangeTau}), we shall also call such sections \textbf{tame}.
    We denote by $\widehat{\Ker}_{\tau} (L_{\Sigma,g})$ the space of all the Jacobi fields of $\tau$-growth on $\Sigma$ in the ambient manifold $(M, g)$.
    \end{definition}

    The following concept allows to quantitatively encode the behaviour of normal sections on approach to a singular point of an $\MSI$.

\begin{definition}\label{def:AsymptRate}
    For $(g,\Sigma)\in\cM^{k,\alpha}_n(M)$ and an $L_{loc}^2$-normal section $v$ defined near $q\in\Sing(\Sigma)$, we introduce the notion of \textbf{asymptotic rate} of $v$ at $q$ as follows
    \begin{align}
        \cA\cR_q(v):= \sup\left\{\gamma\in\R: \lim_{s\searrow 0} \int_{A^g(q, s, 2s)} |v|^2\cdot \rho^{-n-2\gamma}\ d\|\Sigma\| = 0 \right\},  \label{Equ_Asymp Rate}
    \end{align}
    where, as stipulated in the previous section, $\rho(x) = \rho_{\Sigma,g}(x)$. We adopt the convention that $\sup\emptyset = -\infty$. 
    \end{definition}
    
    Heuristically, if $v$ grows like $\rho^\gamma$ near $q$, then $\cA\cR_q(v) = \gamma$. Note that in the setting of hypersurfaces, a similar (yet not identical) notion of asymptotic rate was introduced in \cite{Wang20} and studied in \cite{LW20, LW22}.

    \begin{lemma} \label{Lem_Anal for slower growth func}
        In the setting of the preceding definition, suppose $p, k\geq 2$, $\tau\in \RR$ satisfies \eqref{eq:RangeTau}, and $u\in W^{2,2}_{loc}(\Sigma;\mbfV)$ is a non-zero normal section.
        \begin{enumerate}[label=(\roman*)]
        \item\label{item:AR} If $L_{\Sigma, g} u\in W^{0, p}_{\tau-2}(\Sigma;\mbfV)$ then $\cA\cR_q(u) \in \{-\infty\}\cup \Gamma(\mbfC_q)\cup\{-(n-2)/2\}\cup [\tau, +\infty]$ for every $q\in \Sing(\Sigma)$, where $\mbfC_q:=\mbfC_q(\Sigma)$. 
        \item\label{item:SobAR} If $u\in W^{0,p}_\tau(\Sigma;\mbfV)$ (or, respectively: $\hat{W}^{0,p}_\tau(\Sigma;\mbfV)$), then $\cA\cR_q(u)\geq \tau$ (respectively: $\cA\cR_q(u)\geq 0$) for every $q\in \Sing(\Sigma)$.
        \item\label{item:ARSob} If $\cA\cR_q(u)\geq \tau$ for every $q\in \Sing(\Sigma)$ and $L_{\Sigma, g} u\in W^{k-2, p}_{\tau-2}(\Sigma;\mbfV)$, then $u\in W^{k,p}_{\tau'}(\Sigma;\mbfV)$ for every $\tau'<\tau$.
            \end{enumerate}
    \end{lemma}

    \begin{remark}\label{rem:TriangleForAR}
    If $u_1, u_2$ are $L^2_{loc}$-normal sections of an $\MSI$ $\Sigma$ in $(M,g)$ defined around $q\in\Sing(\Sigma)$ then, by the Minkowski (triangle) inequality, there holds
    \be\label{eq:TriangleForAR}
\cA\cR_q(u_1+u_2)\geq \min_{i=1,2}\left\{ \cA\cR_q(u_1); \cA\cR_q(u_2)\right\}.
    \ee

    \end{remark}

\begin{remark}\label{rem:WeightedHolder}
We note the following weighted version of the H\"older inequality. If $\Sigma =\Sigma^n$ is an $\MSI$ in a Riemannian manifold $(M,g)$ and $1/p_1 + 1/p_2=1$, then
\be\label{eq:WeightedHolder}
u_i\in W^{0,p_i}_{\tau_i}(\Sigma;\mbfV), i=1,2 \ \Rightarrow 
\begin{cases}
u_1 \cdot u_2 \in  L^1(\Sigma;\mbfV) & \ \text{for} \ \tau_1+\tau_2\geq -n \\
u_1 \cdot u_2 \in  W^{0,1}_{\tau}(\Sigma;\mbfV) & \ \text{for} \ \tau_1+\tau_2\geq\tau \\
\end{cases}
\ee
and there hold the estimates (under the corresponding finiteness assumptions):
\be \label{eq:WeightedHolder2}
\|u_1\cdot u_2\|_{L^1(\Sigma)}\leq \|u_1\|_{W^{0,p_1}_{\tau_1}(\Sigma)} \|u_2\|_{W^{0,p_2}_{\tau_2}(\Sigma)}, \ \|u_1\cdot u_2\|_{W^{0,1}_{\tau}(\Sigma)}\leq \|u_1\|_{W^{0,p_1}_{\tau_1}(\Sigma)} \|u_2\|_{W^{0,p_2}_{\tau_2}(\Sigma)}
\ee
  where the unit constant descends from the fact that we had stipulated the convention $\rho\leq 1$.
The proof is done by straightforward reduction to the unweighted (standard) H\"older inequality. (Of course, a fully analogous version of these results holds true for scalar-valued functions $u_1, u_2$ in place of sections of the bundle $\mbfV$).
\end{remark}
    
    \begin{proof}
        To prove \ref{item:AR}, it suffices to show that if $\gamma:= \cA\cR_q(u)\in (-\infty, \tau)$, then $\gamma\in \Gamma(\mbfC_q)\cup\{-(n-2)/2\}$. To see this, suppose for contradiction that $\gamma$ does not belong to such a set, then there exist $\sigma>0$ and $\gamma'<\gamma<\gamma''$ such that \[
            [\gamma'-\sigma, \gamma''+\sigma]\cap (\Gamma(\mbfC_q)\cup\{-(n-2)/2\}\cup [\tau, +\infty)) = \emptyset \,.  \]

        Since the tangent cone of $\Sigma$ at $q$ is a regular cone of multiplicity $1$, there exist $s_0>0$ and $\Lambda>|\gamma'|$, such that:
        \begin{itemize}
            \item after pulling back to $T_qM$ using the exponential map, $\Sigma$ is a graph over the cone $\mbfC:= \mbfC_q \in\cC_{N,n}(\Lambda)$ in $B^g(q, 2s_0)$; we use that graphical parametrization to define $v:=\mbfT^{\Sigma}_{\mbfC,g_{\eucl}}(u)$, that is a section defined on $\mbfC_q$;
            \item for every $s\in (0, s_0]$, conditions \eqref{Equ_App_Close to Cone w Reg Scal bd} in Corollary \ref{Cor_App_Uniform Growth Est for Jac Equ} hold for $\Sigma$ with respect to the scaled metric $s^{-2}g$ (for $\sigma, \Lambda$ as above), \eqref{Equ_App_Uniform away from asymp spectr} also holds for $\gamma'$ in place of $\gamma$, and in addition (by our assumption on $\cA\cR_q(u)$)
            \[
            \limsup_{s\searrow 0}\int_{\mbfA(s, 2s)}|v|^2\cdot |x|^{-n-2\gamma''}\ d\|\mbfC\| = +\infty \,.
        \]
        \end{itemize} 
    Hence, we can construct a monotonically decreasing sequence $\left\{s_j\right\}_{j\geq 1}$ contained in the interval $(0, s_0/2)$, tending to $0$ as $j\to\infty$, such that
    \[
    \int_{\mbfA(s_1, 2s_1)}|v|^2\cdot |x|^{-n-2\gamma''}\ d\|\mbfC\|> 0 
    \]
    and for every $j\geq 1$ and every $s\in [s_j, s_1]$, we have 
        \begin{align*}
            \int_{\mbfA(s_j, 2s_j)}|v|^2\cdot |x|^{-n-2\gamma''}  \ d\|\mbfC\|\geq \int_{\mbfA(s, 2s)}|v|^2\cdot |x|^{-n-2\gamma''}\ d\|\mbfC\|\,.
        \end{align*}
        Equivalently, by letting $v_j(x):= v(s_j x)$, there holds
        \begin{align}
            \int_{\mbfA(1, 2)}|v_j|^2\cdot |x|^{-n-2\gamma''}\ d\|\mbfC\|\geq \int_{\mbfA(s, 2s)}|v_j|^2\cdot |x|^{-n-2\gamma''}\ d\|\mbfC\|, \qquad \forall\, s\in [1, s_j^{-1}s_1]\,.  \label{Equ_Meager_InPf_u_j Decay near infty}
        \end{align}
        Note that this in particular implies, for every $j\geq 1$,
        \begin{align}
            \|v_j\|_{L^2(\mbfA(1,2))}^2 & \geq C(\gamma'')\int_{\mbfA(1, 2)}|v_j|^2\cdot |x|^{-n-2\gamma''}\ d\|\mbfC\| \nonumber \\
            & \geq C(\gamma'')\int_{\mbfA(s_j^{-1}s_1, 2s_j^{-1}s_1)}|v_j|^2\cdot |x|^{-n-2\gamma''}\ d\|\mbfC\| \nonumber \\
            & = C(\gamma'') s_j^{2\gamma''}\int_{\mbfA(s_1, 2s_1)} |v|^2\cdot |x|^{-n-2\gamma''}\ d\|\mbfC\| > 0 \,. 
         \label{Equ_Meager_InPf_|u_j|_A(1,2)>cs^gamma''}
        \end{align}

        On the other hand, it follows from Proposition \ref{Prop_MSE_Section Transp} (see also Remark \ref{Rem_Pre_Jac field equ on MH near cone}) that the section $v$ solves a problem of the form
          \begin{align}
        \nabla^\perp\cdot (\nabla^\perp w + b_0(x)) + \langle A_\mbfC, w\rangle A_\mbfC  + |x|^{-1} b_1(x) = f(x),  
        \end{align}
        where $\nabla^{\perp}=\nabla_{\mbfC, g_{\eucl}}^\perp$ is the Levi-Civita connection of the normal bundle $\mbfV$ to $\mbfC$ in Euclidean $\R^N$, there holds a bound that reads
 \be\label{eq:BoundsCoeffs}
            |b_0|(x) + |b_1|(x) \leq \eps \left(|x|^{-1}|v|(x) + |\nabla^\perp v|(x) \right)
        \ee
        for suitable $\varepsilon>0$, and we have set $f=\mbfT^{\Sigma}_{\mbfC,g_{\eucl}}(L_{\Sigma, g}u)$. Similarly, for any $j\geq 1$ the section $v_j$ shall satisfy an analogous PDE with 
        $b^j_0, b^j_1$ in lieu of $b_0, b_1$ respectively (satisfying a bound of the form \eqref{eq:BoundsCoeffs} for a suitable sequence $\eps_j \to 0$)
        and the datum $f$ replaced by $f_j(x):= s_j^2 f(s_j x)$.

       Hence, by applying Corollary \ref{Cor_App_Uniform Growth Est for Jac Equ} to the data $ \Sigma, M, g_j:= s_j^{-2}g$, $v_j$ as defined above and $\gamma'$, we see that for every $j\geq 1$ and $s\in (0, 1)$, since $\mbfA(K^{-1}s, s)$ is contained in $\mbfA(K^{-i-2}, K^{-i})$ for some non-negative integer $i$, we have
        \begin{align}
            \|v_j\|_{L^2(\mbfA(K^{-1}s, s))} \leq C(\Lambda, \sigma)\left(F_j + \|v_j\|_{L^2(\mbfA(K^{-1}, 1))} \right)s^{n/2 + \gamma'}\,, \label{Equ_Meager_InPf_u_j Decay near 0}
        \end{align}
        where 
        \[
            F_j \leq \left\||x|^{2-\gamma'-n/2}f_j \right\|_{L^2(\mbfB(2))}  
             \leq C(\Lambda,\sigma,\tau)\|L_{\Sigma, g}u\|_{L^2_{\tau-2}(B^g(q,2s_0)\cap\Sigma)} \cdot s_j^\tau\,.
        \]
(Note that the same computation allows, more generally, to prove uniform $L^2_{loc}$ bounds for the data $f_j$.)
        
        Combined with \eqref{Equ_Meager_InPf_|u_j|_A(1,2)>cs^gamma''}, the previous estimate implies,
        \begin{align}
            F_j \leq C(\sigma, \Lambda, \tau, \gamma'', u, L_{\Sigma, g}u, s_1) s_j^{\tau -\gamma''}\|v_j\|_{L^2(\mbfA(1, 2))} \,.  \label{Equ_Meager_InPf_Error Est for F_j}
        \end{align}
        By \eqref{Equ_Meager_InPf_u_j Decay near infty}, \eqref{Equ_Meager_InPf_u_j Decay near 0}, \eqref{Equ_Meager_InPf_Error Est for F_j} and the Rellich Compactness Theorem, $\|v_j\|_{L^2(\mbfA(K^{-1},2))}^{-1}\cdot v_j$ subconverges in $L^2_{loc}$ (and weakly in $W^{1,2}_{loc}$) to some $\hat{v}_\infty$ in $W^{1,2}_{loc}$, which is non-zero by the choice of normalization, and weakly solves $L_{\mbfC} \hat v_\infty = 0$ (see Proposition \ref{Prop_MSE_Section Transp}) as well as the decay estimates, respectively near $0$ and $\infty$:
        \begin{align*}
            \sup_{s>1} \int_{\mbfA(s, 2s)} |\hat v_\infty|^2 \cdot |x|^{-n-2\gamma''}\ d\|\mbfC\| < +\infty\,, \ \text{and} \  &
            \sup_{s\in (0, 1)} \int_{\mbfA(K^{-1}s, s)} |\hat v_\infty|^2\cdot |x|^{-n-2\gamma'}\ d\|\mbfC\| < +\infty\,.
        \end{align*}
        But since in particular $[\gamma', \gamma'']\cap \Gamma(\mbfC) = \emptyset$, there is no such non-zero Jacobi fields on $\mbfC$, which is a contradiction.
        
For part \ref{item:SobAR}, let us first assume $u\in W^{0,p}_\tau(\Sigma;\mbfV)$; let $\tau'<\tau$ (without further restrictions). The case $p=2$ is straightforward; otherwise, if $p>2$, due to the H\"older inequality we have
\[
\int_{A^g(q,s,2s)} \!\!\! |u|^2\cdot \rho^{-n-2\tau'}\ d\|\Sigma\|\leq \!\!\left(\int_{A^g(q,s,2s)} \!\!\! |u|^p\cdot \rho^{-n-p\tau}\ d\|\Sigma\|\right)^{2/p} \!\!\! \left(\int_{A^g(q,s,2s)} \!\!\! \rho^{-n+\frac{2p(\tau-\tau')}{p-2}}\ d\|\Sigma\|\right)^{1-2/p}
\]
therefore, as one lets $s\searrow 0$ both factors on the right-hand side tend to zero and so will the left-hand side.
But then, by the arbitrariness of $\tau'<\tau$ this precisely means that $\cA\cR_q(u)\geq \tau$ for every $q\in \Sing(\Sigma)$. When instead one works in the augmented spaces, namely if $u\in \hat{W}^{0,p}_\tau(\Sigma;\mbfV)$ then $u=u_0+\phi$ for $u_0 \in W^{0,p}_\tau(\Sigma;\mbfV)$ and $\phi\in W_{\text{Tr}}$ (cf. Definition \ref{def:TranslLikeSection}), and thus it suffices to note that, thanks to the Minkowski (triangle) inequality (see Remark \ref{rem:TriangleForAR}) there holds
\[
\cA\cR_q(u)\geq \min_{i=1,2}\left\{ \cA\cR_q(u_0); \cA\cR_q(\phi)\right\}
\]
and $\cA\cR_q(u_0)\geq\tau$ (as just shown) while trivially $\cA\cR_q(\phi)\geq 0$; since by \eqref{eq:RangeTau} in particular $\tau\in (0,1)$ the conclusion follows at once.
        
        To prove \ref{item:ARSob}, first notice that by applying Corollary \ref{Cor_App_Uniform Growth Est for Jac Equ} in the same way as the proof of \ref{item:AR}, for every $\tau'<\tau$ and every $q\in \Sing(\Sigma)$, there exists $s_0>0$ small enough and $K>2$ such that for all $s\in (0, s_0]$,  
        \[
            \|u\|_{W^{0,2}_{\tau'}(A^g(q,K^{-1}s, s)\cap\Sigma)} \leq C(\Sigma, g, \tau, \tau') \left(\|L_{\Sigma, g}u\|_{W^{0,2}_{\tau'-2}(B^g(q,s_0)\cap\Sigma)} + \|u\|_{L^2(A^g(q,K^{-1}s_0, s_0)\cap\Sigma)} \right) \,.
        \]
        On the other hand, the standard interior elliptic estimate (in scaling-invariant form) reads for every $s\in (0, s_0]$, 
        \[
           \|u\|_{W^{k,p}_{\tau'}(A^g(q,K^{-1}s, s)\cap\Sigma)} \leq C(\Sigma, g)\left(\|L_{\Sigma, g} u\|_{W^{k-2,p}_{\tau'-2}(A^g(q,K^{-2}s, 2s)\cap\Sigma)} + \|u\|_{W^{0,2}_{\tau'}(A^g(q,K^{-2}s, 2s)\cap\Sigma)} \right).
           \]
           Thus, if we plug-in the former in the latter we get for all $s\leq s_0/2$ that
           \[
 \|u\|_{W^{k,p}_{\tau'}(A^g(q,K^{-1}s, s)\cap\Sigma)} \leq C(\Sigma, g, \tau, \tau')\left(\|L_{\Sigma, g} u\|_{W^{k-2,p}_{\tau-2}(\Sigma)} + \|u\|_{L^2(A^g(q,K^{-1}s_0, 2s_0)\cap\Sigma)} \right) 
               \]
               that is a bound independent of $s$.
               As a result, if we take any $\tau''<\tau'$, since there clearly holds
        \[
\|u\|_{W^{k,p}_{\tau''}(A^g(q,K^{-1}s, s)\cap\Sigma)} \leq s^{(\tau'-\tau'')} \|u\|_{W^{k,p}_{\tau'}(A^g(q,K^{-1}s, s)\cap\Sigma)} 
        \]
        we can bound
        \[
\|u\|_{W^{k,p}_{\tau''}(A^g(q,K^{-J}s, s)\cap\Sigma)}\leq C(\Sigma,g,\tau,\tau',u)s^{(\tau'-\tau'')}\sum_{j=0}^{J-1}K^{-j(\tau'-\tau'')}.
        \]
        Letting $J\to\infty$ and observing that the uniform bound on the right-hand side (which is a geometric series) implies the finiteness of $\|u\|_{W^{k,p}_{\tau''}(B^g(q,s_0)\cap\Sigma)}$, by arbitrariness of $\tau'', \tau'$ we readily get to the desired conclusion.
\end{proof}
    
      \begin{corollary}\label{cor:AR_Jacobi}
        For $\tau\in \RR$ satisfying (\ref{eq:RangeTau}), $\widehat{\Ker}_{\tau} (L_{\Sigma,g})$ is independent of $\tau$; in particular, if $u\in \widehat{\Ker}_{\tau} (L_{\Sigma, g})$ then for every $q\in\Sing(\Sigma)$ there exists $\mbfx\in T_q M$ such that $\cA\cR_q(u-\mbfx^\perp)\geq 1$.
    \end{corollary}
    \begin{proof}
        Let $u\in \widehat{\Ker}_{\tau}(L_{\Sigma,g})$ for some $\tau$ satisfying \eqref{eq:RangeTau} and let $\tau'\in (\tau,1)$. Recall that this means $u=\phi+u_0$ for $u_0\in W^{k,2}_{\tau}(\Sigma;\mbfV)$ and $\phi$ a translation-like section; note, also, that such $\phi$ is uniquely determined as soon as $\tau>0$ and, in addition, by direct computation there holds  $L_{\Sigma,g}\phi \in W^{0, 2}_{-\varepsilon}(\Sigma;\mbfV)$ for all $\varepsilon>0$ (see Lemma \ref{lem:JacobiTranslFunct}). 

        The latter remark (together with the very definition of Jacobi field, which gives $L_{\Sigma,g} u=0$) implies in particular that $L_{\Sigma,g} u_0\in W^{0,2}_{\tau'-2}(\Sigma;\mbfV)$ and so, by items \ref{item:AR} and \ref{item:SobAR} above, there holds $\cA\cR_q(u_0)\geq \tau'$ for every $q\in\Sing(\Sigma)$. At that stage, we appeal to item \ref{item:ARSob} to conclude that $u_0\in W^{0,2}_{\tau''}(\Sigma;\mbfV)$ for any $\tau''<\tau'$, which -- by the arbitrariness of $\tau'$ -- completes the proof.
        \end{proof}

        \begin{remark}\label{rem:IndipDimCoker}
      By combining this corollary with Theorem \ref{thm:Count} it follows that $\dim\widehat{\Coker}_{\tau}(L_{\Sigma,g})$ is also independent of $\tau$ as this varies in the chosen range, i.\,e. for $\tau$ satisfying \eqref{eq:RangeTau}.
        \end{remark}

\subsection{Canonical neighborhoods} We shall begin here by introducing some notions and objects that will play a prominent role in this section. For $(g, \Sigma)\in \cM^{k,\alpha}_n(M)$, we set
\[
\injrad(g, \Sigma):= \min\{ \injrad(M, g), \min\{\dist_g(p, p')/2: p\neq p'\in \Sing(\Sigma)\}\}\,,
\]
and
\be\label{eq:RangeLambda}
\cR^{k,\alpha}_n(g,\Sigma):=\left\{\Lambda>0 \st \|g\|_{C^{k, \alpha}} \leq \Lambda, \quad \mfr_{\Sigma, g} \geq \Lambda^{-1}\rho_{\Sigma,g} \right\}.
\ee
Hence, we also define
\be\label{eq:DomainRadiusMap}
\cT^{k,\alpha}_n(M):=\left\{(g,\Sigma,\Lambda)\in\cM^{k,\alpha}_n(M)\times\R_{+} \st \Lambda\in \cR^{k,\alpha}_n(g,\Sigma)\right\}.
\ee

    \begin{definition} \label{Def_injrad, canonical neighb}
        For $(g, \Sigma)\in \cM^{k,\alpha}_n(M)$, and $\Lambda\in \cR^{k,\alpha}_n(g,\Sigma)$, we define $\cL^{k,\alpha}(g, \Sigma; \Lambda, \delta)$ to be the space of all pairs $(g', \Sigma')\in \cM^{k,\alpha}_n(M)$ such that:
        \begin{itemize}
            \item $\|g'\|_{C^{k,\alpha}}\leq \Lambda$, and $\|g-g'\|_{C^k}\leq \delta$;
            \item $\Sigma'$ is a connected $\MSI$ in $(M, g')$ satisfying
            \begin{equation}
            \mfr_{\Sigma', g'} \geq \Lambda^{-1}\rho_{\Sigma',g'}, \quad \text{and} \quad
                \mbfF(|\Sigma|_g, |\Sigma'|_{g'})\leq \delta;\label{Equ_Canonical Nbhd_|A_Sigma|<r^(-1)} 
            \end{equation}
            \item there exists a bijection $\Sing(\Sigma)\to \Sing(\Sigma')$, such that for every $p\in \Sing(\Sigma)$ its image $p'\in\Sing(\Sigma')$ satisfies
            \begin{equation} \label{Equ_Canonical dist_density}
                \dist_g(p, p')\leq \injrad(g,\Sigma)/2, \quad \theta_{|\Sigma'|}(p') = \theta_{|\Sigma|}(p).
            \end{equation}
        \end{itemize}

        We call such $\cL^{k,\alpha}(g, \Sigma; \Lambda, \delta)$ a \textbf{canonical (pseudo-)neighborhood} of $(g, \Sigma)$ in $\cM^{k,\alpha}_n(M)$; it is understood that its topology is the one induced from $\cM^{k,\alpha}_n(M)$.
    \end{definition}

    Clearly, when $g, \Sigma$ are fixed, $\cL^{k,\alpha}(g, \Sigma;\Lambda, \delta)$ is monotonically increasing (with respect to set-theoretic inclusion) in both $\Lambda$ and $\delta$. Also note that in general, $\cL^{k,\alpha}(g, \Sigma; \Lambda, \delta)$ is not a genuine topological neighborhood of $(g, \Sigma)$ in $\cM^{k,\alpha}_n(M)$, as it is not necessarily an open subset (nor will it contain an open subset). Indeed, there might be $(g', \Sigma')\in \cM^{k,\alpha}_n(M)$ arbitrarily close to $(g, \Sigma)$ but not contained in $\cL^{k,\alpha}(g, \Sigma; \Lambda, \delta)$, which happens, for example, when multiple singular points collapse to a single one. 
    
    \begin{remark}\label{rem:DensityConesClosure}
    Note that for every $(g', \Sigma')\in \cL^{k,\alpha}(g, \Sigma; \Lambda, \delta)$, \eqref{Equ_Canonical Nbhd_|A_Sigma|<r^(-1)} implies that every tangent cone of $\Sigma'$ at any of its singular points also belongs to $\cC_{N,n}(\Lambda)$.
    \end{remark}

    One of the main reasons behind this specific definition of canonical neighborhood lies in the associated compactness property, which we now state.

    \begin{lemma}[Compactness of canonical neighborhoods] \label{Lem_Cptness for scL^k}
        For $(g, \Sigma)\in \cM^{k,\alpha}_n(M)$ and $\Lambda\in \cR^{k,\alpha}_n(g,\Sigma)$, there exists $\varsigma_0(g, \Sigma, \Lambda)\in (0, 1)$ with the following property. 
        
        For every $\delta \in (0, \varsigma_0)$, $\cL^{k,\alpha}(g, \Sigma; \Lambda, \delta)$ is compact under $C^k$-norm in the first factor and $\mbfF$-distance in the second factor. Moreover, there holds $C^{k,\beta}_{loc}$ convergence in the second factor for any $\beta\in (0,1)$.
    \end{lemma}
    \begin{proof}
        The claims easily follows from the Arzel\`a-Ascoli compactness theorem, the definition of canonical neighborhood and the implications of assuming a uniform bound on the regularity scale of all $\MSI$ in $\cL^{k,\alpha}(g, \Sigma; \Lambda, \delta)$ (see in particular Remark \ref{rem:RegCurvBound}). 
    \end{proof}

    We explicitly note that this compactness result implies, among other things, that the scale for the radius functions of any $\Sigma'$ in $(M,g')$ can be chosen uniformly (bounded from below) as one varies $(g',\Sigma')\in \cL^{k,\alpha}(g, \Sigma; \Lambda, \delta)$.

    A second, important reason to consider such canonical neighborhoods is that the following local Sard-Smale-type theorem holds.

    \begin{theorem} \label{Thm_Local Sard Thm}
        Let $(g, \Sigma)\in \cM^{k,\alpha}_n(M)$ with $\widehat{\ind}_{\tau}(L_{\Sigma,g})<0$, $\tau \in \left(\sup_{p\in \Sing(\Sigma)} \gamma_-(\mbfC_p), 1 \right)$ and $\Lambda\in \cR^{k,\alpha}_n(g,\Sigma)$. Then there exists $\varsigma=\varsigma(g, \Sigma, \Lambda)>0$ such that for every $\delta \in (0, \varsigma]$,
        \begin{align*}
            \cG^{k,\alpha}(g, \Sigma; \Lambda, \delta) := \cG^{k,\alpha}(M) \setminus \Pi\left( \cL^{k,\alpha}(g,\Sigma; \Lambda, \delta)\right)
        \end{align*}
        is an open, dense subset of $\cG^{k,\alpha}(M)$. 
    \end{theorem}

     To derive a global Sard-Smale type theorem on $\cM^{k,\alpha}_n(M)$, it thus suffices to show that it can be covered by countably many canonical neighborhoods.
            
    \begin{theorem} \label{Thm_Countable Decomp}
        Let $\kappa: \cT^{k,\alpha}_n(M)\to \RR_+$ be a positive function, not necessarily continuous. Then there exists a countable number of triples $\{(g_j, \Sigma_j; \Lambda_j)\}\in \cT^{k,\alpha}_n(M)$ such that \[
            \cM^{k,\alpha}_n(M) = \bigcup_{j\geq 1}\cL(g_j, \Sigma_j; \Lambda_j, \kappa_j),   
            \]
        where $\kappa_j:= \kappa(g_j, \Sigma_j; \Lambda_j)$.
    \end{theorem}
    Theorem \ref{Thm_Countable Decomp} will be proved in Section \ref{Sec:Count_Decomp}, based on decomposition arguments inspired by the work \cite{Edelen2021}.

    Heuristically, recalling one of the key points of \cite{Whi91}, $\widehat{\ind}_{\tau}(L_{\Sigma,g})<0$ (which implies the non-surjectivity of the Jacobi operator of $\Sigma$ in $(M,g)$) suggests that the ``tangent map'' of $\Pi$ restricted to $\cM^{k,\alpha}_n$ is itself not surjective at $(g,\Sigma)$. Hence one may think of $\cG(g, \Sigma; \Lambda, \kappa_0)$ as the set of regular values of $\Pi|_{\cL(g, \Sigma; \Lambda, \kappa_0)}$. When we only focus on smooth minimal hypersurfaces, such a Sard-Smale Theorem was proved in \cites{Whi91} by showing that --- working with spaces of metrics having finite, in fact $C^k$, degree of regularity --- $\cM^k$ is a Banach manifold, and $\Pi$ is a Fredholm map with Fredholm index $0$. The discussion about how to derive the corresponding result for smooth metrics, which task turns out to be not entirely trivial, was then later presented in \cite{Whi17}; see also \cite[Section 7]{AmbCarSha18} for a thorough study of this aspect, and the application to the free-boundary counterpart of such a result.

    Here, however, it is unreasonable to expect any Banach manifold structure on $\cM^{k,\alpha}_n(M)$; the way we prove Theorem \ref{Thm_Local Sard Thm} is rather by unwrapping the proof of the infinite dimensional Sard-Smale Theorem by hand crucially employing the canonical neighborhoods defined above. This poses various technical challenges, which we shall describe at due course.

\subsection{Induced Jacobi fields for pairs}\label{subs:JacobiPairs}

Given a pair $(g, \Sigma) \in \cM^{k, \alpha}_n(M)$, for $\Lambda\in\cR^{k,\alpha}_n(g,\Sigma)$ and sufficiently small $\delta$, we can define a \emph{semi-metric} on $\cL^{k,\alpha}(g, \Sigma; \Lambda, \delta)$ by 
    \[\begin{aligned}
        & \mbfD^\cL[(g_1,\Sigma_1), (g_2, \Sigma_2)] :=\|g_1-g_2\|_{C^{k,\alpha}(M)} \\
        & + \|\mbfG^{\Sigma_2}_{\Sigma_1, g_1}\|_{L^2\left(\Sigma_1 \setminus \cup_{x\in \Sing(\Sigma_1)}B^{g_1}(x, C^\cL \bfr^{\Sigma_2}_{\Sigma_1, g_1}(x))\right)}
         + \|\Sigma_2\|_{g_2}(\cup_{x\in \Sing(\Sigma_1)}B^{g_2}(x, 2C^\cL \bfr^{\Sigma_2}_{\Sigma_1, g_1}(x)))\\
        & + \|\mbfG^{\Sigma_1}_{\Sigma_2, g_2}\|_{L^2\left(\Sigma_2 \setminus \cup_{x\in \Sing(\Sigma_2)}B^{g_2}(x, C^\cL \bfr^{\Sigma_1}_{\Sigma_2, g_2}(x))\right)}
        + \|\Sigma_1\|_{g_1}(\cup_{x\in \Sing(\Sigma_2)}B^{g_1}(x, 2C^\cL \bfr^{\Sigma_1}_{\Sigma_2, g_2}(x))) \,.
    \end{aligned}
    \]
    Here, the constant $C^\cL$ is defined as follows. Since $\Sigma$ is an $\MSI$ in $(M, g)$, by Definition \ref{def:MSI} and Remark \ref{rem:MSInotation}, there are finitely many singular points $\Sing(\Sigma) = \{p_i\}^Q_{i = 1}$ and at each singular point $p_i$ the submanifold $\Sigma$ has a unique tangent cone $\mbfC_i \coloneqq \mbfC_{p_i}(\Sigma)$. Using the fact that each set $\Gamma(\mathbf{C}_i)$ is discrete, we can choose a small enough $\sigma \coloneqq \sigma(g,\Sigma) \in (0, 1)$ such that for every $p_i \in \Sing(\Sigma)$,
    \begin{equation}\label{eq:choice_sigma}
      \dist_\R(1 - 2\sigma, \Gamma(\mathbf{C}_i) \cup \{- (n - 2)/ 2\}) \geq 2\sigma\,.
    \end{equation}
   At that stage, we set 
    \begin{equation}\label{eqn:C^cL}
        C^\cL := C_2(1 - 2\sigma, \sigma, \Lambda)>0
    \end{equation} as afforded by Corollary~\ref{Cor_L^2 Growth Est}; clearly, we can assume without loss of generality that $C^\cL\geq 1$ which ensures the well-posedness of the definition above. It is also worth stressing that $C^\cL$ depends on the choice of $\sigma$, but we will fix one admissible $\sigma$ for each pair $(g, \Sigma)$ from now on (we will in fact require equality in \eqref{eq:choice_sigma}).
    
    Recall that (consistently with Definition \ref{def:Pairs}) norms of tensors are taken with respect to the (smooth) reference metric $\hat{g}$ which we have fixed earlier in the discussion. The reason we call $\mbfD^\cL$ a \emph{semi-metric} rather than a \emph{metric} is that it satisfies all the metric axioms except, possibly, for the triangle inequality.

    \begin{remark}\label{rem:AllardMetrization}
    By Allard's regularity theorem~\cite{All72}, a sequence $(g_j, \Sigma_j)$ converges to $(g_\infty, \Sigma_\infty)$ in $\cL^{k,\alpha}$ if and only if $\mbfD^\cL[(g_j, \Sigma_j), (g_\infty, \Sigma_\infty)]\to 0$ when $j\to \infty$. If that is the case (as in Lemma \ref{Lem_Cptness for scL^k}) then there holds $C^{k,\beta}_{loc}$ convergence of $\Sigma_j$ to $\Sigma_{\infty}$ for any $\beta\in (0,1)$, understood as graphical convergence with multiplicity one.
    \end{remark}
        
    The following statement builds upon the fine analysis carried out in Appendix \ref{app:GraphicalPar}, specifically appealing to Corollary~\ref{Cor_L^2 Growth Est} and Corollary \ref{Cor_L^2 Growth est for pair} therein.

    \

    \begin{lemma} \label{Lem_Tangent vectors to scL^k}
        Given $(g, \Sigma) \in \mathcal{M}^{k, \alpha}_n(M)$, $r\in (0, \injrad(g, \Sigma))$, $\Lambda\in\cR^{k,\alpha}_n(g,\Sigma)$, and $\delta > 0$, suppose that:
        \begin{enumerate}[label=(\roman*)]
        \item\label{item:SubspaceData}{$\cF$ be a finite-dimensional subspace of $C^{k,\alpha}(M)$, consisting of functions supported in the complement of the union of the metric balls $B^g(p,r)$ for $p\in\Sing(\Sigma)$;}
        \item\label{item:ConvergencePairs}{$\{(\bar{g}_j, \bar{\Sigma}_j)\}_j$, $\{(g_j^0, \Sigma_j^0)\}_j$ and $\{(g_j^1, \Sigma_j^1)\}_j$ be three sequences in $\cL^{k,\alpha}(g, \Sigma; \Lambda, \delta)$ such that $(\bar{g}_j, \bar{\Sigma}_j)\to (g, \Sigma)$ as $j \to \infty$, for every $j$, $(g_j^0, \Sigma_j^0)$ and $(g_j^1, \Sigma_j^1)$ are distinct pairs, and both sequences also satisfy $(g^i_j, \Sigma^i_j) \to (g, \Sigma)$ as $j\to \infty$;}
        \item\label{item:ConfFactors}{for $i=0,1$ and $j\geq 1$, $g_j^i = (1+f_j^i)\bar{g}_j$ for $f^i_j\in\cF$,  and either $f_j^0\equiv f_j^1$ (in which case we set $f_\infty \equiv 0$), 
                or 
                \[
                \frac{f_j^1 - f_j^0}{\|f_j^1 - f_j^0\|_{C^{k,\alpha}(M)}}\to f_\infty \ \text{ in } C^{k,\alpha}(M)
                \] with the limit function satisfying the condition that $\nabla^\perp_{\Sigma, g}(f_\infty):=(\nabla f_{\infty})^{\perp_{\Sigma,g}}$ is not identically $0$.}
 \end{enumerate}
        For all $j\geq 0$ and $i \in \{0, 1\}$, we set 
		\[
			d_j:= \mbfD^\cL [(g^0_j, \Sigma_j^0), (g^1_j, \Sigma_j^1)] >0\,, \quad u^{(i)}_j:= \mbfG^{\Sigma_j^i}_{\bar{\Sigma}_j, \bar{g}_j}\in  L^\infty(\bar{\Sigma}_j;\mbfV_{j})\,, \quad v^{(i)}_j:= \mbfT^{\bar{\Sigma}_j}_{\Sigma,g}(u^{(i)}_j)\in  L^\infty(\Sigma;\mbfV)
		\]
  where $\mbfV_j$ (respectively: $\mbfV$) denotes the normal bundle to $\bar{\Sigma}_j$ with respect to the metric $\bar{g}_j$ (resp.: to $\Sigma$ with respect to the metric $g$).
		Then we have, after passing to a subsequence, as $j \to \infty$,
        \begin{align*}
            d_j^{-1}\cdot(v^{(1)}_j - v^{(0)}_j) \to \hat{u}_\infty  \ \text{ in }L^2_{loc}(\Sigma;\mbfV) \, \qquad d_j^{-1}\cdot (f_j^1 - f_j^0) \to \hat{f}_\infty \ \text{ in }C^{k,\alpha}(M),
        \end{align*}
        where $\mbfV$ is the normal bundle to $\Sigma$ with respect to the metric $g$, $\hat{f}_\infty$ is a non-negative multiple of $f_\infty$ and $\hat{u}_\infty\in W^{k,2}_{loc}(\Sigma;\mbfV)$ is a non-zero solution of 
		\[
            L_{\Sigma, g} \hat{u}_\infty = \frac{n}{2}\nabla^\perp_{\Sigma, g}(\hat{f}_\infty).  
		\]
        Moreover, $\hat{u}_\infty$ is a tame section (in the sense of Definition \ref{def:SlowerGrowth}).
    \end{lemma}
    \begin{proof}
        For $(g, \Sigma)$ as in the statement, we choose $\sigma$ as described above (see \eqref{eq:choice_sigma}) and set $\gamma=1-2\sigma$.
       Since $\bar \Sigma_j, \Sigma^0_j, \Sigma^1_j \to \Sigma$ in $\cL^{k, \alpha}$, for all sufficiently large $j$, for any (necessarily non-trivial) tangent cone $\mbfC$ at points belonging to $\Sing(\bar{\Sigma}_j), \Sing(\Sigma^0_j)$ or $\Sing(\Sigma^1_j)$, we have
        \begin{equation}\label{eqn:sigma_choice}
            \dist_\R(1 - 2\sigma, \Gamma(\mathbf{C}) \cup \{- (n - 2)/ 2\}) \geq \sigma\,.
        \end{equation}
        
        Note, further, that within the set $\cup_{p\in \Sing(\Sigma)} B^g(p, r)$, for all $j\geq 1$ there holds
        \[
            \bar g_j = g^0_j = g^1_j
        \]
        Therefore, we can apply Corollary~\ref{Cor_L^2 Growth Est} (keeping in mind the $C^{k,\beta}_{loc}$ convergence of both $\Sigma^0_j$ and $\Sigma^1_j$ to $\Sigma$, cf. Remark \ref{rem:AllardMetrization}). Thus there exist $C_0 = C_0(\sigma, \Lambda) > 0$, $\tilde C_0 = \tilde C_0(\sigma, \Lambda) > 2$, and $r_0 = r_0(g, \Sigma; \Lambda, \sigma) \in (0, \min(r / (2\tilde C_0 + 1), 1/2))$ such that, for every $\bar{\mathbf{x}}_j \in \Sing(\bar \Sigma_j)$, set $\mathbf{x}^0_j$ and $\mathbf{x}^1_j$ the unique element in $\Sing(\Sigma^0_j) \cap B^{\bar g_j}(\bar{\mathbf{x}}_j, r_0)$ and $\Sing(\Sigma^1_j) \cap B^{\bar g_j}(\bar{\mathbf{x}}_j, r_0)$ respectively,  we shall have for $i = 0, 1$
        \begin{equation}\label{eqn:estimate_|x^1_j - x^0_j|_r^i_j}
           \frac{1}{2} |\mbfx^1_j-\mbfx^0_j| \leq \bfr^{(i)}_j(\mbfx^{1-i}_j) \coloneqq \bfr^{\Sigma^i_j}_{\Sigma^{1-i}_j, \bar g_j}(\mbfx^{1-i}_j) \leq C_0 \|\mbfG^{\Sigma^i_j}_{\Sigma^{1-i}_j, \bar g_j}\|_{L^2(\Sigma^{1-i}_j\cap A^{\bar g_j}(\mbfx^{1-i}_j, 2 r_0, \tilde C_0 r_0))},
        \end{equation}
        \begin{equation}\label{eqn:estimate_mbfG^1_0}
            \|\mbfG^{\Sigma_j^i}_{\Sigma^{1-i}_j, \bar g_j}\|_{C^0 (\Sigma^{1-i}_j\cap A^{\bar g_j}(\mbfx^{1-i}_j, C^\cL\bfr^{(i)}_j(\mbfx^{1-i}_j), 2r_0))} \leq C_0 \|\mbfG^{\Sigma_j^i}_{\Sigma^{1-i}_j, \bar g_j}\|_{L^2(\Sigma^{1-i}_j\cap A^{\bar g_j}(\mbfx^{1-i}_j, 2 r_0, \tilde C_0 r_0))}\,.
        \end{equation}

        In addition, we note that, letting $j\to\infty$ in the equation
        $g_j^i = (1+f_j^i)\bar{g}_j$ we get (by virtue of the convergence assumptions \ref{item:ConvergencePairs}, based on the specification in \ref{item:ConfFactors}) in fact 
        \be\label{eq:f^i_j}
            \|f^i_j\|_{C^{k,\alpha}(M)} \to 0, \ \ i=0,1 \ \ (j\to\infty).
        \ee
        
        It follows from the minimal submanifold system (see Proposition \ref{Prop_MSE_Main}, cf. \cite[Theorem~B.1]{LW22}) that, for sufficiently large $j$, $w_j \coloneqq u^{(1)}_j - u^{(0)}_j$ satisfies the equation
        \begin{equation}\label{eqn:w_j_f_j}
            L_{\bar{\Sigma}_j, \bar g_j} w_j - \frac{n}{2} \nabla^\perp_{\bar{\Sigma}_j, \bar g_j}(f^1_j - f^0_j) + \nabla^\perp_{\bar{\Sigma}_j, \bar g_j} \cdot \tilde b_0 + \frac{\tilde b_1}{\rho_j} = 0
        \end{equation}
        on $\bar \Sigma_j \setminus \cup_{\mbfx_j\in \Sing(\bar \Sigma_j)} B^{\bar g_j}(\mbfx_j,\max_{i=0,1}C^{\cL}\bfr^{(i)}_j(\mbfx^{1-i}_j))$. Here, $\rho_j:=\rho_{\bar{\Sigma}_j,\bar{g}_j}$, and $\tilde b_0$ and $\tilde b_1$ are error terms satisfying pointwise estimates
        \begin{equation}\label{eqn:Error_estimates}
            \begin{aligned}
            |\tilde b_0|, |\tilde b_1| \leq C \cdot \left(\sum_{i = 0,1} \frac{|u^{(i)}_j|}{\rho_j} + |\nabla^{\perp} u^{(i)}_j| + \|f^i_j\|_{C^2}\right)
                \cdot \left(\frac{|w_j|}{\rho_j} + |\nabla^{\perp} w_j| + \|f^1_j - f^0_j\|_{C^2}\right)\,
            \end{aligned}
        \end{equation}
        where, for notational convenience, we have written $\nabla^{\perp}$ in lieu of $\nabla^{\perp}_{\bar{\Sigma}_j,\bar{g}_j}$, and $C=C(g,\Sigma,\Lambda,\delta)$ is a constant independent of $j$.

        \textbf{Claim 1.} For any subsequence of $\left\{f^1_j-f^0_j\right\}_{j\geq 1}$, we can find a further subsequence such that
        \[
            \frac{(f^1_j - f^0_j)}{d_j} \to \hat{f}_{\infty}= c \cdot  f_\infty \in \cF\,,
        \]
        for some $c \geq 0$. (The value of $c$ could depend on the chosen subsequence.)
        \begin{proof}
            If $f^1_j = f^0_j$, then this is trivially true with $c = 0$.
    
            Otherwise, by definition,
            \[
                d_j\geq  \|g^1_j - g^0_j\|_{C^{k,\alpha}(M)}  \geq C\|f^1_j - f^0_j\|_{C^{k,\alpha}(M)}\,,
            \]
            and the existence of the limit (recalling the finite-dimensionality of $\cF$) follows immediately.
        \end{proof}

        \textbf{Claim 2.} Let $d'_j := \|w_j\|_{L^2(\bar \Sigma_j \setminus B^{\bar g_j}(\Sing(\bar \Sigma_j), r_0))}$. Then 
        \[
            \liminf_{j \to \infty} \frac{d_j}{d'_j} \in (0, +\infty)\,.
        \]
        \begin{proof}
            Without loss generality, and without renaming, we can restrict to a subsequence whose limit is equal to $\liminf_{j \to \infty} \frac{d_j}{d'_j}$. Suppose for the sake of contradiction that either $\lim_{j \to \infty} \frac{d_j}{d'_j} = 0$ or $+\infty$. Let us discuss such two cases separately.
			
            \medskip
            \paragraph*{Case 1} $\boxed{\lim_{j \to \infty} \frac{d_j}{d'_j} = 0}.$ 
            
            Note that, by definition of $d_j$, for any $j$ large enough there holds
            \[
                d_j \equiv \mbfD^\cL [(g^0_j, \Sigma_j^0), (g^1_j, \Sigma_j^1)] 
                \geq \|\mbfG^{\Sigma^1_j}_{\Sigma^0_j, g^0_j}\|_{L^2(\Sigma^0_j \setminus B^{g^0_j}(\Sing( \Sigma^0_j), r_0 / 3))}.
            \]

            Next, by the convergence assumptions in item \ref{item:ConvergencePairs}, for any $\delta>0$ there exists $j_0=j_0(\delta)$ such that  $j\geq j_0(\delta)$ implies that outside of $B^{g^0_j}(\Sing(\Sigma^0_j), r_0 / 8)$, the submanifolds $\Sigma^0_j$, $\Sigma^1_j$ and $\bar \Sigma_j$ are all $\delta$-$C^3$ graphs over $\Sigma$. In addition, for all sufficiently large $j$, there holds the inclusion $B^{g^0_j}(\Sing( \Sigma^0_j), 2 r_0 / 3)  \subset B^{\bar g_j}(\Sing(\bar \Sigma_j), r_0)$. It then follows from \eqref{Equ_GraphParam_|u^2_0-u^1_0| approx |u^2_1|} in Lemma \ref{Lem_App_GraphParam_Compare diff graph func} (employing a suitable covering argument) that for all sufficiently large $j$ there exists a constant $C> 0$ (possibly depending on $g,\Sigma,\Lambda,\delta$ but independent of $j$) such that
            \[\begin{aligned}
                d'_j &= \|\mbfG^{\Sigma_j^1}_{\bar{\Sigma}_j, \bar{g}_j} - \mbfG^{\Sigma_j^0}_{\bar{\Sigma}_j, \bar{g}_j}\|_{L^2(\bar \Sigma_j \setminus B^{\bar g_j}(\Sing(\bar \Sigma_j), r_0))}\\
                    &\leq \|\mbfG^{\Sigma_j^1}_{\bar{\Sigma}_j, \bar{g}_j} - \mbfG^{\Sigma_j^0}_{\bar{\Sigma}_j, \bar{g}_j}\|_{L^2(\bar \Sigma_j \setminus B^{g^0_j}(\Sing(\Sigma^0_j), 2 r_0 / 3))}\\
                    &\leq C \|\mbfG^{\Sigma^1_j}_{\Sigma^0_j, g^0_j}\|_{L^2(\Sigma^0_j \setminus B^{g^0_j}(\Sing(\Sigma^0_j), r_0 / 3))}\\
                    &\leq C d_j\,.
            \end{aligned}\] 

            Therefore, $\frac{d_j}{d'_j} \geq \frac{1}{C}$. This contradicts our assumption that $\lim_{j \to \infty} \frac{d_j}{d'_j} = 0$.
			
		\medskip
            \paragraph*{Case 2} $\boxed{\lim_{j \to \infty} \frac{d_j}{d'_j} = +\infty.}$ In this case, we can divide the PDE~\eqref{eqn:w_j_f_j} by $d_j$, and apply Corollary \ref{Cor_MSE_converg preserve equ} to \eqref{eqn:w_j_f_j}, \eqref{eqn:Error_estimates} to derive,
            \[
                \frac{v^{(1)}_j-v^{(0)}_j}{d_j}=\frac{\mbfT^{\bar{\Sigma}_j}_{\Sigma,g} (w_j)}{d_j} \to \hat u_\infty\ \text{ in } L^2_{loc}(\Sigma;\mbfV), \qquad \frac{(f^1_j - f^0_j)}{d_j} \to \hat f_\infty \text{ in }C^{k,\alpha}(M)\,,
            \]
            for some weak solution $\hat u_\infty\in W^{1,2}_{loc}(\Sigma; \mbfV)$ to
            \begin{equation}\label{eqn:ellip_PDE2}
                -L_{\Sigma, g} \hat u_\infty + \frac{n}{2}\nabla^\perp_{\Sigma, g}(\hat f_\infty) = 0\,.
            \end{equation}
            As a result, we have 
            \[
                \|\hat u_\infty\|_{L^2(\Sigma \setminus B^{\bar g}(\Sing(\Sigma), 2r_0))} \leq \lim_{j \to \infty} \frac{\|w_j\|_{L^2(\bar \Sigma_j \setminus B^{\bar g_j}(\Sing(\bar \Sigma_j), r_0))}}{d_j} = \lim_{j \to \infty} \frac{d'_j}{d_j} = 0\,.
            \] 
          Therefore, $\hat u_\infty \equiv 0$ in $\Sigma \setminus B^g(\Sing(\Sigma), 2r_0)$.
          Recall that $\hat f_\infty \equiv 0$ in $B^g(\Sing(\Sigma), 3r_0)$. By the unique continuation property of solutions to~\eqref{eqn:ellip_PDE2} (for which \cite{Aro57} would suffice), we know that $\hat u_\infty \equiv 0$ on $\Sigma$. Then, since $\nabla^\perp_{\Sigma, g}(f_\infty)$ is not identically $0$ and $\hat f_\infty = c f_\infty$, we then have $c = 0$, i.\,e., $\hat f_\infty \equiv 0$. Hence, as one lets $j\to\infty$ there holds in fact
            \[
                \frac{\|g^1_j - g^0_j\|_{C^{k,\alpha}(M)}}{d_j} \to 0\,.
            \]
                Again by \eqref{Equ_GraphParam_|u^2_0-u^1_0| approx |u^2_1|} in Lemma \ref{Lem_App_GraphParam_Compare diff graph func}, there exists a constant $C>0$ independent of $j$ such that
            \begin{align*}
                &\quad \frac{\|\mbfG^{\Sigma^1_j}_{\Sigma^0_j, g^0_j}\|_{L^2(\Sigma^0_j \setminus B^{g^0_j}(\Sing(\Sigma^0_j), 2r_0))} + \|\mbfG^{\Sigma^0_j}_{\Sigma^1_j, g^1_j}\|_{L^2(\Sigma^1_j \setminus B^{g^1_j}(\Sing(\Sigma^1_j), 2r_0))}}{d_j}\\
                &\leq C\cdot \frac{\|w_j\|_{L^2(\bar \Sigma_j \setminus B^{\bar g_j}(\Sing(\bar \Sigma_j), r_0))}}{d_j} \to 0\,,
            \end{align*}
            as $j \to \infty$. By the pointwise estimates for $\mbfG^{\Sigma^1_j}_{\Sigma^0_j, g^0_j}$ and $\mbfG^{\Sigma^0_j}_{\Sigma^1_j, g^1_j}$ collected in \eqref{eqn:estimate_mbfG^1_0}, there exists a possibly larger constant (still independent of $j$ such that 
            \begin{align*}
                &\quad \frac{\|\mbfG^{\Sigma^1_j}_{\Sigma^0_j, g^0_j}\|_{L^2(\bigcup_{\mbfx \in \Sing(\Sigma^0_j)} \Sigma^0_j \cap A^{\bar g_j}(\mbfx, C^\cL \bfr^{(1)}_j(\mbfx), 2 r_0))} + \|\mbfG^{\Sigma^0_j}_{\Sigma^1_j, g^1_j}\|_{L^2(\bigcup_{\mbfx \in \Sing(\Sigma^1_j)} \Sigma^1_j\cap A^{\bar g_j}(\mbfx, C^\cL \bfr^{(0)}_j(\mbfx), 2 r_0))}}{d_j}\\
                &\leq C \cdot \frac{\|\mbfG^{\Sigma^1_j}_{\Sigma^0_j, g^0_j}\|_{L^2(\bigcup_{\mbfx \in \Sing(\Sigma^0_j)} \Sigma^0_j \cap A^{\bar g_j}(\mbfx, 2 r_0, \tilde C_0 r_0))} + \|\mbfG^{\Sigma^0_j}_{\Sigma^1_j, g^1_j}\|_{L^2(\bigcup_{\mbfx \in \Sing(\Sigma^1_j)} \Sigma^1_j \cap A^{\bar g_j}(\mbfx, 2 r_0, \tilde C_0 r_0))}}{d_j}\\
                &\leq C \cdot \frac{\|\mbfG^{\Sigma^1_j}_{\Sigma^0_j, g^0_j}\|_{L^2(\Sigma^0_j \setminus B^{g^0_j}(\Sing(\Sigma^0_j), 2r_0))} + \|\mbfG^{\Sigma^0_j}_{\Sigma^1_j, g^1_j}\|_{L^2(\Sigma^1_j \setminus B^{g^1_j}(\Sing(\Sigma^1_j), 2r_0))}}{d_j} \to 0\,,
            \end{align*}
            as $j \to \infty$, where, in the last inequality, we have exploited the fact that for sufficiently large $j$, $\bar g_j = g^0_j = g^1_j$ in $B^{g^0_j}(\Sing(\Sigma^0_j), 3r_0) \cup B^{g^1_j}(\Sing(\Sigma^1_j), 3r_0)$. Hence, we can conclude that
            \begin{equation*}
                \frac{\|\mbfG^{\Sigma^1_j}_{\Sigma^0_j, g^0_j}\|_{L^2(\Sigma^0_j\setminus \bigcup_{\mbfx \in \Sing(\Sigma^0_j)}B^{g^0_j}(\mbfx, C^\cL \bfr^{(1)}_j(\mbfx)))} + \|\mbfG^{\Sigma^0_j}_{\Sigma^1_j, g^1_j}\|_{L^2(\Sigma^1_j\setminus \bigcup_{\mbfx \in \Sing(\Sigma^1_j)}B^{g^1_j}(\mbfx, C^\cL \bfr^{(0)}_j(\mbfx)))}}{d_j}\to 0.
            \end{equation*}
               Moreover, since both $\Sigma^0_j$ and $\Sigma^1_j$ converge to $\Sigma$ in $\cL^{k, \alpha}$, 
            by~\eqref{eqn:estimate_|x^1_j - x^0_j|_r^i_j}, 
            \[
                \bfr^{\Sigma^0_j}_{\Sigma^1_j, g^1_j}, \quad \bfr^{\Sigma^1_j}_{\Sigma^0_j, g^0_j} \to 0 \quad \text{uniformly, respectively on}\, \Sing(\Sigma^1_j), \,\Sing(\Sigma^0_j)\,.
            \] 
            It follows from Corollary~\ref{Cor_Converg in all Scales} that, if for each $\bar \mbfx_j \in \Sing(\bar \Sigma_j)$ we let $\mbfx \in \Sing(\Sigma)$ be the unique element in $\Sing(\Sigma)\cap B^{\bar g_j}(\bar \mbfx_j, r_0)$, there holds -- after suitably translating in order to make the singular points coincide, and blowing up to unit scale -- varifold convergence of $\Sigma^0_j$ (respectively: $\Sigma^1_j$) in a neighborhood of $\mbfx^1_j$ of radius $2C^\cL \bfr^{\Sigma^0_j}_{\Sigma^1_j, g^1_j}(\mbfx^{1}_j)$ (resp.: in a neighborhood of $\mbfx^0_j$ of radius $2C^\cL \bfr^{\Sigma^1_j}_{\Sigma^0_j, g^0_j}(\mbfx^0_j)$) to the truncated cone $\mbfC_{\mbfx}(\Sigma)\cap \BB(1) $; in turn, this
            implies, for all sufficiently large $j$ that
            \begin{align*}
                &\quad \frac{\|\Sigma^1_j\|_{g^1_j}(\bigcup_{\mbfx \in \Sing(\Sigma^0_j)} B^{g^0_j}(\mbfx, 2C^\cL \bfr^{\Sigma^1_j}_{\Sigma^0_j, g^0_j}(\mbfx))) + \|\Sigma^0_j\|_{g^0_j}(\bigcup_{\mbfx \in \Sing(\Sigma^1_j)} B^{g_1}(\mbfx, 2C^\cL \bfr^{\Sigma^0_j}_{\Sigma^1_j, g^1_j}(\mbfx)))}{d_j} \\
                &\leq Q C \cdot\frac{\left(\max_{\mbfx \in \Sing(\Sigma^0_j)}\bfr^{\Sigma^1_j}_{\Sigma^0_j, g^0_j}\right)^n(\mbfx) + \left(\max_{\mbfx \in \Sing(\Sigma^1_j)} \bfr^{\Sigma^0_j}_{\Sigma^1_j, g^1_j}(\mbfx)\right)^n}{d_j}\\
                &\leq QC\cdot \frac{\|\mbfG^{\Sigma^1_j}_{\Sigma^0_j, g^0_j}\|^n_{L^2(\Sigma^0_j \setminus B^{g^0_j}(\Sing(\Sigma^0_j), 2r_0))} + \|\mbfG^{\Sigma^0_j}_{\Sigma^1_j, g^1_j}\|^n_{L^2(\Sigma^1_j \setminus B^{g^1_j}(\Sing(\Sigma^1_j), 2r_0))}}{d_j}
                \to 0\,
            \end{align*}
            where this last conclusion follows from \eqref{eqn:estimate_|x^1_j - x^0_j|_r^i_j} and we have denoted by $Q$ the number of singular points of $\Sigma$ (as stipulated in Section \ref{sec:Prelim}).
	 These three facts, that we just collected, when combined together contradict the definition of $d_j$.
	\end{proof}

        Summarizing, by Claim 2 and Corollary \ref{Cor_MSE_converg preserve equ}, if we divide the PDE~\eqref{eqn:w_j_f_j} by $d_j$, we get subsequential $L^2_{loc}$ convergence of $\{ d^{-1}_j \mbfT^{\bar{\Sigma}_j}_{\Sigma,g} (w_j)\}_j$ to a non-zero weak solution $\hat u_\infty\in W^{1,2}_{loc}(\Sigma; \mbfV)$ of the limit equation 
        \[
            -L_{\Sigma, g} \hat u_\infty + \frac{n}{2}\nabla^\perp_{\Sigma, g}(\hat f_\infty) = 0\,.
        \]
        By Corollary~\ref{Cor_L^2 Growth est for pair}, for each $\bar \mbfx_j \in \Sing(\bar \Sigma_j)$ the corresponding ratio $\frac{\mbfx^0_j - \mbfx^1_j}{d_j}$ converges (again up to extracting a subsequence), which determines a translation-like section $\phi\in W_{\text{Tr}}(\Sigma; \mbfV)$. Moreover, by the last inequality in the statement of Corollary~\ref{Cor_L^2 Growth est for pair}, for every $p \in \Sing(\Sigma)$,
        \[
            \cA\cR_p(\hat{u}_\infty - \phi) \geq \gamma > \gamma_{-}(\mbfC_p)\,.
        \]
        Note that this conclusion relies on the fact that the $s^{-1}$ term in the last inequality of Corollary~\ref{Cor_L^2 Growth est for pair} does not affect the asymptotic rate, as for every fixed $s > 0$, the coefficient of $s^{-1}$ converges to $0$ in the limit as $j \to \infty$. Therefore, (recalling that we had set $\gamma=1-2\sigma$, satisfying \eqref{eqn:sigma_choice}) by item \ref{item:ARSob} of Lemma \ref{Lem_Anal for slower growth func} there holds $\hat{u}_\infty \in \hat W^{0,2}_{\gamma'}(\Sigma; \mbfV)$ for every $\gamma'<\gamma$, and thus $\hat{u}_\infty$ is indeed a tame section.
    \end{proof}

\subsection{The local Sard-Smale theorem}
The goal of this section is to prove Theorem \ref{Thm_Local Sard Thm}.
    Recalling the notion of canonical neighborhoods $\cL^{k,\alpha}$ presented in Definition \ref{Def_injrad, canonical neighb}, let us start with the first preparatory statement:
    \begin{lemma}[Compactness of tame Jacobi fields] \label{Lem_Cptness for Slower Growth Jac Fields}
        Let $(M,g)$ be a closed Riemannian manifold of dimension $N\geq 3$, let
        $\Sigma$ be a connected, $n$-dimensional $\MSI$ therein, and let $\varsigma_0$ be as in Lemma \ref{Lem_Cptness for scL^k}.   
        Assume the existence of a sequence $\set{(g_j, \Sigma_j)}_{j\geq 1}$ of pairs in $\cL^{k,\alpha}(g, \Sigma; \Lambda, \varsigma_0)$ satisfying the following two conditions:
        \begin{enumerate} [label=(\roman*)]
            \item\label{item:CompactJac1} $(g_j, \Sigma_j) \to (g_\infty, \Sigma_\infty) \in \cL^{k,\alpha}(g, \Sigma; \Lambda, \varsigma_0)$ in $\cM_n^{k,\alpha}(M)$ as $j\to\infty$;
            \item\label{item:CompactJac2}  for every $j\geq 1$ there exists a tame Jacobi field $u_j$ on $\Sigma_j$ such that  $\|u_j\|_{L^2(\Sigma_j)} = 1$.
        \end{enumerate}
        Then, after passing to a subsequence, $v_j:=\mbfT^{\Sigma_j}_{\Sigma_{\infty},g_{\infty}}(u_j)$ converges in $L^2_{loc}$ (and weakly in $W^{1,2}_{loc}$) to some tame Jacobi field $u_\infty$; furthermore for any $\eps>0$ there exist $s=s(\eps)\in (0,1)$ and $j_0=j_0(\eps)$
          such that there holds 
          \[
\sum_{p_j\in\Sing(\Sigma_j)}\int_{B^{g_j}(p_j,s)}|u_j|^2\,d\|\Sigma_j\|\leq\eps, \ \ \text{for all} \ j\geq j_0\,;
\]
as a consequence, in particular $\|u_\infty\|_{L^2(\Sigma_\infty)} = 1$.  
    \end{lemma}
    \begin{proof}

       By Corollary \ref{Cor_MSE_converg preserve equ}, the sequence $\left\{v_j\right\}$ subconverges to some Jacobi field $u_\infty$ on $\Sigma_\infty$ in $L^2_{loc}$ and, by classical (interior) elliptic estimates, such a section belongs to $W^{k,2}_{loc}(\Sigma_\infty;\mbfV_\infty)$;
    the effort here is to prove the non-concentration claim, and to show that $u_\infty$ is tame.
       
       We first fix $\sigma\in (0, 1)$ such that \[
           \min_{p\in \Sing(\Sigma_\infty)}\dist_\RR(1-\sigma, \Gamma(\mbfC_p(\Sigma_\infty))) = \sigma\,.
       \]
       For every $p\in \Sing(\Sigma_\infty)$, there exists $r_0=r_0(\Sigma_\infty, g_\infty)>0$ such that the assumptions in Corollary \ref{Cor_App_Uniform Noncon for Jac fields} are satisfied for $(\Sigma_\infty, r_0^{-2}g_\infty)$ near $p$.  
    
       By the way we have defined the notion of convergence in $\cL^{k,\alpha}$ and thanks to the analysis we carried out in Appendix \ref{sec:QuantUniq}, if $p_j\in \Sing(\Sigma_j)$ are approaching $p_\infty$ (thus with $\theta_{\Sigma_j}(p_j)=\theta_{\Sigma_{\infty}}(p)$) we know that for $j$ large enough, the assumptions in Corollary \ref{Cor_App_Uniform Noncon for Jac fields} are also satisfied for $(\Sigma_j, r_0^{-2}g_j)$ near $p_j$. Therefore, Corollary \ref{Cor_App_Uniform Noncon for Jac fields} provides an $\varepsilon_1(\Lambda, \sigma)\in (0, 1)$ and a uniform $L^2_{1-\sigma}$-estimate on $u_j$ near $p_j$ of the form
       \be\label{Equ_Cptness of Slower growth Jac}   
           \|\phi_j\|_{L^\infty(\Sigma_j)} + \sup_{i\geq 0}\|u_j-\phi_j\|_{L^2_{1-\sigma}(\Sigma_j\cap A^{g_j}(p_j, 2^{-i-1}r_0, 2^{-i}r_0))} \leq C( \Lambda, \sigma)\|u_j\|_{L^2(\Sigma_j\cap A^{g_j}(p_j, \varepsilon_1 r_0, 2r_0))}
           \ee
           for some translation-like section $\phi_j \in W_{\text{Tr}}(\Sigma_j;\mbfV_j)$.
       In particular, this implies for every $s\in (0, r_0)$ and again all sufficiently large indices $j$
       \[
           \|u_j\|_{L^2(\Sigma_j\cap B^{g_j}(p_j, s))} \leq C(\Lambda,\sigma)s^{n/2}\,;
       \]
      from this inequality it is straightforward to derive the $L^2$-non-concentration claim in the statement: specifically we get at once that 
       \be\label{eq:NonConcGE}
\int_{\Sigma_{\infty}}|u_{\infty}|^2\,d\|\Sigma_{\infty}\|\geq 1.
\ee
        By passing to the limit as $j\to \infty$ for the transferred sections, (\ref{Equ_Cptness of Slower growth Jac}) also holds with $u_\infty$ in place of $u_j$ and some translation-like function $\phi_\infty \in W_{\text{Tr}}(\Sigma_{\infty};\mbfV_{\infty})$ on $\Sigma_\infty$ in place of $\phi_j$. Hence, the $L^2_{loc}$-convergence implies that in fact equality must hold in \eqref{eq:NonConcGE}. Moreover, by the choice of $\sigma$ and Lemma \ref{Lem_Anal for slower growth func} \ref{item:ARSob}, the section $u_\infty$ is also a tame Jacobi field, which completes the proof.
    \end{proof}

    \begin{corollary} \label{Cor_Upper Semi-conti of dim Ker^+ L}
      (Setting as above.)  There exists $\varsigma_1(g,\Sigma, \Lambda)\in (0, \varsigma_0)$ such that for every $(g', \Sigma')\in \cL^{k,\alpha}(g, \Sigma; \Lambda, \varsigma_1)$, we have \[
            \dim \widehat{\Ker}_{\tau} (L_{\Sigma', g'}) \leq \dim \widehat{\Ker}_{\tau}(L_{\Sigma, g}) \,.  \]
    \end{corollary}
    \begin{proof}
        Let us argue by contradiction: if the assertion was false, then for any $j\geq 1$ large enough one could find $(g_j, \Sigma_j) \in \cL^{k,\alpha}(g, \Sigma; \Lambda, 1/j)$ such that $\dim \widehat{\Ker}_{\tau}(L_{\Sigma_j, g_j}) > \dim \widehat{\Ker}_{\tau}(L_{\Sigma, g})$. In particular, such a sequence of pairs will converge to $(g,\Sigma)$ in $\cM^{k,\alpha}_n$ .
        One can then take for each such $j$ an $L^2(\Sigma_j;\mbfV_j)$-orthonormal family, of cardinality exactly equal to $d:=\widehat{\Ker}_{\tau}(L_{\Sigma, g})+1$, in $\widehat{\Ker}_{\tau}(L_{\Sigma_j, g_j})$, say $\{u^{(1)}_j,\ldots, u^{(d)}_j \}$; we then know, by appealing to Lemma \ref{Lem_Cptness for Slower Growth Jac Fields} that one can extract, after passing to a subsequence, limit elements $u^{(1)}_{\infty},\ldots,u^{(d)}_{\infty}$ in $L^2(\Sigma;\mbfV)$ such that for each $i=1,\ldots, d$ there holds $\mbfT^{\Sigma_j}_{\Sigma,g}(u^{(i)}_{j})\to u^{(i)}_{\infty}$ in $L^{2}_{loc}(\Sigma;\mbfV)$ as $j\to\infty$. But, in fact, it follows at once 
        from the non-concentration claim proven above (also part of Lemma \ref{Lem_Cptness for Slower Growth Jac Fields})
        that the family in question
            consists of pairwise orthogonal sections also having unit norm in $L^2(\Sigma;\mbfV)$; hence $\widehat{\Ker}_{\tau}(L_{\Sigma, g})$ would have dimension no less than $d$, a contradiction.
    \end{proof}

    Let then 
    \begin{align*}
        \cL^{k,\alpha}_{top}(g, \Sigma; \Lambda, \delta):= \{(g', \Sigma')\in \cL^{k,\alpha}(g, \Sigma; \Lambda, \delta): \dim \widehat{\Ker}_{\tau} (L_{\Sigma', g'}) =\dim \widehat{\Ker}_{\tau} (L_{\Sigma, g})\}\,.
    \end{align*}
    Lemma \ref{Lem_Cptness for Slower Growth Jac Fields} and Corollary \ref{Cor_Upper Semi-conti of dim Ker^+ L} imply that $\cL^{k,\alpha}_{top}(g, \Sigma; \Lambda, \delta)$ is closed in $\cL^{k,\alpha}(g, \Sigma; \Lambda, \delta)$ for $\delta\leq \varsigma_1$.
        
In the setting above, recalling the content of Corollary \ref{cor:AR_Jacobi} and Remark \ref{rem:IndipDimCoker}, we further set for notational convenience
    \begin{align*}
        I:= \dim \widehat{\Ker}_{\tau} (L_{\Sigma, g})\,, & &
        J:= \dim \widehat{\Coker}_{\tau} (L_{\Sigma, g}) \,.
    \end{align*}
    Note that, not only both numbers are independent of $\tau$ as long as \eqref{eq:RangeTau} is satisfied, but also by Theorem \ref{thm:Count} there \emph{always} holds $J\geq I$.

       \

    For $(\bar{g}, \bar{\Sigma})\in \cL^{k,\alpha}_{top}(g, \Sigma; \Lambda, \varsigma_1)$, let \[
      \pi^{L^2}_{\bar{\Sigma}, \bar{g}}: L^2(\bar{\Sigma};\mbfV) \to \widehat{\Ker}_{\tau} (L_{\bar{\Sigma},\bar{g}})
    \] 
    be the $L^2$-orthogonal projection to such a finite dimensional (hence: closed) subspace. 
    The following lemma guarantees that we can parametrize slices of $\cL^{k,\alpha}_{top}$ by compact subset of $\widehat{\Ker}$.

    \begin{lemma} \label{Lem_Loc Parametriz of scL^k}
        Let $\varsigma_1$ be as in Corollary \ref{Cor_Upper Semi-conti of dim Ker^+ L}, and let $\tau$ satisfy \eqref{eq:RangeTau}. Then there exist constants $\varsigma_2(g, \Sigma, \Lambda)\in (0, \varsigma_1)$, $r_0(g, \Sigma, \Lambda)>0$ and a linear subspace $\cF\subset C^{k,\alpha}_c(M\setminus B^g(\Sing(\Sigma), {10r_0}))$ of dimension $J$, also depending only on $g,\Sigma, \Lambda$, satisfying the following property
        
            Denoted for simplicity $\cF\cdot \bar{g}:= \{(1+f)\bar{g}: f\in \cF\}$,  for every $(\bar{g}, \bar{\Sigma})\in \cL^{k,\alpha}_{top}(g, \Sigma; \Lambda, \varsigma_2)$ the map
            \begin{align*}
                \mbfP_{\bar{g}, \bar{\Sigma}}: \cL^{k,\alpha}_{top}(g, \Sigma; \Lambda, \varsigma_2)\cap \Pi^{-1}(\cF \cdot\bar{g}) & \to \widehat{\Ker}_{\tau} (L_{\bar{\Sigma}, \bar{g}}), \\
                (g', \Sigma') & \mapsto  \pi_{\bar{\Sigma}, \bar{g}}^{L^2}\big( \mbfG^{\Sigma'}_{\bar{\Sigma}, \bar{g}}\cdot \zeta_{\bar{\Sigma}, \bar{g}, r_0}\big),     
            \end{align*}
            is bi-Lipschitz onto its image, with bi-Lipschitz constant $\leq C(g, \Sigma, \Lambda)$, and thus injective. 
    \end{lemma}

    \begin{remark} \label{Rem_Metric Equi on cF}
        Note that, for what pertains to the second part of the statement above, it is understood that the semi-metric we use on $\cL^{k,\alpha}_{top}$ is $\mbfD^\cL$. 
    \end{remark}

    Note that, for instance, when $I=0, J>0$ and $g=\overline{g}, \ \Sigma=\overline{\Sigma}$ the statement in question implies that the domain of the map $\mbfP_{\bar{g}, \bar{\Sigma}}$, which is $\cL^{k,\alpha}(g, \Sigma; \Lambda, \varsigma_2)\cap \Pi^{-1}(\cF \cdot\bar{g})$, must consist of a single point. This outcome can, at least heuristically, interpreted in terms of transversality of the two sets in question.
    
    \begin{proof}

    In order to prove the statement, we first need the following technical claim. Set  \be\label{eq:SpecialValueTau}
     \tau_0:=\frac{1}{2}\left(1+\sup_{p\in \Sing(\bar{\Sigma})} \gamma^-(\mbfC_p)\right),
    \ee
    which of course satisfies \eqref{eq:RangeTau},  let $\varphi_1, ..., \varphi_J\in W^{k-2,2}_{\tau_0-2}(\Sigma;\mbfV)$ be a basis (after projecting to the quotient space) of $\widehat{\Coker}_{\tau_0}(L_{\Sigma, g})$; without loss of generality (by the way we defined weighted Sobolev spaces) we can assume all such sections to be supported away from the singular set of $\Sigma$.  Set then
        \[
            r_0(g, \Sigma, \Lambda)\in (0, \min_{1\leq i\leq J}\{\injrad(g, \Sigma), \dist_g(\spt(\varphi_i), \Sing(\Sigma))\}/40)\,
        \]
        chosen
        small enough that the linear map
        \[
            \widehat{\Ker}_{\tau_0}(L_{\Sigma, g}) \to \widehat{\Ker}_{\tau_0}(L_{\Sigma, g}), \quad v \mapsto \pi^{L^2}_{\Sigma, g}(\zeta_{\Sigma, g, r_0}\cdot v)   \]
        is an isomorphism, where $\zeta_{\Sigma, g, r_0}$ is defined in (\ref{Equ_Def zeta_{Sigma, g, r_0}}). 
        Then, by a standard regularization argument we choose $f_i\in C^{k,\alpha}_c(M\setminus B^g(\Sing(\Sigma), 10r_0))$ such that, after projecting to the quotient space,
        $\mathrm{span}_\RR\langle \nabla^\perp_{\Sigma, g} f_i: 1\leq i\leq J\rangle= \mathrm{span}_\RR\langle \varphi_i: 1\leq i\leq J\rangle$ along $\Sigma$, and thus define 
        \begin{align*}
            \cF:= \mathrm{span}_\RR\langle f_i: 1\leq i\leq J\rangle\,.            
        \end{align*}

\textbf{Claim.} For every $(\bar{g}, \bar{\Sigma})\in \cL^{k,\alpha}_{top}(g, \Sigma; \Lambda, \varsigma_2)$ and every non-zero $f\in \cF$, there is no section $v\in \hat{W}^{k,2}_{\tau_0}(\bar{\Sigma};\mbfV)$ solving $L_{\bar{\Sigma}, \bar{g}} v = \nabla^\perp_{\bar{\Sigma}, \bar{g}}(f)$.

        To see this, suppose, for a contradiction, that there exist $(g_j, \Sigma_j) \to (g, \Sigma)$ in $\cL^{k,\alpha}_{top}(g, \Sigma; \Lambda, \varsigma_1)$, $f_j\in \cF$ non-zero 
        and $u_j\in \hat{W}^{k,2}_{\tau_0}(\Sigma_j;\mbfV_j)$ such that 
        \begin{align*}
            L_{\Sigma_j, g_j}(u_j) = \nabla^\perp_{\Sigma_j, g_j} (f_j)\,, & &
            \pi^{L^2}_{\Sigma_j, g_j} (u_j) = 0\,, & &
            \|u_j\|_{L^2(\Sigma_j)} + \|f_j\|_{C^{k,\alpha}(M)} = 1\,.
        \end{align*}   
        Note that since $f_j = 0$ in $B^g(\Sing(\Sigma), 10 r_0)$, $u_j$ are in fact Jacobi fields in $\Sigma_j\cap B^g(\Sing(\Sigma), 10 r_0)$; such Jacobi fields belong to $\hat{W}^{k,2}_{\tau_0}(\bar{\Sigma};\mbfV)$, hence (by Corollary \ref{cor:AR_Jacobi}) belong in fact to $\hat{W}^{k,2}_{\tau}(\bar{\Sigma};\mbfV)$ for any $\tau$ satisfying \eqref{eq:RangeTau}.
     Now, after passing to a subsequence (which we do not rename), we may assume $f_j\to f_{\infty}\in \cF$ as $j\to \infty$ and, on the other hand, $v_j:=\mbfT^{\Sigma_j}_{\Sigma,g}(u_j)$ subconverges in $L^2_{loc}$ (and weakly in $W^{1,2}_{loc}$) to some $v_\infty\in W_{loc}^{1,2}(\Sigma;\mbfV)$ that still satisfies (in weak sense, a priori)
        \[
            L_{\Sigma, g}v_\infty = \nabla^\perp_{\Sigma, g} (f_\infty).
        \]
        Standard elliptic regularity then gives $v_\infty\in W_{loc}^{k,2}(\Sigma;\mbfV)$.
 Note, further, that since $u_j$ are tame Jacobi fields near $\Sing(\Sigma_j)$, arguing as in the proof of the Lemma \ref{Lem_Cptness for Slower Growth Jac Fields}, one shows that $u_j$ has uniform growth upper bound near $\Sing(\Sigma_j)$, which then implies (by item \ref{item:ARSob} of Lemma \ref{Lem_Anal for slower growth func}) that $v_\infty$ is itself a tame section (in particular $v_\infty\in \hat{W}^{k,2}_{\tau_0}(\Sigma;\mbfV)$). Therefore, we will have in fact:
       \begin{align}\label{eq:Trio}
            L_{\Sigma, g}v_\infty = \nabla^\perp_{\Sigma, g} (f_\infty)\,, & &
            \pi^{L^2}_{\Sigma, g} (v_\infty) = 0 \,, & & 
            \|v_\infty\|_{L^2(\Sigma)} + \|f_\infty\|_{C^{k,\alpha}(M)} = 1\,.
        \end{align}
             Here, the second equation is based on the fact that $(\Sigma_j, g_j)\in \cL^{k,\alpha}_{top}$ and thus $\dim \widehat{\Ker}_{\tau_0}(L_{\Sigma_j,g_j}) = \dim \widehat{\Ker}_{\tau_0}(L_{\Sigma, g})$ so that, in view of Lemma \ref{Lem_Cptness for Slower Growth Jac Fields}, an orthonormal basis of $\widehat{\Ker}_{\tau_0}(L_{\Sigma_j,g_j})$ will converge to an orthonormal basis of $\widehat{\Ker}_{\tau_0}(L_{\Sigma,g})$.

        By our choice of $\cF$, it is clear that the right-hand side of the first equation of \eqref{eq:Trio} must vanish, and so must $f_\infty$, which then implies that $v_\infty$ is in fact a tame Jacobi field on $\Sigma$ satisfying, in addition, $\|v_\infty\|_{L^2(\Sigma)}=1$. However, this contradicts the condition $\pi^{L^2}_{\Sigma, g} (v_\infty) = 0$, thereby completing the proof of the claim.
\medskip       
            
        To prove the lemma, we will also argue by contradiction. Observe, preliminarily, that by Corollary \ref{cor:AR_Jacobi} and the fact that we are considering the $L^2$-distance on  $\widehat{\Ker}_{\tau} (L_{\bar{\Sigma}, \bar{g}})$ the  assertion in question is \emph{a posteriori} independent of $\tau$ (within the usual range) and thus it suffices, without loss of generality, to prove it for $\tau=\tau_0$ as we henceforth tacitly assume.
        
        Suppose $\mbfP_{\bar{g}_j, \bar{\Sigma}_j}$ is not uniformly bi-Lipschitz for $(\bar{g}_j, \bar{\Sigma}_j)\in \cL^{k,\alpha}_{top}(g, \Sigma; \Lambda, \varsigma_2)$ with $(\bar{g}_j, \bar{\Sigma}_j)\to (g, \Sigma)$ in $\cL^{k,\alpha}$. (Here the constant $\varsigma_2\in (0,\varsigma_1)$ is chosen, at this stage, so that the conclusion of the previous claim holds true; in the end it shall possibly be chosen even smaller.) Then, this means that there exists a sequence of pairs $(g_j^0, \Sigma_j^0), (g_j^1, \Sigma_j^1)\to (g, \Sigma)$ in $\cL^{k,\alpha}_{top}(g, \Sigma; \Lambda, \varsigma_2)$ such that 
        \begin{itemize} 
            \item $g_j^i = (1 + f_j^i)\bar{g}_j$, $i=0, 1$, $f_j^i\in \cF$;
            \item  set $u_j^{(i)} := \mbfG_{\bar{\Sigma}_j, \bar{g}_j}^{\Sigma_j^i}, v^{(i)}_j=\mbfT^{\bar{\Sigma}_j}_{\Sigma,g}(u^{(i)}_j)$; $\zeta_j:= \zeta_{\bar{\Sigma}_j, \bar{g}_j, r_0}$, $d_j:= \mbfD^{\cL}((g_j^0, \Sigma_j^0), (g_j^1, \Sigma_j^1))$ one of the following holds:
              \begin{align}
                  \text{either } & \  \|\pi^{L^2}_{\bar{\Sigma}_j, \bar{g}_j}((u_j^{(1)}-u_j^{(0)})\zeta_j)\|_{L^2(\bar{\Sigma}_j)} \leq \frac{d_j}{j}, \label{Equ_Either Large Lip inverse} \\  
                  \text{or }     & \ \|\pi^{L^2}_{\bar{\Sigma}_j, \bar{g}_j}((u_j^{(1)}-u_j^{(0)})\zeta_j)\|_{L^2(\bar{\Sigma}_j)} \geq j\cdot d_j. \label{Equ_Or Large Lip}
              \end{align}
        \end{itemize}
        
        By Lemma \ref{Lem_Tangent vectors to scL^k} the normalized difference $(f_j^{(1)} - f_j^{(0)})/d_j$ subconverges in $C^{k,\alpha}(M)$ to some $\hat{f}_\infty \in \cF$ and $(v_j^{(1)}-v_j^{(0)})/d_j$ subconverges  in $W^{1,2}_{loc}$ to some non-zero, tame section $\hat{u}_\infty$, solving 
        \begin{align}
            L_{\Sigma, g} \hat{u}_\infty = \frac{n}{2}\nabla^\perp_{\Sigma, g}(\hat{f}_\infty).   \label{Equ_Tang vec of scL cap cF solves L u = c nu(f)}
        \end{align}

        Now, since $\zeta_j$ is supported away from $\Sing(\Sigma)$, by the aforementioned convergence clearly $\zeta_j(v_j^{(1)}-v_j^{(0)})/d_j$ subconverges in $L^2$ to $\hat{u}_\infty\cdot \zeta_{\Sigma, g, r_0}$ and thus (\ref{Equ_Or Large Lip}) cannot possibly hold for $j$ large enough.  

        We then need to rule out the other possibility, namely \eqref{Equ_Either Large Lip inverse}. Towards that goal, by (\ref{Equ_Tang vec of scL cap cF solves L u = c nu(f)}) and the claim above, we derive $\hat{f}_\infty = 0$, which means $0\neq \hat{u}_\infty \in \widehat{\Ker}_{\tau_0}(L_{\Sigma, g})$. 
        Then since $\zeta_j(v_j^{(1)}-v_j^{(0)})/d_j$ subconverges in $L^2(\Sigma)$ to $\hat{u}_\infty\cdot \zeta_{\Sigma, g, r_0}$, passing to the limit as $j\to\infty$ in \eqref{Equ_Either Large Lip inverse} we would conclude $\pi^{L^2}_{\Sigma, g}(\hat{u}_\infty\cdot \zeta_{\Sigma, g, r_0}) = 0$ (here by the same argument as for proving the previous claim, we rely on the assumption $(\Sigma_j, g_j)\in \cL^{k,\alpha}_{top}$). However, by the choice of $r_0$ we have made, this would force $\hat{u}_\infty$ to be the trivial section, a contradiction.   
    \end{proof}

   Having established this, we can prove Theorem \ref{Thm_Local Sard Thm} by unwrapping the proof of the Sard-Smale Theorem, as we are now about to explain.

    \begin{proof}[Proof of Theorem \ref{Thm_Local Sard Thm}.]
        We claim that the statement holds true for $\varsigma=\varsigma_2(g,\Sigma,\Lambda)$, where the latter constant is as determined in Lemma \ref{Lem_Loc Parametriz of scL^k}; so, in what follows, let indeed $\delta\in (0,\varsigma_2]$.
    
        First of all, we note that openness follows directly from Lemma \ref{Lem_Cptness for scL^k}: indeed, the image --- via the continuous map $\Pi$ --- of the compact set $\cL^{k,\alpha}(g, \Sigma; \Lambda, \delta)$ is compact with respect to $C^k$ convergence and therefore closed (in $\cG^{k,\alpha}(M)$).

          That being said, we shall prove the denseness inductively on $I:= \dim \widehat{\Ker}_{\tau}(L_{\Sigma, g})\geq 0$.  
        By induction, we can assume that the denseness has been established for every pair $(g, \Sigma)$ with $\dim \widehat{\Ker}_{\tau}(L_{\Sigma, g}) \leq I-1$ (Note that when $I=0$, the base case of induction, we are actually making no assumption.) The goal now is to prove it for $I$.

        \

        \textbf{Claim.}
          For each $\bar{g}\in \cG^{k,\alpha}(M)\setminus \cG^{k,\alpha}(g, \Sigma; \Lambda, \delta)=\Pi(\cL^{k,\alpha}(g, \Sigma; \Lambda, \delta))$, there exists a sequence of metrics $g_j \in \cF\cdot \bar{g}\setminus \Pi(\cL^{k,\alpha}_{top}(g, \Sigma; \Lambda, \delta))$ such that $g_j\to \bar{g}$ in $\cG^{k,\alpha}(M)$ when $j\to \infty$.

\
            
        We first finish the proof of denseness assuming this claim: it suffices to show that each $g_j$ in the claim can be approximated by metrics in $\cG^{k,\alpha}(g, \Sigma; \Lambda, \delta)$. Since $\cL^{k,\alpha}_{top}(g, \Sigma; \Lambda, \delta) = \cL^{k,\alpha}(g, \Sigma; \Lambda, \delta)$ when $I=0$, we only need to handle the case $I\geq 1$.
            
        Let $j\geq 1$ be fixed. Suppose, without loss of generality, that $g_j\notin \cG^{k,\alpha}(g, \Sigma; \Lambda, \delta)$. Hence $\Pi^{-1}(g_j)\cap \cL^{k,\alpha}(g, \Sigma; \Lambda, \delta)\neq \emptyset$. 
        By definition of $\cL^{k,\alpha}_{top}$ and Corollary \ref{Cor_Upper Semi-conti of dim Ker^+ L} we have, for every $(g_j, \Sigma')\in \cL^{k,\alpha}(g, \Sigma; \Lambda, \delta)$, $\dim \widehat{\Ker}_{\tau}(L_{\Sigma', g_j}) \leq I-1$. Hence by the inductive assumption, there exists $\delta_{g_j, \Sigma'} = \delta(g_j, \Sigma', \Lambda)>0$ such that $\cG(g_j, \Sigma'; \Lambda, \delta_{g_j, \Sigma'})$ is open and dense in $\cG^{k,\alpha}(M)$. 
        Since $\cL^{k,\alpha}(g_j, \Sigma'; \Lambda, \delta_{g_j, \Sigma'})$ contains an open neighborhood of $(g_j, \Sigma')$ in $\cL^{k,\alpha}(g, \Sigma; \Lambda, \delta)$ and, by Lemma \ref{Lem_Cptness for scL^k}, $\Pi^{-1}(g_j) \cap \cL^{k,\alpha}(g, \Sigma; \Lambda, \delta)$ is compact, we thus know that there exists a finite sequence of pairs $(g_j, \Sigma^{(1)})$, $(g_j, \Sigma^{(2)})$, $\dots$, $(g_j, \Sigma^{(i)}) \in \Pi^{-1}(g_j)\cap \cL^{k,\alpha}(g, \Sigma; \Lambda, \delta)$, $1\leq i\leq i_0$, such that 
        \begin{align}
            \bigcup_{i=1}^{i_0} \cL^{k,\alpha}(g_j, \Sigma^{(i)}; \Lambda, \delta_{g_j, \Sigma^{(i)}}) \supset \Pi^{-1}(g_j) \cap \cL^{k,\alpha}(g, \Sigma; \Lambda, \delta).  \label{Equ_Pi^-1(g_j) covered by smaller nullity neighb}
        \end{align}
        Since $\cG_j := \bigcap_{i=1}^{i_0} \cG(g_j, \Sigma^{(i)}; \Lambda, \delta_{g_j, \Sigma^{(i)}})$ is still open and dense in $\cG^{k,\alpha}(M)$, there exists a sequence $\{g'_m\}_{m\geq 1} \subset \cG_j$ such that $g'_m \to g_j$ in $\cG^{k,\alpha}$ as $m\to \infty$. Again by compactness of $\cL^{k,\alpha}(g, \Sigma; \Lambda, \delta)$, when $m\geq m_0$ is large enough, \[
            \cL^{k,\alpha}(g, \Sigma; \Lambda, \delta)\cap \Pi^{-1}(g'_m) \subset \bigcup_{i=1}^{i_0} \cL^{k,\alpha}(g_j, \Sigma^{(i)}; \Lambda, \delta_{g_j, \Sigma^{(i)}}) \,.
        \]
        By our choice of $g'_m$ this is only possible if $\cL^{k,\alpha}(g, \Sigma; \Lambda, \delta)\cap \Pi^{-1}(g'_m) = \emptyset$, and so $g'_m \in \cG^{k,\alpha}(g, \Sigma; \Lambda, \delta)$ for $m$ large enough. This concludes the proof of denseness assuming the Claim. \\

        \textit{Proof of Claim.} Suppose, without loss of generality, that $(\bar{g}, \bar{\Sigma})\in \cL^{k,\alpha}_{top}(g, \Sigma; \Lambda, \delta)$, for otherwise one just needs to take $g_j \equiv \bar{g}$.
        
        Recall by Lemma \ref{Lem_Loc Parametriz of scL^k}, $\cL^{k,\alpha}_{top}(g, \Sigma; \Lambda, \delta)\cap \Pi^{-1}(\cF\cdot \bar{g})$ is bi-Lipschitz embedded as a compact subset of $\widehat{\Ker}_{\tau}(L_{\bar{\Sigma},\bar{g}})$; we let, for the sake of notational convenience, $\cZ:= \mbfP_{\bar{g}, \bar{\Sigma}}(\cL^{k,\alpha}_{top}(g, \Sigma; \Lambda, \delta)\cap \Pi^{-1}(\cF\cdot \bar{g}))$ be its image in $\widehat{\Ker}_{\tau}(L_{\bar{\Sigma},\bar{g}})$.  Consider the composite map, \[
                \Pi: \cZ \cong \cL^{k,\alpha}_{top}(g, \Sigma; \Lambda, \delta)\cap \Pi^{-1}(\cF\cdot \bar{g}) \overset{\Pi}{\to} \cF\cdot \bar{g}\cong \cF,   \]
            that is a Lipschitz map between (compact subsets of) vector spaces.
            
            Since $\widehat{\ind}_{\tau}(L_{\Sigma, g})<0$, we know that $I<J$ and therefore $\Pi(\cZ)$ has dense complement 
         in $\cF\cdot\bar{g}\cong \cF$ (which, by definition, has dimension $J$). Hence we can find $g_j \in \cF\cdot \bar{g}\setminus \Pi(\cL^{k,\alpha}_{top}(g, \Sigma; \Lambda, \delta))$ to approximate $\bar{g}$, thereby finishing the proof of the claim.   \end{proof}

\subsection{Proof of Theorem \ref{thm:generic_perturbation}}

We will proceed in two steps: first we prove the theorem for metrics of finite, in fact $C^{k,\alpha}$, regularity (which is the setting we have so far considered in this section) and then we shall discuss the simple argument that allows to derive the smooth version of the result (namely: Theorem \ref{thm:generic_perturbation} in the specific form we have stated). For that purpose we will employ the following lemma about nested metric spaces.

\begin{lemma}\label{lem:MetricSpaces}
      Let $\{(X_k, d_k)\}_{k\geq k_0}$ be a sequence of complete metric spaces, with associated inclusions $\iota_k: X_{k+1}\hookrightarrow X_{k}$  such that 
      \begin{align}
        d_k\left( \iota_k(x), \iota_k(y) \right) \leq d_{k+1}(x, y)\,, \qquad \forall\, k\geq k_0\,.  \label{Equ_App_Metric Space_d_k<d_(k+1)} 
      \end{align}
      We identify, using such maps, each $X_k$ as a subset of $X_0$, and define 
      \begin{align*}
          X_\infty:= \bigcap_{k\geq k_0}X_k\,, & & d_\infty(x, y):= \sum_{k\geq k_0} 2^{-k}\cdot \frac{d_k(x, y)}{1+d_k(x, y)}, \quad \forall\, x, y\in X_\infty\,.
      \end{align*} 
      Suppose that $X_\infty$ is dense in $(X_k,d_k)$ for all $k\geq k_0$ and, in addition, that
       $G\subset (X_{k_0}, d_{k_0})$ is an open (respectively $G_\delta$) subset  such that for every $k\geq k_0$,
      $G\cap X_k$ is dense in $(X_k, d_k)$.
      Then $G\cap X_\infty$ is an open (resp. $G_\delta$) dense subset in $(X_\infty, d_\infty)$.
  \end{lemma}
  \begin{proof}
      The case when $G$ is a $G_\delta$ subset follows from the case when $G$ is open by appealing to the Baire Category Theorem. So we shall prove the statement assuming $G$ is open.

      First notice that for every $k\geq k_0$, since $(X_k, d_k)\hookrightarrow (X_{k_0}, d_{k_0})$ is continuous (being a finite composition of continuous maps), $G\cap X_k$ is open in $(X_k, d_k)$; the openness of $G\cap X_\infty$ in  $(X_\infty, d_\infty)$ analogously follows from the continuity of the inclusion $(X_{\infty}, d_{\infty})\to (X_{k_0}, d_{k_0})$, which can be checked at once by the way we have defined the distance $d_\infty$.

      To prove the denseness of $G\cap X_\infty$ in $(X_\infty, d_\infty)$, let $x\in X_\infty$ be an arbitrary point. By the denseness of $G\cap X_k$ in $(X_k, d_k)$, for each $k\geq k_0$ there exists $x_k\in G\cap X_k$ such that $d_k(x_k, x)\leq 1/k$; on the other hand, by the denseness of $X_\infty$ in $(X_k, d_k)$, there exists a sequence $\left\{x_k^j\right\}_{j\geq 0}$in $X_\infty$ such that \[
        \lim_{j\to \infty} d_k(x_k^j, x_k) = 0. \]
       Hence, by the openness of $G\cap X_k$ in $(X_k, d_k)$, for every $k\geq k_0$ we can find an index $j_k$ large enough that 
       \[
          \bar x_k:= x_k^{j_k}\in (G\cap X_k)\cap X_\infty = G\cap X_\infty\,, \qquad d_k(\bar x_k, x) \leq 2/k\,.
      \] 
      Then, recalling the definition of $d_{\infty}$, we find, \[
          d_\infty(x, \bar x_k) \leq \sum_{m=k_0}^k 2^{-m}\cdot \frac{2/k}{1+2/k} + \sum_{m=k+1}^\infty 2^{-m} \leq 4/k + 2^{-k} \to 0 \quad \text{ as }k\to \infty\,.
      \]
      Hence $x$ is in the closure of $G\cap X_\infty$ in $(X_\infty, d_\infty)$. 
  \end{proof}

\begin{proof}[Proof of Theorem \ref{thm:generic_perturbation}]

\

\begin{center}
$\boxed{\text{Step 1: metrics of finite regularity} \ C^{k,\alpha}.}$
\end{center}

    For every $(g, \Sigma) \in \cM^{k, \alpha}_n(M)$ with $\widehat{\ind}_{\tau}(L_{\Sigma,g}) < 0$ we know, by Theorem \ref{thm:Count}, that for some $p \in \Sing(\Sigma)$,
    \[
        \rmI(\mbfC_p) > 0\,.
    \]
    Note that there exists $\delta(\mathbf{C}_p) > 0$ such that for any (regular minimal) cone $\mathbf{C}\in\cC_{N,n}$ satisfying $\mbfF(\mbfC_p \cap \mathbb{S}^{N - 1}, \mbfC \cap \mathbb{S}^{N - 1}) \leq \delta(\mathbf{C}_p)$, we also have $\rmI(\mbfC) > 0$.
Hence, given $\Lambda\in\cR^{k,\alpha}_n(g,\Sigma)$,
    by Definition \ref{Def_injrad, canonical neighb} and the compactness Lemma \ref{Lem_Cptness for scL^k},  we can choose $\kappa(g, \Sigma, \Lambda) \in (0, \kappa_0(g, \Sigma, \Lambda)]$, where $\kappa_0$ is the threshold given by Theorem \ref{Thm_Local Sard Thm}, with the following property:
    
    \emph{For every $(g', \Sigma') \in \cL^{k, \alpha}(g, \Sigma; \Lambda, \kappa(g, \Sigma, \Lambda))$, there exists $p' \in \Sing(\Sigma')$ such that 
    \[
        \mbfF(\mbfC_p \cap \mathbb{S}^{N - 1}, \mbfC_{p'} \cap \mathbb{S}^{N - 1}) \leq \delta(\mathbf{C}_p)\,,
    \] and thus,
    \[
        \rmI(\mbfC_{p'}) > 0\,.
    \]
     In particular, $\widehat{\ind}_{\tau}(L_{\Sigma', g'}) < 0$ holds for every $(g', \Sigma') \in \cL^{k, \alpha}(g, \Sigma; \Lambda, \kappa(g, \Sigma, \Lambda))$.
     }

    (In all other cases, one shall simply let $\kappa(g,\Sigma,\Lambda)=\kappa_0(g,\Sigma,\Lambda)$.)

    By Theorem \ref{Thm_Countable Decomp}, for such a function $\kappa: \cE_n^{k,\alpha}(M) \to \RR_+$, we can obtain a countable cover consisting of canonical neighborhoods:
    \[
        \cM_n^{k,\alpha}(M) = \bigcup_{j\geq 1}\cL^{k,\alpha}(g_j, \Sigma_j; \Lambda_j, \kappa_j),   
    \]
    where $\kappa_j:= \kappa(g_j, \Sigma_j; \Lambda_j)$. Then, for every integer $j\geq 1$, we let
    \begin{align*}
        \cL^{k, \alpha}_-(g_j, \Sigma_j; \Lambda_j, \kappa_j) &:= \{(g, \Sigma) \in \cL^{k, \alpha}(g_j, \Sigma_j; \Lambda_j, \kappa_j): \widehat{\ind}_{\tau}(L_{\Sigma, g}) < 0\}\,.
    \end{align*}
    We claim that, for any fixed $j\geq 1$, the set $\cL^{k, \alpha}_-(g_j, \Sigma_j; \Lambda_j, \kappa_j)$ can also be covered by countably many canonical neighborhoods of the form $\cL^{k, \alpha}(g, \Sigma; \Lambda, \kappa(g, \Sigma, \Lambda))$ with $\widehat{\ind}_{\tau}(L_{\Sigma,g}) < 0$. Indeed, by the compactness Lemma \ref{Lem_Cptness for scL^k} again (now considering a countable exhaustion of $\cL_{-}^{k, \alpha}(g_j, \Sigma_j; \Lambda_j, \kappa_j)$ by closed --- hence compact --- sets, defined e.\,g. as $1/i$-sublevel sets of the distance function from the complement of such $\cL_{-}^{k, \alpha}$ in   $\cL^{k,\alpha}(g_j, \Sigma_j; \Lambda_j, \kappa_j)$) it suffices to show that for every $(g, \Sigma) \in \cL^{k, \alpha}(g_j, \Sigma_j; \Lambda_j, \kappa_j)$, the set $\cL^{k,\alpha}(g, \Sigma;\Lambda_j, \kappa(g, \Sigma, \Lambda_j))$ contains an open neighborhood of $(g, \Sigma)$ in $\cL^{k, \alpha}(g_j, \Sigma_j; \Lambda_j, \kappa_j)$. 
    
    Suppose for the sake of contradiction that there exists $(g, \Sigma) \in \cL^{k, \alpha}(g_j, \Sigma_j; \Lambda_j, \kappa_j)$ and a sequence $\{(g_i, \Sigma_i)\} \subset \cL^{k, \alpha}(g_j, \Sigma_j; \Lambda_j, \kappa_j) \setminus \cL^{k, \alpha}(g, \Sigma; \Lambda_j, \kappa(g, \Sigma, \Lambda_j))$,
    so that
    \[
        (g_i, \Sigma_i) \rightarrow (g, \Sigma)\, \ \text{in} \ \cM_n^{k,\alpha}(M). 
    \]
    For large enough $i$, based on the definition of $\cL^{k, \alpha}(g_j, \Sigma_j; \Lambda_j, \kappa_j)$ and the notion of convergence in $\cM^{k,\alpha}_n$, we have
    \begin{itemize}
        \item $g_i$ is a $C^{k, \alpha}$ metric on $M$ with $\|g_i\|_{C^{k, \alpha}}\leq \Lambda_j$ and $\|g-g_i\|_{C^{k}}\leq \kappa(g, \Sigma, \Lambda_j)$;
        \item $\Sigma_i$ is an $\MSI$ in $(M, g_i)$ satisfying
              \[ \mfr_{\Sigma_i, g_i} \geq \Lambda_j^{-1}\rho_{\Sigma_i,g_i}, \quad \mbfF(|\Sigma_i|_{g_i}, |\Sigma|_{g})\leq \kappa(g, \Sigma, \Lambda_j)
              .\]
              
    \end{itemize}
    We further note that (since $(g_i, \Sigma_i) \in \cL^{k, \alpha}(g_j, \Sigma_j; \Lambda_j, \kappa_j))$ there is an ordered bijection between the points of $\Sing(\Sigma_i)$ and $\Sing(\Sigma_j)$, and the densities of the corresponding cones are equal; also, since $(g, \Sigma) \in \cL^{k, \alpha}(g_j, \Sigma_j; \Lambda_j, \kappa_j)$
     there is also a bijection between the points of $\Sing(\Sigma)$ and $\Sing(\Sigma_j)$, and the densities of the corresponding cones are equal. Putting all information together,
    this implies that eventually $(g_i, \Sigma_i) \in \cL^{k, \alpha}(g, \Sigma; \Lambda_j, \kappa(g, \Sigma, \Lambda_j))$, a contradiction.

    Therefore, by now repeating the argument and construction above as one varies $j\geq 1$, we get that the set 
    \[
        \{(g, \Sigma) \in \cM_n^{k,\alpha}(M) : \widehat{\ind}_{\tau}(L_{\Sigma,g}) < 0\} = \bigcup_{j \geq 1} \cL^{k,\alpha}_{-}(g_j, \Sigma_j; \Lambda_j, \kappa_j)
    \]
    can be covered by a countable union of canonical neighborhoods $\{\cL^{k,\alpha}(g'_i, \Sigma'_i; \Lambda'_i, \kappa'_i)\}^\infty_{i= 1}$ where for each $i\geq 1$ there holds $\widehat{\ind}_{\tau}(L_{\Sigma'_i, g'_i}) < 0$; here we have set $\kappa'_i=\kappa(g'_i,\Sigma'_i,\Lambda'_i)$ for $\Lambda'_i=\Lambda_j$ when $(g'_i,\Sigma'_i)\in \cL^{k,\alpha}(g_j, \Sigma_j; \Lambda_j, \kappa_j)$.

    Finally, by Theorem \ref{Thm_Local Sard Thm}, the subset of metrics
    \begin{multline*}
       \{g \in \cG^{k, \alpha} : \forall (g, \Sigma) \in \cM^{k, \alpha}_n(M), \ \widehat{\ind}_{\tau}(L_{\Sigma,g}) \geq 0\}  \\
         = \cG^{k, \alpha} \setminus \{g \in \cG^{k, \alpha} : \exists (g, \Sigma) \in \cM^{k, \alpha}_n(M), \ \widehat{\ind}_{\tau}(L_{\Sigma,g}) < 0\}\\
         = \cG^{k, \alpha} \setminus \Pi(\cup_{i \geq 1} \cL^{k,\alpha}(g'_i, \Sigma'_i; \Lambda'_i, \kappa'_i))
         = \cap_{i \geq 1} \left(\cG^{k, \alpha}(g'_i, \Sigma'_i; \Lambda'_i, \kappa'_i)\right)\,
    \end{multline*}
    is a residual set in the Baire category sense; equivalently, said otherwise, the set of Riemannian metrics $\{g \in \cG^{k, \alpha} : \exists (g, \Sigma) \in \cM^{k, \alpha}_n(M), \widehat{\ind}_{\tau}(L_{\Sigma,g}) < 0\}$ is meagre. This completes the proof of Theorem \ref{thm:generic_perturbation} for $C^{k, \alpha}$-metrics.

\begin{center}

$\boxed{\text{Step 2: transition to smooth metrics.}}$

\end{center}

We apply Lemma \ref{lem:MetricSpaces} with $k_0=7$, taking for every $k\geq k_0$ the metric space $(X_k,d_k)$ to be $\cG^{k,\alpha}(M)$ with the metric induced by its natural norm (cf. Definition \ref{def:Pairs}) and letting 
\[
G=\{g \in \cG^{k_0, \alpha} : \forall (g, \Sigma) \in \cM^{k_0, \alpha}_n(M), \ \widehat{\ind}_{\tau}(L_{\Sigma,g}) \geq 0\}
\]
Note that Step 1 ensures that $G$ is a $G_\delta$ subset in $\cG^{k_0,\alpha}(M)$ and $G\cap \cG^{k,\alpha}(M)$ is dense in $\cG^{k,\alpha}$ for every $k\geq k_0$. Hence, the smooth version of the theorem follows at once.
\end{proof}

  \section{Proof of the main theorems}\label{sec:ProofMain}

\begin{proof}[Proof of Theorem \ref{Thm:GenReg,IsolatedSing}]
Straightforward by combining Theorem \ref{thm:Count}, Theorem \ref{thm:generic_perturbation} and the basic Morse index estimate recalled in Remark \ref{rem:BasicIndexEst}.
\end{proof}

Moving on, let us now see how a suitable area bound implies a definite structure of the singular set of a stationary integral varifold. Recall that we let $\mathbb{S}^d$ denote the round unit sphere in $\R^{d+1}$, and $A_d$ be its $d$-dimensional Hausdorff measure (in particular $A_3=2\pi^2$); if 
$\omega_d$ stands for the Lebesgue measure of the unit ball in $\R^d$ there holds $\omega_d=A_{d-1}/d$. 

\begin{proposition}\label{pro:SmallMass=>Isol}
  For every $\varepsilon>0$, there exists a neighborhood $\scN(\varepsilon)$ of the round metric on $S^4$ such that for every $g\in \scN(\varepsilon)$, every mod 2 cyclic $g$-stationary integral $3$-varifold $V$ with total mass $\leq 2A_3-\varepsilon$ has only strongly isolated singularities.   
\end{proposition} 
 
We must first prove the following simple density bound.

\begin{lemma}\label{lem:MonotIneq}
Let $V\subset\mathbb{S}^n$ be a stationary integral $(n-1)$-varifold, let $p\in \text{spt}(V)$ and let $\Theta_V(p)$ denote the density of $V$ at $p$; that is to say: $\Theta_V(p)=\Theta_V(p,0)$. Then there holds:
\[
\Theta_V(p)\leq \frac{\|V\|}{A_{n-1}}.
\]
\end{lemma}

\begin{proof}
Let $\mbfC$ denote the stationary integral varifold corresponding to the ``cone over $V$'' in $\R^{n+1}$, i.\,e. $\mbfC:=\mathbf{0} \conetimes V$. By the monotonicity formula for stationary integral varifolds one has:
\[
\frac{\|V\|}{A_{n-1}}=\lim_{r\to\infty}\frac{\|\mbfC\|(B_r(0))}{\omega_n r^n}=\lim_{r\to\infty}\frac{\|\mbfC\|(B_r(p))}{\omega_n r^n}\geq \lim_{r\to 0}\frac{\|\mbfC\|(B_r(p))}{\omega_n r^n}=\lim_{r\to 0}\frac{\|V\|(B_r(p))}{\omega_{n-1} r^{n-1}} =\Theta_V(p).
\] \end{proof}

In the case when $V$ is associated to a closed embedded minimal hypersurface in $\mathbb{S}^{n+1}$ the density at each point is of course unitary. Hence:

\begin{corollary}\label{cor:LeastArea}
For any $n\geq 2$ the equatorial $n$-dimensional hypersphere is the (unique) element of least area among all closed embedded minimal hypersurfaces in $\mathbb{S}^{n+1}$.
\end{corollary}

\begin{proof}
The (weak) inequality is implied by Lemma \ref{lem:MonotIneq}. By inspecting the previous proof, we see at once that equality can only occur if the cone in question splits off the line passing through the point $p$ and the origin. Iterating the argument by downward induction we must conclude that $\mbfC$ is in fact an hyperplane through the origin in $\R^{n+1}$, hence the claim.
\end{proof}

\begin{lemma}\label{lem:AreaLowerBd,NonIsolatedSing}
  Let $\mbfC$ be an irregular, mod 2 cyclic $3$-dimensional stationary integral cone in $\R^4$. Then \[
  \Theta_{\mbfC}(\orig)\geq 2.
  \]
\end{lemma}
\begin{proof}
  Let $x\in \SSp^3\cap \Sing(\mbfC)$, and let $\mbfC_x$ be a tangent cone of $\mbfC$ at $x$ that is \emph{not} regular in the sense of Definition \ref{def:MSI}. The case when $\mbfC_x$ has multiplicity at least two is clear; so assume instead that $\mbfC_x$ has unit multiplicity but non-smooth link. It is well-known that, in this case, $\mbfC_x$ splits: there exists an isometry (say $\Phi$) of $\R^4$   sending $x$ to $(0,0,0,1)$, such that $\Phi_*(\mbfC_x) = \R\times \mbfC_x'$ for some non-trivial $2$-dimensional mod 2 cyclic stationary integral cone $\mbfC_x'$ in $\R^3$. Hence, appealing again to the standard monotonicity formula
  \[
    \Theta_\mbfC(\orig) = \lim_{r\to \infty} \frac{\|\mbfC\|(B_r(x))}{\omega_3 r^3} \geq \Theta_{\mbfC}(x) = \Theta_{\mbfC_x}(\orig) = \Theta_{\mbfC_x'}(\orig) \geq 2.
  \]
  where the last inequality follows from the classification of non-trivial $2$-dimensional mod 2 cyclic stationary integral cones in $\R^3$; ultimately that relies on the ``structure theorem'' for 1-dimensional stationary varifolds given in \cite{AA76}, plus an \emph{ad hoc} argument ruling out triple junctions using the mod 2 cyclicity assumption, as can be found e.\,g. in \cite[Lemma A.3]{MarNev14}.
\end{proof}

\begin{proof}[Proof of Proposition \ref{pro:SmallMass=>Isol}]
  We assume for the sake of a contradiction that for some $\eps>0$, there exists smooth metrics $g_j$ on $S^4$ smoothly converging to the round metric and some mod 2 cyclic $g_j$-stationary integral varifolds $V_j$ with not only strongly isolated singularities and total mass $\leq 2A_3-\eps$. Let $p_j\in \Sing(V_j)$ be a point having a tangent cone $\mbfC_j$ that is not regular. Then by Lemma \ref{lem:AreaLowerBd,NonIsolatedSing}, \[
    \Theta_{V_j}(p_j) = \Theta_{\mbfC_j}(\orig) \geq 2.
  \]
  By Allard's Compactness Theorem (see e.\,g. \cite[Theorem 42.7]{Sim83book}), $V_j$ subconverges to some stationary integral varifold $V_\infty$ under the round metric, and clearly $p_j$ subconverges to some $p_\infty \in \Sing(V_\infty)$. In particular, \[
    \|V_\infty\|(\mathbb{S}^4) \leq 2A_3 - \eps,
  \] 
  while by the upper semi-continuity of density under varifold convergence there holds \[
    \Theta_{V_\infty}(p_\infty) \geq \Theta_{V_j}(p_j) \geq 2.
  \]
  Then by Lemma \ref{lem:MonotIneq}, \[
    \|V_\infty\|(\mathbb{S}^4)\geq A_3\Theta_{V_\infty}(p_\infty) \geq 2A_3 \,
  \]
  which contradicts the previous bound and thereby completes the proof.
\end{proof}

\begin{proof}[Proof of Corollary \ref{cor:Smooth}]
Straightforward, by combining Theorem \ref{Thm:GenReg,IsolatedSing} with Proposition \ref{pro:SmallMass=>Isol}.
\end{proof}

\begin{proof}[Proof of Corollary \ref{cor:MainArea}]

For given $\eps>0$ let $\mathscr{N}(\varepsilon)$ as afforded by Corollary \ref{cor:Smooth} and let $g\in\mathscr{N}(\varepsilon)$ be a generic metric (cf. Remark \ref{rem:BaireWhite}). Towards a contradiction, assume the class of closed, embedded minimal hypersurfaces of area less than $4\pi^2-\varepsilon$ contains a sequence $\left\{\Sigma_k\right\}_{k\geq 1}$ with pairwise distinct elements. 

Allard's Compactness Theorem implies that we can extract a subsequence (which we shall not rename) converging to an integral stationary varifold $V$; note that the mass of $V$ is again bounded from above by the same threshold $4\pi^2-\varepsilon$, and furthermore such a stationary varifold is mod 2 cyclic (by appealing e\,g. to \cite[Theorem 3.3]{Whi09}). Thanks to Lemma \ref{lem:MonotIneq} we can thereby derive that the density of $V$ at any point (of its support) is strictly below 2, in fact bounded from above by $2-\frac{\eps}{2\pi^2}$.

Hence, we have by Corollary \ref{cor:Smooth} that $V$ is actually smooth at all points, i.\,e. it is a smooth, closed, embedded minimal hypersurface in $\mathbb{S}^{n+1}$, say $\Sigma$, and the aforementioned density bound implies a posteriori that there does occur smooth graphical convergence of $\Sigma_k$ to $\Sigma$ with multiplicity one. But then, appealing to Sharp's analysis in \cite{Sha17}, the minimal hypersurface $\Sigma$ would be degenerate (i.\,e. it would come with at least a non-trivial Jacobi field), a contradiction. 
\end{proof}

\begin{proof}[Proof of Corollary \ref{cor:MainHsiang}]

Straightforward from Corollary \ref{cor:MainArea}.
\end{proof}

\appendix

\section{The minimal surface system and transfer of normal sections}\label{app:MSE}
  \begin{proposition} \label{Prop_MSE_Main}
    There exist $\kappa_1 = \kappa_1(N)\in (0, 1/4)$ and $C=C(N)>1$ such that the following statement holds. Let:
    \begin{enumerate} [label=(\roman*)]
      \item $g, g'$ be  $C^{5}$ Riemannian metrics on $\BB^N(4)$ with $\|g-g_\eucl\|_{C^{5}}, \|g'-g_\eucl\|_{C^{5}}\leq \kappa_1$;
      \item $f_1, f_2\in C^2(\BB^N(4))$ be functions such that $\|f_1-1\|_{C^2}, \|f_2-1\|_{C^2}\leq \kappa_1$;
      \item $\Sigma$ be a $g$-minimal $\kappa_1$-$C^{5}$ graph in $\RR^N$ over $\BB^n(2)\times \{\orig\}$;
      \item $\mbfV$ be the normal bundle of $\Sigma$ in $(\RR^N, g)$ with induced connection $\nabla^\perp$ from $g$;
      \item $v, v_1, v_2\in C^2(\Sigma; \mbfV)$ with $\|v\|_{C^2}, \|v_1\|_{C^2}, \|v_2\|_{C^2}\leq \kappa_1$.
    \end{enumerate}
(Note that in the items above and throughout the following proof all norms are understood on the respective domains of definition of the corresponding tensors: such domains are not explicitly indicated for the sake of readability.)
    
    Then we have:
    \begin{enumerate} [label=(\Roman*)]
      \item\label{Item_Phi_v: Sigma to graph(v)} $\Phi_v: x\mapsto \exp^g_x(v(x))$ is a $C^2$ embedding of $\Sigma$ into $\RR^N$;
      \item\label{Item_Area Density cA} there exists a $C^3$ function $\cA^{g'} = \cA^{g'}(x, z, \xi)$ such that for every $x\in\Sigma$ the map $x\mapsto  \cA^{g'}(x, \cdot, \cdot)$, defined 
        on the $\kappa_1$-ball centered at $\orig$ in $(\mbfV\oplus (T^*\Sigma\otimes \mbfV))|_x$, 
         satisfies (for the area of $\Phi_v(\Sigma)$ in metric $g'$)
 \[
        \|\Phi_v(\Sigma)\|_{g'} = \int_\Sigma \cA^{g'}(x, v, \nabla^\perp v)\ d\|\Sigma\|_g(x)\,,
      \]
     the local bound $\|\cA^{g'}\|_{C^{3}}\leq C$, and, in local coordinates for $\Sigma$, there holds \[
        \cA^{g'}(x, v, \nabla^\perp v) = \sqrt{\frac{\det[(\Phi_v^* g')_{ij}]}{\det[g_{ij}]}} \,.
      \]
      In particular, $\Phi_v(\Sigma)$ is $g'$-minimal if and only if \[
          \nabla^\perp \cdot \partial_\xi\cA^{g'}(x, v, \nabla^\perp v) - \partial_z\cA^{g'}(x, v, \nabla^\perp v) = 0\,;
      \]
      \item \label{Item_MSE for g' minimal graph}

      if $\Phi_v(\Sigma)$ is minimal with respect to the metric $g'$, then the equation satisfied by $v$ can be written in any of the following two forms:
      \begin{align}
\nabla^\perp\cdot(B_0(x)\nabla^\perp v)
        +  B_1(x) v = k \label{Equ_MSE_Form for reg}
      \end{align}
      where $B_0\in C^1(\Sigma; \End(T^*\Sigma\otimes \mbfV))$
      and $B_1\in C^0(\Sigma; \End(\mbfV))$ satisfy 
      \begin{align*}
        |B_0 - \id| 
        \leq C\left( |v| + |\nabla^\perp v| \right) \,  
      \end{align*}
      for $\id$ the identity map on the bundle $\End(T^*\Sigma\otimes \mbfV)$, and 
      \begin{align*}
        \|B_0\|_{C^1} 
        + \|B_1\|_{C^0} \leq C\,, & &
        \|k\|_{C^0} \leq C\|g' - g\|_{C^2} \,;
      \end{align*}
     \begin{align}
        L_{\Sigma, g} v + \nabla^\perp\cdot b_0(x) + b_1 = k,  \label{Equ_MSE_Form}
      \end{align}
      where $b_0\in C^0(\Sigma;T^*\Sigma\otimes \mbfV)$, $b_1, k\in C^0(\Sigma;\mbfV)$ and 
      \[
      |b_0| + |b_1|  \leq C\left( |v| + |\nabla^\perp v| \right)^2;
      \]
      \item\label{Item_MSE for difference of (1+f_i)g minimal graph} if for $i=1, 2$, $\Phi_{v_i}(\Sigma)$ is minimal with respect to the metric $g_i:= f_ig$, then $w:= v_2-v_1$ satisfies
      \begin{align*}
        L_{\Sigma, g} w - \frac{n}{2}(\nabla (f_2-f_1))^{\perp_{\Sigma, g}} + \nabla^\perp\cdot \tilde b_0(x) + \tilde b_1 = 0,  
      \end{align*}
      where for $\tilde b_0\in C^0(\Sigma;T^*\Sigma\otimes \mbfV)$, $\tilde b_1 \in C^0(\Sigma;\mbfV)$ there hold the pointwise estimates
      \begin{align*}
        |\tilde b_0| + |\tilde b_1| & \leq C\left( \sum_{i=1}^2 |v_i| + |\nabla^\perp v_i| + \|f_i-1\|_{C^2}\right)\cdot \left( |w| + |\nabla^\perp w| + \|f_2-f_1\|_{C^2} \right) \,.
      \end{align*}
    \end{enumerate}
  \end{proposition}
  \begin{remark}\label{rem:Divergence}
    In general, given a Riemannian manifold $(\Sigma, g)$ and a vector bundle $\mbfV$ with a fiberwise specified metric $h$ and a metric compatible connection $\nabla$, for every $b\in C^1(\Sigma; T^*\Sigma\otimes \mbfV)$, we can define $\nabla\cdot b\in C^0(\Sigma; \mbfV)$ by requiring that for every $\varphi\in C^1_c(\Sigma;\mbfV)$ there holds \[
      \int_\Sigma \langle \nabla \cdot b, \varphi\rangle_h\ d\|\Sigma\|_g = -\int_\Sigma \langle b, \nabla\varphi\rangle_{g, h}\ d\|\Sigma\|_g \,.
    \]
    It is easy to check that this is equivalent to the following explicit expression: \[
    \langle\nabla\cdot b, \varphi\rangle_h := -d^*\langle b, \varphi\rangle_h - \langle b, \nabla \varphi\rangle_{g, h} \,,
    \]  
    where $d^*$ is the codifferential on $1$-forms on $(\Sigma, g)$. 
  \end{remark}
  \begin{proof}
    \ref{Item_Phi_v: Sigma to graph(v)} follows from the tubular neighborhood theorem, while \ref{Item_Area Density cA} follows from a direct, rather standard calculation, as we now indicate. Let $\{x^i\}$, where $i=1,\ldots, n$ be the coordinates on $\Sigma$ corresponding to the graphical parametrization over $\BB^n(2)\times \{\orig\}$, and let $\{s_I\}_{I=1}^{N-n}$ be a local orthonormal frame of $\mbfV$, so that one has the parametrizations of $\mbfV$ and $T^*\Sigma\otimes \mbfV$ respectively given by
    \begin{align*}
      \BB^n\times \RR^{N-n} & \to \mbfV\,, & & (x, z)\mapsto (x, \sum_I z^I s_I)\,, \\
      \BB^n\times \RR^{n\times(N-n)} & \to T^*\Sigma\otimes \mbfV \,, & & (x, \xi)\mapsto (x, \sum_{I,i}\xi^I_i\, dx^i\otimes s_I)
    \end{align*}
    Hence, the $C^4$ map, $\mbfV\to \RR^N$, $(x, z)\mapsto \exp_x^g(z)$, pulls back metric $g'$ to a $C^3$ $N\times N$-matrix valued function $[g'_{ab}]$ on $\BB^n\times \RR^{N-n}$; if  
    $v = \sum v^Is_I\in C^2(\Sigma)$, then $\Phi_v^*\, g' = g'_{ij}|_{(x, v)} + g'_{iJ}|_{(x, v)}\partial_jv^J
    + g'_{Ij}|_{(x, v)}\partial_iv^I
    +g'_{IJ}|_{(x, v)}\partial_iv^I\partial_jv^J$, where it is understood that $1\leq i,j\leq n$ and $1\leq I,J\leq N-n$, with sums over repeated indices. Therefore, we set \[
      \cA^{g'}(x, z, \xi) = \sqrt{\det[g_{ij}|_{(x, 0)}]}^{-1}\cdot \sqrt{\det \left[g'_{ij}|_{(x, z)} 
      + g'_{iJ}|_{(x, z)}\xi_j^J
    + g'_{Ij}|_{(x, z)}\xi_i^I
         + g'_{IJ}|_{(x, z)}\xi_i^I\xi_j^J \right]}\,,
    \] 
    which is a $C^3$ function in $(x, z)$ and real analytic in $\xi$. 
    
    We now focus on \ref{Item_MSE for g' minimal graph} and \ref{Item_MSE for difference of (1+f_i)g minimal graph}.   To start with, consider a variation $v_t:= v+t\phi$, where $\phi\in C^2(\Sigma; \mbfV)$. Let then $\vec\phi(x, t):= \partial_t\Phi_{v_t}(x)\in T_{\Phi_{v_t}(x)}\RR^N$. Note that when $v(x)=0$, $\vec\phi(x, 0) = \phi(x)$. Also, for each fixed $x$, $t\mapsto \Phi_{v_t}(x)$ is a geodesic under metric $g$ with constant speed, and $\vec\phi(x, t)$ is its velocity vector field, hence it is parallel along $t\mapsto \Phi_{v_t}(x)$.

   As described before, we will work in local coordinates $\left\{x^i\right\}$ of $\Sigma$, and set $g^t_{ij}:= (\Phi_{v_t}^*g)_{ij}$ and $\partial_i^t:= \partial_i\Phi_{v_t}$. Then we have
    \begin{align*}
      \partial_t g^t_{ij} & = g(\partial_t \partial_i\Phi_{v_t}, \partial_j\Phi_{v_t} )|_{\Phi_{v_t}} + g(\partial_i\Phi_{v_t}, \partial_t \partial_j\Phi_{v_t})|_{\Phi_{v_t}} = g(\nabla^g_{\partial_i^t}\vec\phi, \partial_j^t )|_{\Phi_{v_t}} + g(\partial_i^t, \nabla^g_{\partial_j^t}\vec\phi)|_{\Phi_{v_t}}\,; \\
      \partial_{tt}^2 g^t_{ij} & = g(\nabla^g_{\vec\phi}\,\nabla^g_{\partial_i^t}\vec\phi, \partial_j^t) + g(\partial_i^t, \nabla^g_{\vec\phi}\,\nabla^g_{\partial_j^t}\vec\phi) + 2g(\nabla^g_{\partial_i^t}\vec\phi, \nabla^g_{\partial_j^t}\vec\phi )  \\
      & = -2\Riem_g(\vec\phi, \partial_i^t, \partial_j^t, \vec\phi) + 2g(\nabla^g_{\partial_i^t}\vec\phi, \nabla^g_{\partial_j^t}\vec\phi ) \,.
    \end{align*}
    Denote for simplicity $\cA^g[v_t]:= \cA^g(x, v_t, \nabla^\perp v_t)$, $h^t_{ij}:= g(\nabla^g_{\partial_i^t}\vec\phi, \partial_j^t )$. Then we derive, 
    \begin{align*}
      \frac{d}{dt}\cA^g[v_t] & = \frac12\cA^g[v_t]\cdot \partial_t g_{ij}^t \cdot(g^t)^{ij} = \cA^g[v_t]\cdot h^t_{ij}\cdot (g^t)^{ij}\,; \\
      \frac{d^2}{dt^2}\cA^g[v_t] & = \cA^g[v_t]\left[ \left(h^t_{ij}\cdot (g^t)^{ij}\right)^2 - 2h^t_{ij}h^t_{kl}(g^t)^{ik}(g^t)^{jl} -\Riem_g(\vec\phi, \partial_i^t, \partial_j^t, \vec\phi) + g(\nabla^g_{\partial_i^t}\vec\phi, \nabla^g_{\partial_j^t}\vec\phi) \right]\,.
    \end{align*}
    On the other hand, 
    \begin{align*}
      \frac{d}{dt}\cA^g[v_t] & = \partial_z\cA^g(x, v_t, \nabla^\perp v_t)\cdot \phi + \partial_\xi\cA^g(x, v_t, \nabla^\perp v_t)\cdot \nabla^\perp\phi\,; \\
      \frac{d^2}{dt^2}\cA^g[v_t] & = \partial^2_{zz}\cA^g(x, v_t, \nabla^\perp v_t)(\phi, \phi) + 2\partial^2_{z\xi}\cA^g(x, v_t, \nabla^\perp v_t)(\phi, \nabla^\perp\phi) \\
      & + \partial^2_{\xi\xi}\cA^g(x, v_t, \nabla^\perp v_t)(\nabla^\perp\phi, \nabla^\perp\phi)\,.     
    \end{align*}
    By comparing the coefficients in front of $\phi, \nabla^\perp\phi$ terms, we can find that,
    \begin{align}
      \cA^g(x, 0, 0) & = 1\,, &  \partial_z\cA^g(x, 0, 0) & = 0\,, & \partial_\xi\cA^g(x, 0, 0) & = 0\,; \label{Equ_MSE_A, A_z, A_xi at (0,0)} \\
      \partial_{\xi\xi}\cA^g(x, 0, 0)(\zeta, \zeta) & = |\zeta|^2\,, & \partial^2_{z\xi}\cA^g(x, 0, 0) & = 0\,,  & \partial^2_{zz}\cA^g(x, 0, 0) & = -\II_{\Sigma, g}^2 - \cR_{\Sigma, g} \label{Equ_MSE_A_zz, A_zxi, A_xixi at (0,0)}
    \end{align}
    where $\II_{\Sigma, g}$ denotes the second fundamental form of $\Sigma$ under $g$, and for any section $\varphi$ of  $\mbfV$,
    \begin{align*}
       \II_{\Sigma, g}^2 (\varphi, \varphi) := (\II_{ij}\cdot \varphi) (\II_{kl}\cdot \varphi) g^{ik}g^{jl}\,, & &
       \cR_{\Sigma, g}(\varphi, \varphi) := \Riem_g(\varphi, \partial_i, \partial_j, \varphi)g^{ij}\,. 
    \end{align*}
 Therefore, to prove \eqref{Equ_MSE_Form for reg}, set \[
      B_0(x):= \int_0^1 \partial^2_{\xi\xi}\cA^g(x, tv, t\nabla^\perp v)\ dt \,, \qquad B_0'(x):= \int 2\partial^2_{\xi z}\cA^g(x, tv, t\nabla^\perp v)\ dt
    \] \[
      B_1(x):= -\int_0^1 \partial^2_{zz}\cA^g(x, tv, t\nabla^\perp v)\ dt + \nabla^\perp\cdot \left(\int_0^1 \partial^2_{z\xi}\cA^g(x, tv, t\nabla^\perp v)\ dt\right) \,.
    \]
    We see that \[
      \partial_\xi\cA^g[v] = B_0(x)\nabla^\perp v + \frac12B_0'(x)v\,, \qquad
      \partial_z\cA^g[v] = - B_1(x)v + \frac12 \nabla^\perp \cdot (B_0'(x)v)\,
    \] 
    and $\|B_0\|_{C^1}, \|B_0'\|_{C^1}, \|B_1\|_{C^0}\leq C$. Moreover, since $\|\cA^g\|_{C^3}\leq C$, by \eqref{Equ_MSE_A_zz, A_zxi, A_xixi at (0,0)} we have
    \begin{align*}
      |B_0 - \id| \leq \int_0^1 |\partial^2_{\xi\xi}\cA^g(x, tv, t\nabla^\perp v) - \partial^2_{\xi\xi}\cA^g(x, 0, 0)|\ dt \leq C(|v|+|\nabla^\perp v|).
    \end{align*}
    Also, set \[
      k:= \nabla^\perp\cdot(\partial_\xi \cA^g[v] - \partial_\xi\cA^{g'}[v]) + (\partial_z \cA^g[v] - \partial_z\cA^{g'}[v]) \,,
    \]
    clearly $\|k\|_{C^0}\leq C\|g'-g\|_{C^2}$. Therefore, the form of equation \eqref{Equ_MSE_Form for reg} with the corresponding estimates on the coefficients are established.

    To prove \eqref{Equ_MSE_Form}, recalling that for any $y\in C^2([0,1])$ there holds 
    \[
      |y(1) - y(0) - y'(0)| \leq \frac12\sup_{t\in [0, 1]}|y''(t)|\,,
    \]
    by \eqref{Equ_MSE_A, A_z, A_xi at (0,0)}, \eqref{Equ_MSE_A_zz, A_zxi, A_xixi at (0,0)} and the fact that $\|\cA^g\|_{C^3}\leq C$, we have the pointwise error estimate 
    \begin{align*}
      \left|\partial_\xi\cA^g[v] - \nabla^\perp v\right| + \left|\partial_z\cA^g[v] + (\II_{\Sigma, g}^2 + \cR_{\Sigma, g})(v, \cdot)\right| \leq C(|v| + |\nabla^\perp v|)^2 \,. 
    \end{align*}
  Such bounds 
    allow us to complete the proof of \ref{Item_MSE for g' minimal graph} once we just set
    \[
    b_0 := \partial_\xi\cA^g[v] - \nabla^\perp v, \ \ b_1:= \partial_z\cA^g[v] + (\II_{\Sigma, g}^2 + \cR_{\Sigma, g})(v, \cdot).
    \]

  Lastly, to prove \ref{Item_MSE for difference of (1+f_i)g minimal graph}, first notice that by \ref{Item_Area Density cA}, for $i=1,2$ there holds
    \begin{align*}
      \cA^{g_i}(x, v, \nabla^\perp v) = ((f_i)^{n/2}\circ \Phi_v)\ \cA^g(x, v, \nabla^\perp v)\,.
    \end{align*}
    For every $\varphi\in C^2(\Sigma; \mbfV)$, if we let $v^i_{t}:= v_i+t\varphi$ and $\vec\varphi_i(x, t):= \partial_t\Phi_{v^i_{t}}(x)$, then, \[
      \frac{d}{dt}\cA^{g_i}[v^i_{t}] = ((f_i)^{n/2}\circ\Phi_{v^i_{t}})\ \left(\partial_z \cA^g[v^i_{t}]\cdot \varphi + \partial_\xi\cA^g[v^i_{t}]\cdot \nabla^\perp \varphi\right)
      + (\nabla^g(f_i^{n/2})\cdot \vec\varphi_i(\cdot, t))\ \cA^g[v^i_{t}] \,,
    \]
 where it is tacitly understood that the symbol $\cdot$ refers to products of vectors (and, more generally, of tensors of any type) with respect to the metric $g$; hence we conclude that  
    \begin{align*}
      \partial_\xi\cA^{g_i}[v_i] & = ((f_i)^{n/2}\circ\Phi_{v_i})\ \partial_\xi\cA^g[v_i] \,; \\
      \partial_z\cA^{g_i}[v_i]\cdot \varphi & = ((f_i)^{n/2}\circ\Phi_{v_i})\ \partial_z\cA^g[v_i]\cdot \varphi + (\nabla^g(f_i^{n/2})\cdot \vec\varphi_i(x, 0))\ \cA^g[v_i]\,.
    \end{align*}
    By applying again \eqref{Equ_MSE_A, A_z, A_xi at (0,0)}, \eqref{Equ_MSE_A_zz, A_zxi, A_xixi at (0,0)}, and recalling that $\|\cA^g\|_{C^3}\leq C$, we find
    \begin{align*}
      & \left|\partial_\xi\cA^{g_2}[v_2] - \partial_\xi\cA^{g_1}[v_1] - \nabla^\perp w\right| \\
      +\ & \left|\partial_z\cA^{g_2}[v_2] - \partial_z\cA^{g_1}[v_1] - \frac{n}{2}(\nabla^g(f_2-f_1))^{\perp_{\Sigma, g}} + (\II_{\Sigma, g}^2 + \cR_{\Sigma, g})(w, \cdot)\right| \\
      & \quad \leq C\left( \sum_{i=1}^2 |v_i| + |\nabla^\perp v_i| + \|f_i-1\|_{C^2}\right)\cdot \left( |w| + |\nabla^\perp w| + \|f_2-f_1\|_{C^2} \right)\,.
    \end{align*}
    This proves \ref{Item_MSE for difference of (1+f_i)g minimal graph}. 
  \end{proof}

The utility of \eqref{Equ_MSE_Form for reg}, compared to \eqref{Equ_MSE_Form}, relies on the fact that it allows to easily upgrade $L^2$ local bounds to $C^0$ bounds, as detailed in the following statement.

\begin{corollary}\label{cor:L2toC0}
    In the setting of the preceding proposition, if $\Phi_v(\Sigma)$ is minimal with respect to the metric $g'$ then we have the following estimate on $v$,
      \begin{align}
        \|v\|_{C^0(\Sigma\cap B^g(1))} \leq C(\|v\|_{L^2(\Sigma)} + \|g'-g\|_{C^2(\Sigma)})\,. \label{Equ_MSE_L^2 bd C^0}
      \end{align}
\end{corollary}

\begin{proof}
We know that the section $v$ satisfies the PDE \eqref{Equ_MSE_Form for reg}, which we will treat as a problem in divergence form with the term  $B_1(x) v$ brought to the right-hand side. Hence, standard $L^p$ estimates (see e.\,g. \cite[Theorem 3.35]{AmbCarMas18}) imply that for every $1<p< n$ and $1<r_1<r_2<3/2$, we have \[
      \|v\|_{L^{p^*}(B^g(r_1)\cap \Sigma)}  \leq C(p)\|v\|_{W^{1,p}(B^g(r_1)\cap \Sigma)} \leq C(p, r_1, r_2)(\|v\|_{L^p(B^g(r_2)\cap \Sigma)} + \|g'-g\|_{C^2(\Sigma)})
    \]
    where the first inequality follows from Sobolev embedding with $p^*:= pn/(n-p)$ and the second is the aforementioned elliptic estimate. That being said, \eqref{Equ_MSE_L^2 bd C^0} then follows from this bound (starting at $p=2$) with a bootstrap argument and the Sobolev embedding $W^{1,p}\hookrightarrow C^0$ for $p>n$.
\end{proof}
  \begin{proposition}\label{Prop_MSE_Section Transp}
    There exist $\kappa_2 = \kappa_2(N)\in (0, 1/4)$ and $C = C(N)>1$ such that the following statement holds.  Let $\kappa\in (0, \kappa_2)$ and
    \begin{enumerate} [label=(\roman*)]
      \item\label{Item_MSE_SecTrans_assum1} $g, g'$ be $C^4$ Riemannian metrics on $\BB^N(4)$ with $\|g-g_\eucl\|_{C^4}, \|g'-g_\eucl\|_{C^4}\leq \kappa_2$ and $\|g-g'\|_{C^4}\leq \kappa$;
      \item\label{Item_MSE_SecTrans_assum2} $\Sigma$ be a $g$-minimal $\kappa_2$-$C^4$ graph in $\RR^N$ over $\BB^n(2)\times \{\orig\}$;
      \item\label{Item_MSE_SecTrans_assum3} $v\in C^2(\Sigma; \mbfV)$ with $\|v\|_{C^2}\leq \kappa$; $\Phi_v$ be defined as in \ref{Item_Phi_v: Sigma to graph(v)} of Proposition \ref{Prop_MSE_Main}; 
      \item\label{Item_MSE_SecTrans_assum4} $\mbfV$ be the normal bundle of $\Sigma$ in $(\RR^N, g)$ with induced connection $\nabla^\perp$ from $g$; $\mbfV'$ be the normal bundle of  $\Sigma':=\Phi_v(\Sigma)$ in $(\RR^N, g')$ with induced connection $\nabla'^\perp$ from $g'$;
      \item\label{Item_MSE_SecTrans_assum5} $\mbfT^{\Sigma'}_{\Sigma,g}: \mbfV'\to \mbfV$ be the bundle maps defined in Definition \ref{def:Transfer},  which induces linear maps (still denoted by $\mbfT^{\Sigma'}_{\Sigma,g}$) for every integer $0\leq k\leq n$: \[
        C^0(\Sigma'; \Lambda^k\Sigma'\otimes\mbfV')\to C^0(\Sigma; \Lambda^k\Sigma\otimes\mbfV)\,.  \]
    \end{enumerate}
    Then,
    \begin{enumerate} [label=(\Roman*)]
      \item\label{Item_MSE_SecTrans_Concl1} for every $\beta'\in C^0(\Sigma'; \Lambda^k\Sigma'\otimes\mbfV')$, set $\beta:= \mbfT^{\Sigma'}_{\Sigma,g}(\beta')\in C^0(\Sigma; \Lambda^k\Sigma\otimes\mbfV)$; we have
      \begin{align*}
        (1-C\kappa)|\beta(x)| \leq |\beta'(\Phi_v(x))| & \leq (1+C\kappa)|\beta(x)|\,; 
      \end{align*}
      \item\label{Item_MSE_SecTrans_Concl2} for every $u'\in C^1(\Sigma'; \mbfV')$ and $u:= \mbfT^{\Sigma'}_{\Sigma,g}(u')$, \[
        \left|\mbfT^{\Sigma'}_{\Sigma,g}(\nabla'^\perp u')(x) - \nabla^\perp u(x) \right| \leq C\kappa|u(x)|\,;
      \]
      \item\label{Item_MSE_SecTrans_Concl3} there exist bundle endomorphisms $\cS:\mbfV\to \mbfV$, $\cT: T^*\Sigma\otimes\mbfV\to T^*\Sigma\otimes\mbfV$ such that, denoted by $\id$ the identity map on either bundle in question, there holds \[
        |\cS - \id| + |\cT - \id| \leq C\kappa
      \] 
      and for every $\beta'\in C^1(\Sigma'; T^*\Sigma'\otimes \mbfV')$ and $\beta:= \mbfT^{\Sigma'}_{\Sigma,g}(\beta')$, we have the pointwise estimate \[
        \left|\mbfT^{\Sigma'}_{\Sigma,g}(\nabla'^\perp \cdot \beta') - \cS(\nabla^\perp \cdot \cT(\beta)) \right| \leq C\kappa|\beta|\,.
      \] 
    \end{enumerate}
    In particular, for every $u'\in C^2(\Sigma'; \mbfV')$ and $b_0', b_1'\in C^1(\Sigma'; T^*\Sigma'\otimes \mbfV')$ such that \[
      - L_{\Sigma', g'} u' + \nabla'^\perp \cdot b_0' + b_1' = 0\,,
    \]
    there exists $b_0, b_1\in C^1(\Sigma; T^*\Sigma\otimes \mbfV)$ such that $u:= \mbfT^{\Sigma'}_{\Sigma,g}(u')$ solves,
    \begin{align}
      - L_{\Sigma, g} u + \nabla^\perp \cdot b_0 + b_1 = 0\,, \label{Equ_MSE_Jaco Oper}
    \end{align}
    and 
    \begin{align}
     \begin{split}
      & |b_0(x) - \mbfT^{\Sigma'}_{\Sigma,g}(b_0')(x)| + |b_1(x) - \mbfT^{\Sigma'}_{\Sigma,g}(b_1')(x)| \\
      & \quad \leq C\kappa\left(|b_0'(\Phi_v(x))| + |b_1'(\Phi_v(x))| + |u(x)| + |\nabla^\perp u(x)| \right)  \,. 
     \end{split} \label{Equ_MSE_Error est of Jac Oper under Transp}
    \end{align}
  \end{proposition}
  
  \begin{proof}
Throughout this proof, let us write $\mbfT:= \mbfT^{\Sigma'}_{\Sigma, g}$ for the sake of simplicity. Let then
    \begin{itemize}
      \item $\bfx': \Sigma'\to \RR^n$ be the projection onto the first $n$-factors, and $\bfx:= \bfx'\circ \Phi_v$;
      \item For every $x'\in \Sigma'$ and $1\leq j\leq N-n$, let \[
        e'_j(x'):= (0, \dots, 0, \underbrace{1}_{n+j\text{-th}}, \dots, 0)^{\perp_{T_{x'}\Sigma', g'}}\,
      \]
      and set $e_j(x):= \mbfT(e_j')(x)$.
    \end{itemize}
    It is straightforward to note that we can take $\kappa_2(N)$ small enough that if \ref{Item_MSE_SecTrans_assum1}-\ref{Item_MSE_SecTrans_assum5} hold, then: 
    \begin{itemize}
      \item $\bfx' = (\bfx'^1, \dots, \bfx'^n)$ and $\bfx = (\bfx^1, \dots, \bfx^n)$ are coordinate systems of, respectively, $\Sigma'$ and $\Sigma$, such that the metric $g'$ and $g$ restricted to $\Sigma'$ and $\Sigma$ and their Christoffel symbols satisfy 
      \begin{align}
         |g'_{ij} - \delta_{ij}|,\ |g_{ij} - \delta_{ij}|,\ |\Gamma'{}_{ij}^k|,\ |\Gamma_{ij}^k| \leq C\kappa_2\,; & & 
         |\Phi_v^*g'_{ij} - g_{ij}| \leq C\kappa \,; \label{Equ_MSE_Metric and Conn close Euc}          
      \end{align}
      \item $\{e'_j\}$ and $\{e_j\}$ are frames of, respectively, $\mbfV'$ and $\mbfV$, such that the induced metrics $h'_{ij}:= g'(e'_i, e'_j)$ and $h_{ij}:= g(e_i, e_j)$ satisfy 
      \begin{align}
         |h'_{ij} - \delta_{ij}|,\ |h_{ij} - \delta_{ij}| \leq C\kappa_2\,; & &
         |h'_{ij}\circ\Phi_v - h_{ij}| \leq C\kappa \,, \label{Equ_MSE_Bundle metric close Euc}  
      \end{align}
      the connection forms $\omega'{}^i_j$ and $\omega^i_j$ of $\nabla'^\perp$ and $\nabla^\perp$ under these two frames satisfy 
      \begin{align}
         |\omega'{}^i_j|_{g'},\ |\omega^i_j|_g \leq C\kappa_2\,; & &
         |\Phi_v^* \omega'{}^i_j - \omega^i_j|_g \leq C\kappa\,. \label{Equ_MSE_Conn form close 0}
      \end{align}
    \end{itemize} 
    Thus for every $\beta' = \beta'{}^j\otimes e'_j\in C^0(\Sigma'; \Lambda^k\Sigma'\otimes \mbfV')$, we have \[
      \beta:= \mbfT^{\Sigma'}_{\Sigma, g}(\beta') = \Phi_v^* \beta'{}^j\otimes e_j\,.
    \]
    In particular, combined with \eqref{Equ_MSE_Metric and Conn close Euc} and \eqref{Equ_MSE_Bundle metric close Euc}, this proves \ref{Item_MSE_SecTrans_Concl1}.

    Let then $u' = u'{}^je'_j\in C^1(\Sigma'; \mbfV')$ and $u^j:= u'{}^j\circ\Phi_v$ (so that by linearity $u:= \mbfT(u') = u^je_j$). Then \ref{Item_MSE_SecTrans_Concl2} follows directly from \eqref{Equ_MSE_Metric and Conn close Euc}-\eqref{Equ_MSE_Conn form close 0} and the identity
    \begin{align*}
        \mbfT(\nabla'^\perp u') = \mbfT \left((du'{}^j + u'{}^k\omega'{}^j_k)\otimes e'_j \right) = (du^j + u^k\Phi_v^*\omega'{}^j_k)\otimes e_j = \nabla^\perp u + u^k\cdot(\Phi_v^*\omega'{}^j_k - \omega^j_k)\otimes e_j\,.
    \end{align*}

    To prove \ref{Item_MSE_SecTrans_Concl3}, first recall that for $\beta' = \beta'{}^j\otimes e'_j = \beta'{}^j_p\, d\bfx'{}^p\otimes e'_j\in C^1(\Sigma; T^*\Sigma'\otimes \mbfV')$, keeping in mind Remark \ref{rem:Divergence},
    \begin{align*}
        \nabla'^\perp\cdot \beta' & = -\left( d^*(\beta'{}^jh'_{jk}) + g'(\beta'{}^i, \omega'{}^j_k)h'_{ij} \right)h'{}^{kl}e'_l \\
        & = \left( \partial_{q}(\beta'{}^j_p\cdot g'{}^{pq} h'_{jk}) - \beta'{}^j_p h'_{jk}\cdot \partial_q(g'{}^{pq}) - g'{}^{pq}\Gamma'{}^r_{pq}(\beta'{}^j_r h'_{jk}) - g'(\beta'{}^i, \omega'{}^j_k)h'_{ij} \right)h'{}^{kl}e'_l\,.
    \end{align*}
    Hence, having set $\beta^j_p := \beta'{}^j_p\circ \Phi_v$ (so that $\beta^j := \beta^j_p\, d\bfx^p = \Phi_v^*\beta'{}^j$ and $\beta:= \beta^j\otimes e_j = \mbfT(\beta')$), there holds
    \begin{align*}
        \mbfT(\nabla'^\perp\cdot \beta') 
        & = \partial_{q}\left(\beta^j_p\cdot (g'{}^{pq} h'_{jk})\circ\Phi_v\right)(h'{}^{kl}\circ\Phi_v) e_l \\
        & \quad - \left(\beta^j_p h_{jk}\cdot \partial_q(g^{pq}) + g^{pq}\Gamma^r_{pq}(\beta^j_r h_{jk}) + g(\beta^i, \omega^j_k)h_{ij} \right)h^{kl}e_l + O_N(\kappa|\beta|_{g}) \\
        & = \partial_{q}\left(\beta^j_p \cT^{\bar j p}_{j \bar p}\cdot g^{\bar p q}h_{\bar j k}\right)h^{k\bar l}\cdot \cS^{l}_{\bar l} e_{ l}  \\
        & \quad - \left(\beta^j_p h_{jk}\cdot \partial_q(g^{pq}) + g^{pq}\Gamma^r_{pq}(\beta^j_r h_{jk}) + g(\beta^i, \omega^j_k)h_{ij} \right)h^{kl}e_l + O_N(\kappa|\beta|_{g}) \\
        & = \cS(\nabla^\perp \cdot \cT(\beta)) + O_N(\kappa|\beta|_g) \,,
    \end{align*}
    where we have set
    \[
\cT^{\bar j p}_{j \bar p}:= (g'{}^{pq} h'_{jk})\circ\Phi_v\cdot g_{\bar p q}h^{\bar j k}, \hspace{6mm} \cS^l_{\bar l}:= (h'{}^{kl}\circ\Phi_v)\cdot h_{k\bar l}\,.
\] 
    and $O_N(\kappa|\beta|_g)$ is a function that is pointwise bounded by $C\kappa|\beta|$ for $C=C(N)$ as throughout this appendix; here $\cS$ and $\cT$ are thus the bundle endomorphisms defined by
    \begin{align*}
        & \cT: T^*\Sigma\otimes \mbfV \to T^*\Sigma\otimes \mbfV\,, &  \alpha^j_p\, d\bfx^p\otimes e_j & \mapsto \alpha^j_p\cT^{\bar j p}_{j \bar p}\, d\bfx^{\bar p}\otimes e_{\bar j}\,; \\
        & \cS:  \mbfV \to \mbfV\,, & w^l e_l & \mapsto w^l\cS^{\bar l}_l e_{\bar l} \,.
    \end{align*}

    The smallness estimates claimed above as well as the estimates for $|\cT-\id|, |\cS-\id|$ both follow from \ref{Item_MSE_SecTrans_Concl1} and \eqref{Equ_MSE_Metric and Conn close Euc}-\eqref{Equ_MSE_Conn form close 0}.  This finishes the proof of \ref{Item_MSE_SecTrans_Concl3}. 

    Finally, to prove \eqref{Equ_MSE_Jaco Oper} and \eqref{Equ_MSE_Error est of Jac Oper under Transp}, recall that using the notation in the proof of Proposition \ref{Prop_MSE_Main}, \[
      L_{\Sigma, g} u = \nabla^\perp \cdot \nabla^\perp u + \II^2_{\Sigma, g}u + \cR_{\Sigma, g}u\,.  \]
    Hence, it suffices to combine \ref{Item_MSE_SecTrans_Concl1}-\ref{Item_MSE_SecTrans_Concl3} with the fact that (pointwise, with respect to the metric $g$) there holds
    \[
      |\Phi_v^*\II^2_{\Sigma', g'} - \II_{\Sigma, g}| + |\Phi_v^*\cR_{\Sigma', g'} - \cR_{\Sigma, g}| \leq C\kappa \,.
    \]  \end{proof}

  \begin{corollary} \label{Cor_MSE_converg preserve equ}
    Let $M$ be a smooth manifold, $g_j$ be a sequence of Riemannian metrics that $C^{4}$-converges to $g$, $\Sigma_j$ be a sequence of $\MSI$ in $(M, g_j)$ with normal bundle $\mbfV_j$, $C^2_{loc}$-converging to an $\MSI$ $\Sigma$ in $(M, g)$ with normal bundle $\mbfV$. Suppose 
    \begin{enumerate}[label=(\roman*)]
      \item\label{item_B_Omega} $\Omega\Subset \Sigma$, $\Phi_j: \Omega\to \Sigma_j$ is the graphical parametrization in metric $g$, $\mbfT_j:=\mbfT^{\Sigma_j}_{\Sigma,g}$ be the bundle maps defined in Definition \ref{def:Transfer}.
      \item\label{item_B_equ} $\Omega_j:= \Phi_j(\Omega)$, $u_j\in C^1(\Omega_j; \mbfV_j)$ be weak solutions to \[
        -L_{\Sigma_j, g_j} u_j + \nabla^\perp \cdot b_{0,j} + b_{1,j} + h_j = 0
      \]
      on $\Omega_j$ with $\limsup \|u_j\|_{L^2(\Omega_j)}<+\infty$, where $h_j\in L^2(\Omega_j; \mbfV_j)$, $\mbfT_j(h_j)\to h\in L^2(\Omega; \mbfV)$ in $L^2$, and the following estimate concerning the continuous data $b_{0,j}, b_{1,j}$ holds on $\Omega_j$ for some $\varepsilon_j \searrow 0$: \be\label{eq:ExtraTerms}
        |b_{0,j}| + |b_{1,j}| \leq \varepsilon_j (|u_j|+|\nabla^\perp u_j| + \|h_j\|_{L^2(\Omega_j)})\,.
      \ee
    \end{enumerate}
    Then for every Lipschitz domain $\Omega'\Subset \Omega$, after passing to a subsequence, $v_j:= \mbfT_j(u_j)$ converges in $L^2(\Omega'; \mbfV)$ and weakly in $W^{1,2}(\Omega'; \mbfV)$ to some $u_\infty\in W^{1,2}(\Omega'; \mbfV)$, which is a weak solution to \[
      -L_{\Sigma, g} u_\infty + h = 0\,.
    \]
  \end{corollary}
  \begin{proof}
    Let $\eta\in C^1_c(\Omega)$ be a non-negative cutoff function such that $\eta|_{\Omega'}=1$, and let then $\eta_j:= \eta\circ \Phi_j^{-1}\in C^1(\Omega_j)$. Note that for any $j$ sufficiently large, $\|\eta_j\|_{C^1(\Omega_j)}\leq 2\|\eta\|_{C^1(\Omega)}$.

    Multiplying equation in \ref{item_B_Omega} by $u_j\eta_j^2$ and integrating over $\Omega_j$ yields, through standard manipulations, for all $j$ large enough the bound (that is just a Caccioppoli-type inequality)
    \begin{align*}
      \int_{\Omega_j} |\nabla^\perp u_j|^2\eta_j^2 \,d\|\Sigma_j\|
      & \leq \int_{\Omega_j} |\nabla^\perp u_j|\eta_j \cdot 2|u_j||\nabla^\perp \eta_j|\,d\|\Sigma_j\| \\
      & + \int_{\Omega_j}(C(M,g,\Sigma)|u_j|^2\eta_j^2 + |b_{0,j}\cdot \nabla^\perp (u_j\eta_j^2)| + (|b_{1,j}|+|h_j|)|u_j\eta_j^2|)\,d\|\Sigma_j\| \\ 
      & \leq \int_{\Omega_j}\frac12|\nabla^\perp u_j|^2\eta_j^2 \,d\|\Sigma_j\| + \int_{\Omega_j} C(M,g,\Sigma, \|\eta\|_{C^1})(|u_j|^2 + |h_j|^2)\,d\|\Sigma_j\|.
    \end{align*}

    We can then absorb the the first summand on the right-hand side into the left-hand side to get a uniform $W^{1,2}$ bound for $u_j$ on the domain $\Phi_j(\Omega')$.
    
    Combined with Proposition \ref{Prop_MSE_Section Transp} \ref{Item_MSE_SecTrans_Concl1} and \ref{Item_MSE_SecTrans_Concl2}, this implies \[
        \limsup_{j\to \infty}  \|v_j\|_{W^{1,2}(\Omega')}< +\infty \,.
    \] 
    Therefore, $v_j = \mbfT^{\Sigma_j}_{\Sigma, g}(u_j)$ subconverges weakly in $W^{1,2}(\Omega';\mbfV)$ and strongly in $L^2(\Omega';\mbfV)$ to some $u_\infty\in W^{1,2}(\Omega'; \mbfV)$. The claim that $u_\infty$ is a weak solution to $-L_{\Sigma, g}(u_\infty) + h = 0$ follows from Proposition \ref{Prop_MSE_Section Transp} \ref{Item_MSE_SecTrans_Concl1}, \eqref{Equ_MSE_Jaco Oper} and \eqref{Equ_MSE_Error est of Jac Oper under Transp} by exploiting ---  to handle the extra terms in the weak form of the equation --- the Cauchy-Schwartz inequality and the dominated convergence theorem as one lets $j\to\infty$ (in view of \eqref{eq:ExtraTerms}).
  \end{proof}

\section{A three-circle inequality and related decay estimates}\label{sec:3circles}

 The following growth rate monotonicity formula is a direct consequence of the decomposition (\ref{Equ_Pre_Decomp Jac into homogen Jac field}), which was in fact first introduced in \cite{Sim83} in a different form.
    \begin{lemma} \label{Lem_Ana on SMC_Growth Rate Monoton}
        For any $\sigma>0$, there exists $K = K(\sigma)>2$ with the following property.  
            
        For any $\gamma\in \RR$ and any minimal cone $\mbfC\in \cC_{N,n}$ satisfying \[
            \dist_\RR (\gamma, \Gamma(\mbfC)\cup \{-(n-2)/2\}) \geq \sigma\,,  \]
        if $u\in C^2_{loc}(\mbfA(K^{-3}, 1);\mbfV)\cap L^2(\mbfA(K^{-3}, 1);\mbfV)$ solves $L_\mbfC u = 0$ then its growth rate \[
            J_K^\gamma (u; r):= \int_{\mbfA(K^{-1}r, r)} |u|^2\cdot |x|^{-2\gamma - n} d\|\mbfC\|\,,  \]
        satisfies the inequality 
        \[
            J_K^\gamma(u; K^{-2}) - 2 J_K^\gamma(u; K^{-1}) + J_K^\gamma(u; 1) \geq 0\,.   \]
        Moreover, equality holds if and only if $u\equiv 0$.
    \end{lemma}

Recall that, consistently with the general notational conventions we had stipulated in Section \ref{sec:Prelim} we have denoted by $\mbfV=\mbfV(\mbfC)$ the normal bundle to $\mbfC$ in $\R^{N}$ and by $\mbfA(r, s):= \AAa(\mathbf{0}, r, s)\cap \mbfC$ the annular region, on the cone, of radii $r<s$. 
    
    \begin{proof}
        By \eqref{Equ_Pre_Decomp Jac into homogen Jac field}-\eqref{Equ_Pre_Homogen Jac field}, see Section \ref{subs:AnalysisCone1}, we have 
        \begin{align*}
            J^\gamma_K(v; r)  = \int_{K^{-1}r}^r t^{n-1-n-2\gamma}\ dt\int_{\mbfC\cap \SSp^{N-1}} v(t, \omega)^2\ d\omega 
                              = \sum_{j\geq 1}\underbrace{\int_{K^{-1}r}^r t^{-1-2\gamma}\cdot (v_j^+(t)+v_j^-(t))^2\ dt}_{=:\ T_j},
        \end{align*}
        where of course the second equality relies upon Parseval's identity, with respect to the orthonormal (Hilbertian) basis $\varphi_1,\varphi_2,\ldots,\varphi_k,\ldots$ of the space $L^2(\mbfC\cap\SSp^{N-1};\mbfV)$ of square-integrable
        normal sections to the link $S:=\mbfC\cap\SSp^{N-1}$ in the unit sphere of Euclidean $\R^N$.
        
        Recalling the definition of $\mu_j=\lambda_j+(n-1)$ where $\lambda_j$ is the $j$-th eigenvalue of the Jacobi operator of the link $S$ in question (see the second part of Section \ref{subs:FirstSecVar}), if $\mu_j\geq -(n-2)^2/4$ then we have
        \begin{align*}
            T_j & = \int_{K^{-1}r}^r \left(c_j^+ t^{\gamma_j^+ -\gamma} + c_j^- t^{\gamma_j^- - \gamma} \right)^2 t^{-1}\ dt,   & & \ \text{ if }\mu_j > -\frac{(n-2)^2}{4};  \\
            T_j & = \int_{K^{-1}r}^r \left(c_j^+ t^{-(n-2)/2 -\gamma} + c_j^- t^{-(n-2)/2 - \gamma}\log t\right)^2 t^{-1}\ dt, & &  \ \text{ if }\mu_j = -\frac{(n-2)^2}{4};  \\
        \end{align*}
        while if $\mu_j< -(n-2)^2/4$, set $\alpha_j:= \sqrt{-\mu_j-(n-2)^2/4}$, $2c_j^+ =: c_j e^{i\theta_j}$ (with $c_j\in\R$), then
        \begin{align*}
            T_j = \int_{K^{-1}r}^r c_j^2 \left(\cos(\alpha_j\log(t)+\theta_j)t^{-(n-2)/2-\gamma}\right)^2 t^{-1}\ dt\,.
        \end{align*}
        At this stage it suffices to show that for $K=K(\sigma)$ large enough, the desired inequality holds at the level of each $T_j\geq 0$ with equality if and only if $c_j^\pm = 0$. When $\mu_j\geq -(n-2)^2/4$, this follows from \cite[Lemma A.1 and Lemma A.2]{LW22}; instead when $\mu_j< -(n-2)^2/4$, it is a consequence of the following statement.
    \end{proof}
    \begin{lemma} \label{Lem_App_Growth rate mon alpha-beta}
        Suppose $\sigma>0$. Then there exists a real number $K_0 = K_0(\sigma)>2$ with the following property. 
            
        For any $\alpha, \beta, \theta\in \RR$ such that $\alpha>0, |\beta| \geq \sigma$ and any $K\geq K_0$, the integral \[
            \cI_K(r) := \int_r^{Kr}\cos^2(\alpha\log(s)+\theta)\, s^{2\beta-1}\ ds,   
        \]
        satisfies, for every $r>0$, the inequality \[
            \cI_K(K^2r) - 2\cI_K(Kr) + \cI_K(r) > 0\,. 
        \]
    \end{lemma}
    \begin{proof}
        Notice that 
        \begin{align*}
            \int_r^{Kr} \cos^2(\alpha\log(s) + \theta)\, s^{2\beta-1}\ ds = \int_{\log(r)+\theta/\alpha}^{\log(r)+\theta/\alpha +\log(K)} \cos^2(\alpha\tau)\, e^{2\beta\tau-2\beta\theta/\alpha}\ d\tau\,.
        \end{align*}
        Therefore, if we let  $r':=\log(r)+\theta/\alpha$, $K':=\log(K)$, $\beta':= 2\beta$, $\alpha'=2\alpha$ and \[
            \tilde{\cI}_{K'}(r'):= \int_{r'}^{r'+K'} 2\cos^2(\alpha\tau)\, e^{\beta'\tau}\ d\tau = \int_{r'}^{r'+K'} \left(1+\cos(\alpha'\tau)\right) e^{\beta'\tau}\ d\tau\,
        \]
        then it suffices to show that, when $K'\geq K'(\sigma)$ is large enough, for any $r'\in \RR$ there holds \[
            \tilde\cI_{K'}(r'+2K') - 2\tilde\cI_{K'}(r'+K') + \tilde\cI_{K'}(r') > 0 \,.
        \]
        By replacing $r'$ by $-r'-K'$ if necessary, we can further assume, without loss of generality, that $\beta'\geq 2\sigma>0$. Let $\beta'':= \beta'+i\alpha'\in \CC$. Then,
        \begin{align*}
            \tilde\cI_{K'}(r') = \Re\left(\int_{r'}^{r'+K'}  (e^{\beta'\tau} + e^{\beta''\tau})\ d\tau \right) 
            = e^{\beta'r'}\cdot\frac{e^{\beta'K'}-1}{\beta'} + \Re\left( e^{\beta''r'}\cdot\frac{e^{\beta''K'}-1}{\beta''} \right) \,.
        \end{align*}
        We can employ this formula to rewrite the quantity we wish to prove positive; elementary manipulations allow to write:
        \begin{align*}
        & \;\;\;\;\; 
           \tilde\cI_{K'}(r'+2K') - 2\tilde\cI_{K'}(r'+K') + \tilde\cI_{K'}(r') \\
            & = e^{\beta'r'}\cdot\frac{(e^{\beta'K'}-1)^3}{\beta'} + \Re\left( e^{\beta''r'}\cdot\frac{(e^{\beta''K'}-1)^3}{\beta''} \right) \\
             & = \frac{e^{\beta'r'}(e^{\beta'K'}-1)^3}{|\beta''|} \left[ \frac{|\beta''|}{\beta'} + \Re\left( \frac{|\beta''|e^{i\alpha'(r'+3K')}}{\beta''}\cdot \left(\frac{e^{\beta'K'}-e^{-i\alpha'K'}}{e^{\beta'K'} - 1}\right)^3 \right) \right] \\
            & \geq \frac{e^{\beta'r'}(e^{\beta'K'}-1)^3}{|\beta''|} \left[ \frac{|\beta''|}{\beta'} - \left(\frac{|e^{\beta'K'}-e^{-i\alpha'K'}|}{e^{\beta'K'} - 1}\right)^3 \right]\,.
        \end{align*}
        Since 
        \begin{align*}
            \left(\frac{|e^{\beta'K'}-e^{-i\alpha'K'}|}{e^{\beta'K'} - 1}\right)^2 
            & = \frac{\left(e^{\beta'K'}-\cos(\alpha'K')\right)^2 + \sin^2(\alpha'K')}{(e^{\beta'K'} - 1)^2} \\
            & \leq \left( 1 + \frac{\min\{\alpha'K', 1\}^2}{e^{\beta'K'}-1}\right)^2 + \frac{\min\{\alpha'K', 1\}^2}{(e^{\beta'K'} - 1)^2} \,.
        \end{align*}
        Thus when $\beta'K'$ is large enough (arranged by requiring $K\geq K_0(\sigma)$ for suitable $K_0(\sigma)$), we have in fact
        \begin{align*}
            \left(\frac{|e^{\beta'K'}-e^{-i\alpha'K'}|}{e^{\beta'K'} - 1}\right)^6
            \leq 1 + \frac{20\min\{\alpha'K', 1\}^2}{e^{\beta'K'}-1} < 1 + \frac{(\alpha'K')^2}{(\beta'K')^2} = \left(\frac{|\beta''|}{\beta'} \right)^2,
        \end{align*}
       which implies the claim.
    \end{proof}

    We stress that the equality case in Lemma~\ref{Lem_Ana on SMC_Growth Rate Monoton} holds only when the section $u$ is identically $0$. Therefore, for non-zero sections we can suitably strengthen the lemma to an open condition, which allows us to perturb both the metrics and the coefficients in the PDE as described in the following statement. This refinement is particularly useful in our applications.
    
    \begin{corollary}[Perturbed version of Lemma \ref{Lem_Ana on SMC_Growth Rate Monoton}] \label{Cor_Ana on SMC_Growth Rate Monoton with Perturb}
        For $\sigma>0$ and $\Lambda>0$, let $K(\sigma)>2$ be the same as in Lemma \ref{Lem_Ana on SMC_Growth Rate Monoton}.  Then there exists $\varepsilon(\sigma, \Lambda)>0$ small enough that the following property holds.
        
        For $\gamma \in [-\Lambda,\Lambda]$ and any $\mbfC\in \cC_{N,n}(\Lambda)$ satisfying 
        \[
            \dist_\RR (\gamma, \Gamma(\mbfC)\cup \{-(n-2)/2\}) \geq \sigma\,,
        \] 
        let $0\neq u\in W^{1,2}_{loc}(\mbfA(K^{-3}, 1);\mbfV)\cap L^2(\mbfA(K^{-3}, 1);\mbfV)$ be a weak solution to 
        \begin{align}
        \nabla^\perp\cdot (\nabla^\perp u + b_0(x)) + \langle A_\mbfC, u\rangle A_\mbfC  + |x|^{-1} b_1(x) = 0,  \label{Equ_Pre_Pert Jac equ}
        \end{align}
        where $\nabla^{\perp}=\nabla_{\mbfC, g_{\eucl}}^\perp$ is the Levi-Civita connection of the normal bundle $\mbfV$ to $\mbfC$ in Euclidean $\R^N$, and $b_0, b_1\in L^{2}(\Sigma;\mbfV)$ satisfy the following estimates
        \begin{align}
             \|(|b_0| + |b_1| - \eps|\nabla^\perp u|)_+\|_{L^2(\mbfA(K^{-3}, 1))} \leq \varepsilon \|u\|_{L^2(\mbfA(K^{-3}, 1))},  \label{Equ_Pre_Small error in pert Jac equ}
        \end{align}
        on $\mbfA(K^{-3}, 1)$. (Here we use the notation $v_+:= \max\{v, 0\}$.) Then we have, 
        \[
             J_K^\gamma(u; K^{-2}) - 2(1+\varepsilon) J_K^\gamma(u; K^{-1}) + J_K^\gamma(u; 1) > 0\,.
        \]
    \end{corollary}

    \begin{remark} \label{Rem_Pre_Jac field equ on MH near cone} 
        Let $\Sigma := \graph_\mbfC(\phi)\cap \BB(2)$ be a minimal submanifold, under a metric $g$, parametrized by the cone $\mbfC\in \cC_{N, n}(\Lambda)$ (that, let us stress the point, is minimal in Euclidean $\R^N$ so under the metric $g_{\eucl}$), and assume
        \[
            \|g-g_{\eucl}\|_{C^{4}(\BB(2))} + \|\phi\|_{C^2_1(\mbfC\cap \BB(2))}  \leq \tilde\varepsilon\,,
        \]
        where $\tilde\varepsilon\leq \kappa_2$ and $\kappa_2$ has been defined in Proposition \ref{Prop_MSE_Section Transp}. 
        Then for any Jacobi field $u$ on $\Sigma$, the corresponding normal section $v$ over $\mbfC$ under this parametrization (by which we mean $v:=\mbfT^{\Sigma}_{\mbfC,g_{\eucl}}(u)$ in the sense of Definition \ref{def:Transfer}), solves an equation of form \eqref{Equ_Pre_Pert Jac equ} with 
        \begin{align}
            |b_0(x)| + |b_1(x)| \leq C(\Lambda)\,\tilde\varepsilon\cdot (|x|^{-1}|v(x)| + |\nabla^\perp v(x)|)\,.   \label{Equ_3Circle_Rem}            
        \end{align}
        To see this, for every $\check x\in \mbfC$, let $\check r:= \kappa_2\mfr_{\mbfC, g_\eucl}(\check x)/C$, and $\check \eta : x\mapsto \check x+ \check r x$, where $\kappa_2$ is the dimensional constant in Proposition \ref{Prop_MSE_Section Transp} \ref{Item_MSE_SecTrans_Concl3}, and $C=C(N)>2$ is to guarantee that $\eta^{-1}(\mbfC)$ is a $\kappa_2$-$C^4$ graph over the ball of radius $2$ in $T_{\check x}\mbfC$.
        Then, Proposition \ref{Prop_MSE_Section Transp} \ref{Item_MSE_SecTrans_Concl3} applies with $g_\eucl, \check g:= {\check r}^{-2}\check\eta^*g, \check\mbfC:= \check\eta^{-1}(\mbfC), \check\Sigma:=\check\eta^{-1}(\Sigma), \check v:= v\circ \check\eta, u\circ \check\eta$ in place of $g, g', \Sigma, \Sigma', u, u'$ therein, which gives \[
          -L_{\check\mbfC, g_\eucl} (\check v) + \nabla^\perp \cdot \check b_0 + \check b_1=0,
        \]
        where by \eqref{Equ_MSE_Error est of Jac Oper under Transp} there hold the estimates 
        \[
          |\check b_0(0)| + |\check b_1(0)| \leq C(\Lambda)\,\varepsilon \cdot(|\check v(0)| + |\nabla^\perp \check v(0)|) \,.
        \]
        Now \eqref{Equ_3Circle_Rem} at $\check x$ follows by scaling back the equation and setting 
        \[
        b_0 = \check b_0\circ \check \eta^{-1}\cdot \check r^{-1}, \ \ b_1(x) = |x|\check b_0\circ \check\eta^{-1} \cdot \check r^{-2}
        \]
        so that an equation of the form \eqref{Equ_Pre_Pert Jac equ} is satisfied. As a consequence of \eqref{Equ_3Circle_Rem}, given any $\sigma > 0$ and $\Lambda > 1$, for $\varepsilon(\sigma, \Lambda)$ determined in Corollary \ref{Cor_Ana on SMC_Growth Rate Monoton with Perturb}, we can choose $\tilde \varepsilon(\sigma, \Lambda)$ so small
     that the error estimate assumption \eqref{Equ_Pre_Small error in pert Jac equ} is satisfied with $v$ in place of $u$ therein. Lastly, we note that (without much additional effort)
a similar equation with analogous error estimates holds for $\mbfT^{\Sigma'}_{\Sigma,g}(w)$, where $w$ is a graphical section defining a minimal submanifold $\Sigma'$ close to $\Sigma$, under any metric $g'$ suitably close to $g$ (possibly, but not necessarily, $g$ itself), with error estimate \eqref{Equ_3Circle_Rem} replaced by 
        \begin{align*}
            |b_0(x)| + |b_1(x)| \leq C(\Lambda)\cdot (\tilde\varepsilon(|x|^{-1}|w(x)|+|\nabla^\perp w(x)|) + |x|\|g'-g\|_{C^2(\BB(2))})\,. 
        \end{align*}
    \end{remark}

    \begin{proof}
        Suppose, for a contradiction, that there exist sequences $\mbfC_j\in \cC_{N,n}(\Lambda)$, and $u_j\in W^{1, 2}_{loc}(\mbfA_j(K^{-3}, 1);\mbfV_j)\cap L^2(\mbfA_j(K^{-3}, 1);\mbfV_j)$ weak solutions to \eqref{Equ_Pre_Pert Jac equ}, where $b_0, b_1$ are replaced by $b_0^j,b_1^j$ satisfying \eqref{Equ_Pre_Small error in pert Jac equ} with $\varepsilon=1/j$ and the 3-circle inequality fails: \[
             J_K^\gamma(u_j; K^{-2}) + J_K^\gamma(u_j; 1) \leq 2(1+1/j) J_K^\gamma(u_j; K^{-1}) =: 2c_j^2>0.          
        \]
        (Note that here we have employed the usual notation $\mbfA_j(r,s)=\AAa(r,s)\cap \mbfC_j$, understood for all $j\geq 1$.)
        Consider then $\hat{u}_j:= c_j^{-1}u_j, b_i^j:= c_j^{-1} b_i^j$, where $i=0, 1$; of course \eqref{Equ_Pre_Pert Jac equ} and \eqref{Equ_Pre_Small error in pert Jac equ} are also satisfied with such replacements throughout. Thanks to Lemma \ref{Lem_cC_(N,n)(Lambda) bF Cpt} we can extract a subsequence of cones converging smoothly (at the level of spherical section, hence on the whole annulus of radii $K^{-3}/3$ and $3$) to a limit cone $\mbfC_{\infty}$, and thus for every $p\geq 1$, as $j\to \infty$,
    \[
        \gamma_p^\pm(\mbfC_j) \to \gamma^\pm_p(\mbfC_\infty).
    \]
        Now, for $j\geq 1$ large enough $\mbfA_j(K^{-3},1)$ is a subset of 
        $\graph_{\mbfC_{\infty}}(\phi_j)\cap \AAa( K^{-3}/2, 2)$
        and one may consider
        \[
        v_j := \mbfT^{\mbfC_j}_{\mbfC_{\infty},g_{\eucl}}(\hat u_j)
        \]
        for a suitable sequence $\left\{\phi_j\right\}_{j\geq j_0}$ satisfying $\|\phi_j\|_{C^2_1(K^{-3}/2,2)}\leq\tilde{\varepsilon}$ chosen -- based upon Remark \ref{Rem_Pre_Jac field equ on MH near cone} -- so that condition \eqref{Equ_Pre_Small error in pert Jac equ} is fulfilled.
        Then, since (by its definition)
        $\|\hat{u}_j\|_{L^2(\mbfA_j(K^{-3}, 1))}$ is uniformly bounded in $j\geq 1$,
        by classical elliptic theory (interior estimates in Hilbertian Sobolev spaces), up to extracting a subsequence (which we shall not rename), $v_j$ converges, strongly in $L^2_{loc}(\mbfA_{\infty}(K^{-3}, 1);\mbfV_{\infty})$ and weakly in $W^{1,2}_{loc}(\mbfA_{\infty}(K^{-3}, 1);\mbfV_{\infty})$ to some non-zero section $v_\infty$ in $W^{1,2}_{loc}(\mbfA_{\infty}(K^{-3}, 1);\mbfV_{\infty})$. In particular, note that for any $i=0,1,2$ there holds \[
            J^\gamma_K(v_\infty, K^{-i})\leq \limsup_{j\to \infty} J^\gamma_K(v_j, K^{-i}) =\limsup_{j\to \infty} J^\gamma_K(\hat{u}_j, K^{-i})\, 
        \] 
        with equality for $i=1$, that is for the ``intermediate annulus''. (Also, note that the non-triviality of the limit section $v_{\infty}$ follows from the $L^2$ convergence in such an annulus.) Hence, \[
             J_K^\gamma(v_\infty; K^{-2}) + J_K^\gamma(v_{\infty}; 1) \leq 2 J_K^\gamma(v_\infty; K^{-1}) .          
        \]
        Lastly, by \eqref{Equ_Pre_Small error in pert Jac equ} the dominated convergence theorem allows to conclude that $v_\infty$ weakly solves $L_\mbfC v_\infty = 0$ on $\mbfA(K^{-3}, 1)$. This contradicts Lemma \ref{Lem_Ana on SMC_Growth Rate Monoton}, thereby completing the proof.
    \end{proof}

    \begin{corollary} \label{Cor_App_Uniform Growth Est for Jac Equ}
        For $\sigma\in (0, 1)$, $\Lambda>0$, let $K = K(\sigma/2)>2$ be given by Lemma \ref{Lem_Ana on SMC_Growth Rate Monoton}.  Then there exists $\varepsilon_0(\sigma, \Lambda)>0$ small enough that the following property holds.
       
        Let $\gamma\in [-\Lambda, \Lambda]$, $\Sigma$ be an $\MSI$ in a Riemannian manifold $(M,g)$, $q\in \Sing(\Sigma)$; suppose that, after pulling back to $T_qM$ using the exponential map (thus replacing, with slight abuse of notation, $g$ by its pull-back $\exp^\ast_q g$), $\Sigma \cap B^g(q,4)\subset\graph_{\mbfC}(\phi)$ for some $\mbfC\in\cC_{N,n}(\Lambda)$  and
        \begin{align}
            \|g-g_\eucl\|_{C^{4}(\BB(4))}  < \varepsilon_0\,, & &
            \|\phi\|_{C^2_1(\mbfC\cap \BB(4))} < \varepsilon_0\label{Equ_App_Close to Cone w Reg Scal bd}.
        \end{align}
        Assume further
        \begin{align}
          \dist_\RR (\gamma, \Gamma(\mbfC)\cup \{-(n-2)/2\}) \geq \sigma\,.\label{Equ_App_Uniform away from asymp spectr}
        \end{align}
        Let $u\in W^{2,2}_{loc}(\Sigma;\mbfV)$, $v:=\mbfT^{\Sigma}_{\mbfC,g_{\eucl}}(u)$ and $f:=\mbfT^{\Sigma}_{\mbfC,g_{\eucl}}(L_{\Sigma,g}u) $ be satisfying
        \begin{align*}
            \cA\cR_q(u)>\gamma\,, & &
            F := \sup_{\ell\geq 0}\||x|^{2-\gamma -n/2} f\|_{L^2(\mbfA(K^{-\ell-1}, K^{-\ell}))} < +\infty\,.
        \end{align*}
        Then, we have for every $\ell\geq 0$, 
        \begin{align}
            \|v\|_{L^2(\mbfA(K^{-\ell-1}, K^{-\ell}))} \leq C( \Lambda, \sigma) \left(F + \|v\|_{L^2(\mbfA(K^{-1}, 1))} \right)\cdot K^{-\ell(n/2 + \gamma)}\,. \label{Equ_App_Weighted L^2 Bd by RHS and finite scale L^2}
        \end{align}
    \end{corollary}
    \begin{proof}
      Firstly, by taking $\varepsilon_0$ small enough, $v$ satisfies (\ref{Equ_Pre_Pert Jac equ}) and 
        \begin{align*}
            |b_0|(x) + |b_1|(x) \leq \eps \left(|x|^{-1}|v|(x) + |\nabla^\perp v|(x) \right) + |x||f(x)|
        \end{align*}
        where $\varepsilon$ is again as in the statement of Corollary \ref{Cor_Ana on SMC_Growth Rate Monoton with Perturb}.
        Hence, for every $\ell\in \N$, $v_\ell(x):= v(K^{-\ell}x)$ also solves an equation of the form (\ref{Equ_Pre_Pert Jac equ}) with $b_{i, \ell}$ in place of $b_i$, $i=0, 1$, satisfying
        \begin{align}
            |b_{0, \ell}|(x) + |b_{1, \ell}|(x) \leq \eps \left(|x|^{-1}|v_\ell|(x) + |\nabla^\perp v_\ell|(x) \right) + K^{-2\ell}|x||f_\ell(x)|,  \label{Equ_App_In Pf_Error est in Pert Jac equ}
        \end{align}
        on $\mbfA(\orig, K^{-3}, 1)$, where $f_\ell(x):= f(K^{-\ell}x)$. 
        
        Clearly, to prove (\ref{Equ_App_Weighted L^2 Bd by RHS and finite scale L^2}), it suffices to show that there exists $C_0( \Lambda, \sigma)$ so large that
        \begin{align}
            \ell\mapsto J_\ell:=\max\{K^{2\gamma \ell}J^{\gamma}_K(v_\ell; 1), C_0^2F^2\} \label{Equ_App_Tilde{J}_l monotone}
        \end{align}
        is monotone non-increasing in $\ell\geq 0$. To show this, suppose for contradiction that for some $\ell_0\geq 0$, $J_{\ell_0+1}> J_{\ell_0}$, then 
        \begin{align*}
            & F \leq C_0^{-1}C(\Lambda,\sigma)\|v_{\ell_0}\|_{L^2(\mbfA(K^{-2}, K^{-1}))}K^{\gamma \ell_0}\,, \\
            & J^{\gamma}_K(v_{\ell_0}, K^{-1}) = K^{2\gamma}J^{\gamma}_K(v_{\ell_0+1}, 1)\geq J^{\gamma}_K(v_{\ell_0}, 1)\,. 
        \end{align*}
        Combining the first inequality and (\ref{Equ_App_In Pf_Error est in Pert Jac equ}), we can take $C_0(\Lambda, \sigma)$ so big that (\ref{Equ_Pre_Small error in pert Jac equ}) holds for $b_{i, \ell_0}$ in place of $b_i$ and $v_{\ell_0}$ in place of $u$. Hence by Corollary \ref{Cor_Ana on SMC_Growth Rate Monoton with Perturb} and the second of such inequalities, 
        \begin{align*}
           J_{\ell_0+2}\geq K^{2(\ell_0+2)\gamma}J^{\gamma}_K(v_{\ell_0+2}, 1) 
           & = K^{2\ell_0\gamma}J^{\gamma}_K(v_{\ell_0}, K^{-2}) \\
           & > K^{2\ell_0\gamma}J^{\gamma}_K(u_{\ell_0}, K^{-1})
             = J_{\ell_0+1}\,.
        \end{align*}
        Inductively, we thus see $J_\ell$ would be strictly increasing in $\ell\geq \ell_0$ (and thus strictly greater than $C_0^2F^2$). But on the other hand, by definition of $\cA\cR_q(u)$, since \[
          K^{2\gamma \ell}J^\gamma_K(v_\ell, 1) = J^\gamma_K(v, K^{-\ell}) \to 0\,, \quad \text{ as }\ell\to \infty\,
        \] 
        it follows
        that when $\ell$ is large enough, $J_\ell = C_0^2F^2$ unless $F=0$. So if $F>0$ then we directly violate the aforementioned strict monotonicity, otherwise if $F=0$ we get a contradiction from the coexistence of such strict monotonicity with the fact that $J_\ell \to 0$ as $\ell\to\infty$. 
    \end{proof}

   We conclude this appendix with an important ``non-concentration result'' for Jacobi fields. 

    \begin{corollary} \label{Cor_App_Uniform Noncon for Jac fields}
        Let $\sigma, \Lambda, K, \eps_0, \Sigma,$ and $q$ be the same as Corollary \ref{Cor_App_Uniform Growth Est for Jac Equ}. Then there exists $\eps_1(\Lambda,\sigma)\in (0, \eps_0)$ with the following property.
        Suppose (\ref{Equ_App_Close to Cone w Reg Scal bd}) holds for $\eps_1$ in place of $\eps_0$, $\gamma = 1-\sigma$ satisfies (\ref{Equ_App_Uniform away from asymp spectr}) and $u$ is a tame Jacobi field on $\Sigma$, then there exists $\mbfx\in T_qM$ such that 
        \begin{align}
            |\mbfx| + \sup_{\ell\geq 0}\|u - \mbfx^\perp\|_{L^2_{1-\sigma}(A^g(q, K^{-\ell-1}, K^{-\ell})\cap\Sigma)} \leq C(\Lambda, \sigma)\|u\|_{L^2(A^g(q, \varepsilon_1, 2)\cap\Sigma)}\,.  \label{Equ_App_Noncon Est of Slower Growth Jac} 
        \end{align}
    \end{corollary}
    \begin{proof}
     
        By definition of tame Jacobi field and Corollary \ref{cor:AR_Jacobi} there exists $\mbfx\in T_qM$ such that $\cA\cR_q(u-\mbfx^\perp)\geq 1$; hence 
        if we apply Corollary \ref{Cor_App_Uniform Growth Est for Jac Equ} to $\tilde{u}:= u-\mbfx^\perp$ (also keeping in mind item \ref{Item_MSE_SecTrans_Concl1}  of Proposition \ref{Prop_MSE_Section Transp} to transfer $L^2$ bounds from $\mbfC$ to $\Sigma$, as well as the content of Lemma \ref{lem:JacobiTranslFunct} and Remark \ref{rem:LinftyCanNeighb}), we find for every $\ell\geq 0$ 
        \begin{align}
            \|u-\mbfx^\perp\|_{L^2(A^g(q,K^{-\ell-1}, K^{-\ell})\cap\Sigma)} \leq C( \Lambda, \sigma)(|\mbfx| + \|u-\mbfx^\perp\|_{L^2(A^g(q,K^{-1}, 1)\cap\Sigma)})\cdot K^{-\ell(n/2+1-\sigma)}  \label{Equ_App_In Pf_Growth Est for u-x^perp}          
        \end{align}

        Hence, for every $s\in (0, 1/2)$, by virtue of Lemma \ref{lem:L2perp} (that is: equation \eqref{eq:L2perp}, suitably scaled), the Minkowski inequality (employed twice) and \eqref{Equ_App_In Pf_Growth Est for u-x^perp} there holds      
        \begin{align*}
          |\mbfx| & \leq C(\Lambda)s^{-n/2}\|\mbfx^\perp\|_{L^2(A^g(q,s, 2s)\cap\Sigma)} \\
          & \leq C(\Lambda)s^{-n/2}(\|u-\mbfx^\perp\|_{L^2(A^g(q,s, 2s)\cap\Sigma)} + \|u\|_{L^2(A^g(q,s, 2s)\cap\Sigma)}) \\
          & \leq C( \Lambda, \sigma)s^{1-\sigma}(\|u-\mbfx^\perp\|_{L^2(A^g(q,K^{-1}, 1)\cap\Sigma)} +|\mbfx|) + C(\Lambda)s^{-n/2}\|u\|_{L^2(A^g(q,s, 2s)\cap\Sigma)} \\
          & \leq C( \Lambda, \sigma)s^{1-\sigma}(\|u\|_{L^2(A^g(q,K^{-1}, 1)\cap\Sigma)} + |\mbfx|) + C(\Lambda)s^{-n/2}\|u\|_{L^2(A^g(q,s, 2s)\cap\Sigma)}.
        \end{align*}
        Since $\sigma\in (0, 1)$, we can take $s(\Lambda, \sigma)<K^{-1}$ such that $C(\Lambda, \sigma)s^{1-\sigma}<1/2$ (in order to allow for absorption on the left-hand side)
        and then $\varepsilon_1\leq s$ so that we eventually derive 
        \begin{align}
          |\mbfx| \leq C( \Lambda, \sigma)\|u\|_{L^2(A^g(\varepsilon_1, 1)\cap\Sigma)}\,. \label{Equ_App_In Pf_|x|<|u|_(L^2)}             \end{align}

Combining this with \eqref{Equ_App_In Pf_Growth Est for u-x^perp}  finishes the proof.
        \end{proof}

\section{Quantitative uniqueness of tangent cones in all (co-)dimensions}\label{sec:QuantUniq}
    
    In \cite[Theorem~6.3]{Edelen2021}, Edelen proved a quantitative version of the uniqueness of tangent cones for minimal hypersurfaces in eight-dimensional manifolds by adapting Simon's argument \cite{Sim83}, based on the \oldL{}ojasiewicz inequality. That same argument can in fact be extended to general dimension and codimension, and leads to the following statement. 
        
    \begin{proposition} \label{Pro_Ana on SMC_Quanti Uniqueness of Tang Cone}
         Given $0<\eps<1$ and $\Lambda>0$, there exists $\delta = \delta(\eps, \Lambda)>0$ with the following property.
            
        Let $\mbfC \in \cC_{N, n}(\Lambda)$ be an $n$-dimensional minimal cone in $\RR^N$ with normal bundle $\mbfV$, $g$ be a $C^4$ metric on $\BB(4)$, $V\in \cI_n(\BB(4))$ be a $g$-stationary integral varifold, and let $r\in (0, 1/16)$. Suppose: 
        \begin{enumerate}[label=(\roman*)]
            \item\label{item:MetricQuasiEucl} $\|g-g_{\eucl}\|_{C^4(\BB(4))}\leq \delta$;
            \item \label{item:DensityQuasiConst}$\theta_{V}(\mathbf{0}, 2)\leq \theta_{V}(\mathbf{0}, r) + \delta$;
            \item\label{item:GraphAnnuli} there exists $s\in (2r, 1/2)$ so that $V$ is a $\delta$-$C^2$ graph over $\mbfC$ in $A^g(\orig, s, 2s)$.
        \end{enumerate}
        Then, $V$ is an $\eps$-$C^2$ graph over $\mbfC$ in $A^g(\orig, 2r, 1)$.
    \end{proposition}
  
    \begin{proof}
        Suppose for contradiction, that there exist $\eps>0$, $\Lambda>0$ and a sequence of metrics $g_j$, stationary integral varifolds $V_j\in \cI_n(\BB(4))$, $r_j\in (0, 1/16)$, $s_j\in (2r_j, 1/2)$ and $\mbfC_j\in \cC_{N, n}(\Lambda)$ such that for every $j\geq 1$, properties \ref{item:MetricQuasiEucl}-\ref{item:DensityQuasiConst}-\ref{item:GraphAnnuli} hold for $g_j, V_j, \mbfC_j, r_j, s_j, 1/j$ in place of $g, V, \mbfC, r, s, \delta$, but $V_j$ is not an $\eps$-$C^2$ graph over $\mbfC_j$ in $A^{g_j}(\orig, 2r_j, 1)$. 
        
        By Lemma \ref{Lem_cC_(N,n)(Lambda) bF Cpt}, after passing to a subsequence, $\mbfC_j$ smoothly converges to some $\mbfC_\infty\in \cC_{N, n}(\Lambda)$ as $j\to \infty$; then there exists a sequence $\left\{\delta_j\right\}_{j\geq 1}$ with $\delta_j\searrow 0$ such that $V_j$ is a $\delta_j$-$C^2$ graph over $\mbfC_\infty$ in $A^{g_j}(\orig, s_j, 2s_j)$. Let $\bar{s}_j^- \leq s_j< \bar{s}_j^+$ be radii such that $\AAa(\bar{s}_j^-, \bar{s}_j^+)\subset A^{g_j}(\orig, r_j, 2)$ is the largest annulus centered at $\orig$ in which $V_j$ is an $\eps'$-$C^2$ graph over $\mbfC_\infty$, where $\eps'=\eps'(\mbfC_\infty)<\eps$ is to be determined at a later stage along the course of the argument (independently of $j$).

        Our goal is to show that for infinitely many indices $j\geq 1$ there hold in fact the inequalities
        \begin{align}
            \bar{s}_j^- \leq 3r_j/2\,, \qquad \bar{s}_j^+ \geq 3/2\,. \label{Equ_App_UniqTan_s_j^pm attains max}
        \end{align}
        Note that since $\mbfC_j\to \mbfC_\infty$ smoothly away from $\orig$, this immediately implies that for all such indices (possibly with finitely many initial exceptions) $V_j$ is an $\eps$-$C^2$ graph over $\mbfC_j$ in $A^g(\orig, 2r_j, 1)$, which is a contradiction.

        To prove \eqref{Equ_App_UniqTan_s_j^pm attains max}, we shall apply \cite[Theorem 11.1]{Edelen2021}. Recalling the notation for the scaling map $\eta_\lambda: x\mapsto \lambda x$, we first observe that when $j\to \infty$, $s_j^{-2}\eta_{s_j}^*g_j$ $C^4$-converges to $g_\eucl$ and by Allard's Compactness Theorem, after passing to a subsequence, $V_j':= (\eta_{s_j}^{-1})_\# V_j$ $\mbfF$-converges to some $g_\eucl$-stationary integral varifold $V_\infty'$ in $\AAa(\hat{r}_-, \hat{r}^+)$, where \[
            \hat{r}^-:= \lim_{j\to \infty} r_j/s_j\,, \qquad 
            \hat{r}^+:= \lim_{j\to \infty} 2/s_j\,
        \] 
        so that $r\mapsto\theta_{V_\infty'}(\orig, r)$ is constant in $r\in (\hat{r}^-, \hat{r}^+)$, and on the other hand $V_\infty'$ coincides with the (multiplicity one) varifold associated to $\mbfC_\infty$ in $\AAa(1,2)$. 
        Note, in particular, that we allow $\hat{r}^{\pm}$ being $0$ or $+\infty$.

        Hence, by the rigidity case of monotonicity formula, $V_\infty'$ is a regular cone with multiplicity one (namely $\mbfC_{\infty}$) in $\AAa(\hat r^-, \hat r^+)$; then by the Allard Regularity Theorem, $V_j'$ converges smoothly, and with unit multiplicity, to $\mbfC_\infty$ in that same annulus. This implies \[
            \lim_{j\to\infty} s_j^\pm = \hat{r}^\pm \,,
        \]
        where $s_j^\pm:= \bar{s}_j^\pm/s_j$, for either consistent choice of signs. It follows that, if $\hat r^- >0$, then $\bar{s}_j^-\leq 3r_j/2$ for all sufficiently large $j$; and, analogously, if $\hat r^+ <+\infty$, then $\bar{s}_j^+\geq 3/2$ for all sufficiently large $j$. To prove (\ref{Equ_App_UniqTan_s_j^pm attains max}), we are left to deal with the case when either $\hat{r}^- = 0$ or $\hat{r}^+ = +\infty$.
        
        Let $\Sigma:= \mbfC_\infty\cap \SSp^{N-1}$, and let here (with slight abuse of notation) $\nabla^\perp$ denote the connection on the normal bundle $T^\perp \Sigma$ in $\SSp^{N-1}$; let $u_j: \mbfA(s_j^-, s_j^+) \to T^\perp\mbfC_\infty$ be the graphical section defining $V_j'$ over $\mbfC_\infty$ in $\AAa( s_j^-, s_j^+)$. Furthermore, set $t:= -\log r$, $T_j^\pm := -\log(s_j^\pm)$, $v_j(t, \theta):= r^{-1}u_j(r\theta)$ defined on $\Sigma\times (T_j^-, T_j^+)$. 
        Note that for all sufficiently large $j$
        we have \[
         |v_j|_2^*(t):= \sum_{i+j\leq 2, i, j\geq 0} |\partial_t^i(\nabla^\perp)^jv_j(\cdot, t)|_{C^0(\Sigma)} \leq C(\mbfC_\infty)\delta_j\,, \qquad \forall\, t\in [0, \ln 2]\,,  
        \] 
        and by taking $\eps'(\mbfC_\infty)$ small enough (which is fixed here, as determined by the constant $\delta_0$ in \cite[Theorem 11.1]{Edelen2021}), following the same argument as \cite[pages 29 - 32]{Edelen2021}, the assumptions in \cite[Theorem 11.1]{Edelen2021} are satisfied by $v_j$ on $\Sigma\times [0, T_j^+)$ with $n$ in place of $m$, which --- in turn --- implies that for some $\alpha = \alpha(\mbfC_\infty)\in (0, 1/2)$, 
        \begin{align}
          |v_j|_2^*(t) \leq C(\mbfC_\infty)\delta_j^{\alpha}\,, \qquad \forall\, t\in (0, T_j^+)\,.  \label{Equ_App_UniqTan_|v_j|_2^*< delta_j}            
        \end{align}
        
        Also note that the assumption $m>0$ in \cite[Theorem 11.1]{Edelen2021} can be replaced by $|m|\neq 0$. In fact, in \cite{Edelen2021} and \cite{Sim83}, the only place where $m>0$ is used is in the proof of the stability inequality \cite[Lemma 1]{Sim83}, which lead to the $L^1$ estimate \cite[(6.34)]{Sim83} on the interval where the neutral mode dominates $\partial_t v$. This estimate can be still derived when $m<0$ following \cite[4.26-4.32]{Sim85}. 
        By considering $\check v_j(t, \theta):= v_j(\ln 2 -t, \theta)$ and repeating the argument above, the assumptions in \cite[Theorem 11.1]{Edelen2021} are satisfied by $\check v_j$ on $\Sigma\times [0, \ln 2 - T_j^-)$ with $-n$ in place of $m$, which again implies that for some $\alpha' = \alpha'(\mbfC_\infty)\in (0, 1/2)$ there holds 
        \begin{align}
            |\check v_j|_2^*(t) \leq C(\mbfC_\infty)\delta_j^{\alpha'}\,, \qquad \forall\, t\in (0, \ln 2- T_j^-) \,.  \label{Equ_App_UniqTan_|check v_j|_2^*< delta_j}            
        \end{align}

        Combining \eqref{Equ_App_UniqTan_|v_j|_2^*< delta_j} and \eqref{Equ_App_UniqTan_|check v_j|_2^*< delta_j}, we conclude that $V_j'$ is a $\delta_j'$-$C^2$ graph over $\mbfC_\infty$ in $\AAa(s_j^-, s_j^+)$ for some sequence $\delta'_j\to 0$. Therefore, again by the Allard's Compactness Theorem, the characterization of the rigidity case in the  monotonicity formula and Allard Regularity Theorem, for any infinite subset $I\subset \NN$, $V_j^\pm:= (\eta_{s_j^\pm}^{-1})_\# V_j' = (\eta_{\bar{s}_j^\pm}^{-1})_{\#}V_j$ will (in both cases) still $C^2_{loc}$-converge to $\mbfC_\infty$ as $I\ni j\to \infty$ in the annuli
        \[
            \AAa^\pm(I) := \bigcup_{k\geq 1}\bigcap_{j\in I, j\geq k} \eta_{\bar{s}_j^\pm}^{-1}(A^{g_j}(\orig, r_j, 2))\,. 
        \]
        
        If \eqref{Equ_App_UniqTan_s_j^pm attains max} fails for all sufficiently large $j$, then there exists an infinite subset $\bar I\subset \NN$ such that either $\AAa^-(\bar I)\supset \AAa(2/3, 2)$, or $\AAa^+(\bar I)\supset \AAa(1/2, 4/3)$.
        But by tracing back the scaling factor, in the former case, $V_j$ is an $\eps^-_j$-$C^2$ graph over $\mbfC_\infty$ in \[
          \AAa(\bar s_j^-, \bar s_j^+)\cup \AAa(3\bar s_j^-/4, 7s_j^-/4) \supset \AAa(3\bar s_j^-/4, \bar s_j^+) \,;
        \]   
        for $j\in \bar I$ and some $\eps^-_j\to 0$. (We wish to stress that $\varepsilon_j^-$-graphicality in the first interval on the left-hand side follows from the fact that $V_j'$ is $\delta_j'$-$C^2$ graph over $\mbfC_\infty$ in $\AAa(s_j^-, s_j^+)$, while for the second interval on left-hand side the claim follows from the $C^2_{loc}$-convergence of $V_j^-$ to $\mbfC_\infty$ in $\AAa(2/3, 2)\subset\AAa^-(\bar I)$.) 
        Similarly, $\AAa^+(\bar I)\supset \AAa(1/2, 4/3)$ implies that $V_j$ is an $\eps^+_j$-$C^2$ graph over $\mbfC_\infty$ in \[
          \AAa(\bar s_j^-, \bar s_j^+)\cup \AAa(3\bar s_j^+/4, 5s_j^+/4) \supset \AAa(\bar s_j^-, 5\bar s_j^+/4) \,;
        \]   
        for $j\in \bar I$ and $\eps^+_j\to 0$. 
        
     Thus, either way, when $j$ is large enough $\AAa(\bar{s}_j^-, \bar{s}_j^+)$ would not be the largest annulus in $A^{g_j}(\orig, r_j, 2)$ centered at $\orig$ in which $V_j$ is an $\eps'$-$C^2$ graph over $\mbfC_\infty$. This is a contradiction.
    \end{proof}

    This in particular implies the following uniform convergence result at all scales, which we will repeatedly employ in Appendix \ref{app:GraphicalPar}.
    
    \begin{corollary} \label{Cor_Converg in all Scales}
        Let $\Lambda>0$, $\{\Sigma_j\}_{j\geq 1}$ be a sequence of $n$-dimensional $\MSI$ in $(\BB(4), g_j)$ with  $\Sing(\Sigma_j)=\{\orig\}$ and $\mfr_{\Sigma_j, g_j} \geq \Lambda^{-1}\rho_j$ in $\BB(3)$, where $\rho_j$ denotes the distance to $\orig$ in metric $g_j$. Suppose that, as one lets $j\to \infty$, there holds
        \begin{enumerate}[label=(\roman*)]
            \item\label{enum_Pre_Metric Converg} $\|g_j - g_{\eucl}\|_{C^4(\BB(4))}\to 0$,
            \item\label{enum_Pre_Density Drop Converg} $\theta_{|\Sigma_j|}(\orig, 2) - \theta_{|\Sigma_j|}(\orig) \to 0$.
        \end{enumerate}
        Then there exists some $\mbfC\in \cC_{N, n}(\Lambda)$ such that after passing to a subsequence (which we do not rename), both $\Sigma_j$ and its tangent cone $\mbfC_j$ at $\orig$ are all $\eps_j$-$C^2$ graphs over $\mbfC$ in $\BB(1)$ for some sequence $\left\{\eps_j\right\}_{j\geq j_0}$ such that $\eps_j\to 0$ as one lets $j\to\infty$.
    \end{corollary}
    \begin{proof}
        First note that, by the constraint on regularity scale function, $\mbfC_j\in \cC_{N, n}(\Lambda)$. By Lemma \ref{Lem_cC_(N,n)(Lambda) bF Cpt}, $\mbfC_j$ subconverges to some $\mbfC$ smoothly away from $\orig$. This implies, in particular, that for any $j\geq j_0$ large enough $\Sigma_j$ is a $\kappa_j$-$C^2$ graph over $\mbfC$ in $A^{g_j}(\orig,s_j,2s_j)$ for suitable sequences $\left\{\kappa_j\right\}_{j\geq j_0}$ and $\left\{s_j\right\}_{j\geq j_0}$ both tending to zero as one lets $j\to\infty$.
        
        Hence, set $\eps_j=1/j$ let $\delta_{j}:=\delta(1/j, \Lambda)$ be as afforded by Proposition \ref{Pro_Ana on SMC_Quanti Uniqueness of Tang Cone}. Thanks to the assumptions \ref{enum_Pre_Metric Converg}, \ref{enum_Pre_Density Drop Converg}, for any $j\geq j_0$ there exists $i_0=i_0(j)$ such that
      for any $i\geq i_0(j)$
          the hypotheses of  Proposition \ref{Pro_Ana on SMC_Quanti Uniqueness of Tang Cone} are satisfied by $\Sigma_i$ (and $g_i$) for any $i\geq i_0(j)$. Thus the sequence $\{\Sigma_{i_{0}(j)}\}_{j\geq j_0}$ has the desired properties. 
    \end{proof}

\section{Comparing graphical sections defining singular minimal submanifolds }\label{app:GraphicalPar}

\begin{lemma}\label{Lem_App_GraphParam_Compare diff graph func} There exists $\delta_0(N)\in (0, 1/4)$ and $C(N)>0$ such that for every $\delta\in (0, \delta_0)$ the following statements hold true. 

Let $M$ be a smooth manifold, $p\in M$, let $g^0$ (respectively: $g^1, g^2$)  be a $C^7$ (respectively: $C^5$) Riemannian metrics on $M$ such that $\injrad_{M,g^0}(p)>5$, $\|(\exp^{g^0}_p)^* g^i-g_\eucl\|_{C^{5}(\BB(5))}\leq \delta_0$, and for $i=0,1,2$ let $\Sigma^i$ be a submanifold in $M$ such that 
$(\exp^{g^0}_{p})^{-1}(\Sigma^i)$ is a $\delta_0$-$C^4$ graph over $\BB^n(3)\subset T_{p}\Sigma^i$; for $j=1, 2$ suppose further that $\Sigma^{j}$ are $\delta$-$C^2$ graphs over $\Sigma^0$ in $(B^{g^0}(2), g^0)$, as per Definition \ref{def:KappaGraph}. Under such assumptions, in what follows we will be writing, for the sake of notational simplicity, $g^i$ in lieu of $(\exp^{g_0})^*_p g^i$ and $\Sigma^i$ in lieu of $(\exp^{g^0}_{p})^{-1}(\Sigma^i)$. Then, given $\mbfx\in \RR^N$ with $|\mbfx|\leq \delta$:
\begin{enumerate}[label=(\roman*)]
\item\label{Item_GraphParam_Graph over each other} for all $i, j=0, 1, 2$, both $\Sigma^j$ and  $\Sigma^j+\mbfx$ are $\sqrt{\delta}$-$C^2$ graphs over $\Sigma^i$ in $(B^{g^i}(1), g^i)$; we denote by $u^j_i$ and $\tilde u^j_i$ the graphical section of $\Sigma^j$ and $\Sigma^j+\mbfx$ over $\Sigma^i$, respectively;
            \item
            \label{Item_GraphParam_Graph diff est} 
            when $g^0=g^1=g^2 =: g$ there holds,
            \be
              \frac12\|u^2_0 - u^1_0\|_{C^0(B^g(1/8)\cap \Sigma^0)}  \leq \|u^2_1\|_{C^0(B^g(1/4)\cap \Sigma^1)} \leq 2\|u^2_0 - u^1_0\|_{C^0(B^g(1/2)\cap \Sigma^0)}\,, \label{Equ_GraphParam_|u^2_0-u^1_0| approx |u^2_1|} 
              \ee
and, furthermore, for every $i, j=0, 1, 2$ there also holds
      \be        
              \left\|\tilde u^j_i - u^j_i - \mbfx^{\perp_{\Sigma^i, g}} \right\|_{C^0(B^{g}(1/4)\cap \Sigma^i)} \leq C \left(\|u^j_i\|_{C^0(B^g(1/2)\cap \Sigma^i)} + |\mbfx| \right)|\mbfx|\,. \label{Equ_GraphParam_|tilde u - u - x^perp| bd}
            \ee
        \end{enumerate}
    \end{lemma}

      \begin{proof}     
        \ref{Item_GraphParam_Graph over each other} follows directly from implicit function theorem, as we now explain. 
        It suffices to show that $\Sigma^2$ and $\Sigma^2+\mbfx$ are $\sqrt\delta$-$C^2$ graphs over $\Sigma^1$. For $i=0, 1, 2$, let $\mbfV_i$ be the $g^i$-normal bundle of $\Sigma^i$, let $\mbfV_i(R)$ denote the $R$-neighborhood of the zero section in $\mbfV_i$; and let $\mbfE_i: \mbfV_i\to \RR^N$ be the $g^i$-exponential map $(x, z)\mapsto \exp^{g^i}_x(z)$, that is of class $C^3$ (see Section \ref{sec:Prelim}). 
        By taking $\delta_0(N)\in (0, 1/4)$ sufficiently small, we can require that for some $C(N) > 1$ the map
        $\mbfE_i$ is a  diffeomorphism from $\mbfV_i(4\delta_0)$ to some domain in $\RR^N$ containing the $\delta_0$-neighborhood $U$, in metric $g^0$, of $\cup_{0\leq j\leq 2} (\Sigma^j\cap B^{g^j}(3/2))$, and $\|\mbfE_i\|_{C^3(\mbfV_i(4\delta_0))}, \|\mbfE_i^{-1}\|_{C^3(U)}\leq C$;
         furthermore there exists a local orthonormal frame $\{s_I\}_{I=1}^{N-n}$ of $\mbfV_1$ such that \[
              \rmP_i: \Sigma^i\times \RR^{N-n} \to \mbfV_i, \quad (x, z^I)\mapsto \left(x, \sum z^I s_I(x) \right)
          \]
          is a $C^3$ bundle isomorphism with $\|\rmP_i\|_{C^3(\Sigma^i\times \BB^{N-n}(1))} + \|\rmP^{-1}_i\|_{C^3(\mbfV_i(1))}\leq C$.
        
        Let then $\Pi_i: U\to \Sigma_i$, $\rmW_i: U\to \RR^{N-n}$ be the first and second component of $\rmP_i^{-1}\circ \mbfE_i^{-1}: U\to \Sigma^i\times \RR^{N-n}$ respectively; by the discussion above, $\|\Pi_i\|_{C^3(U)}, \|\rmW_i\|_{C^3(U)} \leq C$. 
        Therefore, we can define the following two maps, whose $C^2$ norms are also bounded by the same constant on their respective domains:
        \begin{align*}
          \Xi: (\Sigma^0\cap B^{g^0}(5/4))\times [0, 1]^2 & \to \Sigma^1, & & (x, s, t)\mapsto \Pi_1(\mbfE_0(x, (u^1_0 + s(u^2_0-u^1_0))(x)) + t\mbfx)\,, \\
          w: (\Sigma^0\cap B^{g^0}(5/4))\times [0, 1]^2 & \to \RR^{N-n}, & & (x, s, t)\mapsto \rmW_1(\mbfE_0(x, (u^1_0 + s(u^2_0-u^1_0))(x)) + t\mbfx) \,.
        \end{align*}
        By taking $\delta_0(N)\in (0, 1/4)$ even smaller if necessary, we have that:
        \begin{itemize}
          \item $\Xi(\cdot , 0, 0) = \mbfE_0(\cdot , u^1_0(\cdot))$ is a diffeomorphism onto a domain $\Omega_1$ of $\Sigma^1$ containing the set $\Sigma^1\cap B^{g^1}(6/5)$, with the estimate $|D\Xi(\cdot , 0,0)^{-1}|_{C^0(\Omega^1)}\leq C$ as well as \[
          \|\partial_s \Xi\|_{C^1((\Sigma^0\cap B^{g^0}(5/4))\times[0,1]^2)}, \|\partial_t\Xi\|_{C^1((\Sigma^0\cap B^{g^0}(5/4))\times[0,1]^2)}\leq C\delta\,.
          \]Hence, $\Xi_{s,t}:=\Xi(\cdot, s, t)$ are all diffeomorphisms with $\|\Xi_{s,t}^{-1}\|_{C^2(\Sigma^1\cap B^{g^1}(1))}\leq C$;
          \item since $w(x, 0,0) = 0$ and $\|\partial_sw\|_{C^2 ((\Sigma^0\cap B^{g^0}(5/4))\times [0, 1]^2)} + \|\partial_t w\|_{C^2((\Sigma^0\cap B^{g^0}(5/4))\times [0, 1]^2)}\leq C\delta$, we have for every $s, t\in [0, 1]$, $w_{s,t}:= w(\cdot, s, t)$ has $\|w_{s,t}\|_{C^2((\Sigma^0\cap B^{g^0}(5/4))\times [0, 1]^2)}\leq C\delta$.
        \end{itemize}
        Therefore, both $\Sigma^2$ and $\Sigma^2+\mbfx$ are graphs over $\Sigma^1$ in $B^{g^1}(1)$, with graphical sections \[
          u^2_1(x) = \rmP_1(x, w_{1, 0}\circ \Xi_{1, 0}^{-1}(x)) \,, 
          \qquad
          \tilde u^2_1 = \rmP_1(x, w_{1, 1}\circ \Xi_{1, 1}^{-1}(x)) \, 
        \] 
        and, by the estimates above, have $C^2$ norms on $\Sigma^1\cap B^{g^1}(1)$ bounded by $C\delta\leq \sqrt{\delta}$.

    Moving to item \ref{Item_GraphParam_Graph diff est}, to prove the left-hand side inequality of \eqref{Equ_GraphParam_|u^2_0-u^1_0| approx |u^2_1|}, consider the map \[
          \Phi: \Sigma^0 \times [0, 1] \to \RR^N\,, \quad (x, s)\mapsto E_1(E_0(x, u^1_0(x)), su^2_1(x))\,.
        \]
        Then by possibly taking $\delta_0$ smaller, for every $s\in [0, 1]$, $\Phi(\Sigma^0, s)$ is a $\sqrt{\delta}$-$C^2$ graph over $\Sigma^0$ in $B^g(1)$ with graphical section $v_s$. Moreover, $v_1 = u^2_0$, $v_0 = u^1_0$ and 
        \begin{align*}
            \|u^2_0 - u^1_0\|_{C^0(B^g(1/8)\cap\Sigma^0)} \leq \sup_{s\in [0, 1]} \|\partial_s v_s\|_{C^0(B^g(1/8)\cap\Sigma^0)} & \leq (1+C\delta_0) \|\partial_s \Phi\|_{C^0((B^g(1/4)\cap\Sigma^0)\times [0, 1])} \\
            & \leq (1+C\delta_0)\|u^2_1\|_{C^0(B^g(1/4)\cap\Sigma^1)}.
        \end{align*}
        where the second inequality follows from the implicit function theorem and the last inequality follows from a direct computation. Taking $\delta_0$ small enough finishes the proof.
    To prove the right-hand side of \eqref{Equ_GraphParam_|u^2_0-u^1_0| approx |u^2_1|}, one can essentially repeat the argument above for \[
          \Phi: (\Sigma^1\cap B^g(1)) \times [0, 1] \to \RR^N\,, \quad (x, s)\mapsto E_0(\Pi_0(x), (su^2_0 + (1-s)u^1_0)(x))\,.
        \]

        Lastly, in order to prove \eqref{Equ_GraphParam_|tilde u - u - x^perp| bd}, it suffices to only deal with the case when $j=1, i=0$ since the other cases are substantially analogous. Consider the map \[
          \Psi: \Sigma^0\times [0, 1]^2 \to \RR^N\,, \quad (x, s, t)\mapsto E_0(x, su^1_0(x)) + t\mbfx \,.
        \]
        Again, by possibly taking $\delta_0$ even smaller, for every $s, t\in [0, 1]$, $\Psi(\Sigma^0, s, t)$ is a graph (of class $C^2$) over $\Sigma^0$ in $B^g(1)$ with graphical section $w_{s, t}$. Moreover, $w_{1, 1} = \tilde u^1_0$, $w_{1, 0} = u^1_0$, $\partial_t w_{0, 0} = \mbfx^{\perp_{\Sigma^0, g}}$, and since $(x, w_{s,t}(x)) = E_0^{-1}(\Psi(x, s, t))$, by the chain rule we have for the first partial derivatives of $w_{s,t}$
  \begin{align*}
          \|\partial_s w_{s,t}\|_{C^0(B^g(1/4)\cap \Sigma^0)} & \leq C\|\partial_s\Psi\|_{C^0((B^g(3/8)\cap \Sigma^0)\times [0, 1]^2)} \leq C\|u_0^1\|_{C^0(B^g(1/2)\cap \Sigma^0)}\,, \\
          \|\partial_t w_{s,t}\|_{C^0(B^g(1/4)\cap \Sigma^0)} & \leq C\|\partial_t\Psi\|_{C^0((B^g(3/8)\cap \Sigma^0)\times [0, 1]^2)} \leq C|\mbfx| \,, 
          \end{align*}
          and for the second derivatives of later relevance there holds instead
          \begin{align*}
          \|\partial^2_{tt} w_{s,t}\|_{C^0(B^g(1/4)\cap \Sigma^0)} & \leq C(\|\partial_t\Psi\|^2_{C^0((B^g(3/8)\cap \Sigma^0)\times [0, 1]^2)} + \|\partial^2_{tt}\Psi\|_{C^0((B^g(3/8)\cap \Sigma^0)\times [0, 1]^2)}) \leq C|\mbfx|^2 \\
          \|\partial^2_{st} w_{s,t}\|_{C^0(B^g(1/4)\cap \Sigma^0)} & \leq C(\||\partial_t\Psi||\partial_s\varphi|\|_{C^0((B^g(3/8)\cap \Sigma^0)\times [0, 1]^2)} + \|\partial^2_{st}\Psi\|_{C^0((B^g(3/8)\cap \Sigma^0)\times [0, 1]^2)}) \\
          & \leq C|\mbfx|\|u_0^1\|_{C^0(B^g(1/2)\cap \Sigma^0)}\,.
        \end{align*} 

        Since \[
          w_{1,1} - w_{1,0} - \partial_t w_{0,0} = \int_0^1 \left( \int_0^1\frac{d}{d\lambda} (\partial_tw_{\lambda, t\lambda})\ d\lambda \right)\ dt = \int_0^1  \left( \int_0^1 \partial^2_{st} w_{\lambda, t\lambda} + t\partial^2_{tt}w_{\lambda, t\lambda}\ d\lambda \right)\ dt\,,
        \]
        we then find 
        \begin{align*}
            \left\|\tilde u^1_0 - u^1_0 - \mbfx^{\perp_{\Sigma^0, g}}\right\|_{C^0(B^g(1/4)\cap\Sigma^0)} & \leq \sup_{s, t\in [0, 1]} \left(\|\partial^2_{st} w_{s, t}\|_{C^0(B^g(1/4)\cap\Sigma^0)} + \|\partial^2_{tt} w_{s, t}\|_{C^0(B^g(1/4)\cap\Sigma^0)} \right) \\
            & \leq C \left(\|u^1_0\|_{C^0(B^g(1/2)\cap \Sigma^0)} + |\mbfx| \right)|\mbfx|\,,
        \end{align*} 
        as claimed.
    \end{proof}

    Next, we show that if two $\MSI$ with a common singular point are sufficiently close to each other, then each can be written as a graph of the other with effective estimates. 
    
    \begin{remark}\label{rem:RoleRefMetric}
    As stipulated in Section \ref{sec:GenReg}, in what follows we assume to have fixed a smooth Riemannian metric $\hat g$ on $M$, which we employ to measure the norm of tensors. We make the following observation: given $\Lambda>0$ if $\|g^0\|_{C^7(M)}\leq \Lambda$ there exist $\delta(\Lambda)>0$ and $C(\Lambda)>0$ such that $\|g^0-g^1\|_{C^5(M)}\leq\delta(\Lambda)$ implies $\|(\exp^{g^0}_p)^* g^0-(\exp^{g^0}_p)^* g^1\|_{C^5(\BB(1))}\leq C\|g^0-g^1\|_{C^5(M)}$.
    \end{remark}

     \begin{lemma} \label{Lem_Growth Estimate}
        Let $\gamma,\sigma \in  (0, 1)$ and $\Lambda>0$. Then there exist constants $\delta_1 = \delta_1(\sigma, \Lambda)\in (0,1)$, $K = K(\sigma)>2$ and $C_1(\sigma, \Lambda)>0$ such that the following holds.
Let $M$ be a smooth manifold, $p\in M$, and let $g^0$ (respectively: $g^1$) be a $C^7$-(respectively: $C^5$) Riemannian metric such that $\|g^0\|_{C^7(M)}\leq \Lambda$, $\injrad_{M,g^0}(p)>4$, $\|(\exp^{g^0}_p)^* g^i-g_\eucl\|_{C^{5}(\BB(4))}\leq \delta_0$, and for $i=0,1$ let $\Sigma^i$ be an $\MSI$ in $(M,g^i)$ such that:
\begin{enumerate}[label=(\roman*)]
 \item\label{item:3FirstLemGraph}  for $i=0, 1$, $\Sing(\Sigma^i)\cap B^{g^i}(p, 3) = \{p\}$,  and
            \[
                \theta_{|\Sigma^i|}(p, 2) - \theta_{|\Sigma^i|}(p) \leq \delta_1\,,
            \]
            as well as
            \[
                \dist_\R(\gamma, \Gamma(\mathbf{C}_p(\Sigma^i)) \cup \{- (n - 2)/ 2\}) \geq \sigma\,;
            \]
 \item\label{item:4FirstLemGraph}  for $i=0, 1$, the regularity scale function, defined in Definition \ref{def:RegularScale}, satisfies in $B^{g^i}(p, 3)$
            \[
              \mfr_{\Sigma^i, g^i} \geq \Lambda^{-1}\dist_{g^i}(\cdot, p) \,;
            \]
            
            \item\label{item:5FirstLemGraph}  the graphical radius, defined in Definition \ref{def:GraphRadius}, satisfies
            \[
                \bfr^{\Sigma^1}_{\Sigma^0, g^0}(p)\leq \delta_1\,;
            \]
            \item\label{item:LastLemGraph} the graphical section, $u:= \mbfG^{\Sigma^1}_{\Sigma^0, g^0}$, defined in Definition \ref{def:GraphRadius}, satisfies
            \[
                \|u\|_{C^2(A^{g^0}(p, 1, 2)\cap\Sigma^0)} \leq \delta_1\,.
            \]
\end{enumerate}

Then in $B^{g^0}(p, 1)$, $\Sigma^1$ is a graph of class $C^2$ over $(\Sigma^0, g^0)$, i.\,e., $\bfr^{\Sigma^1}_{\Sigma^0, g^0}(p) = 0$, and the graphical function $u$ satisfies the estimate,
        \[
            \|u\|_{C^0_{\gamma}(B^{g^0}(p, 1/2)\cap\Sigma^0)} \leq C_1(\sigma, \Lambda)\cdot \left(\|g^0-g^1\|_{C^5(M)} + \|u\|_{L^2_{\gamma}(A^{g^0}(p, K^{-3}, 1)\cap\Sigma^0)} \right) \,.
        \]
        \end{lemma}

        \begin{proof}
We only need to work in $T_pM$ by pulling back metrics and $\MSI$ using the exponential map in metric $g^0$; we shall thus apply the same notational conventions declared in the statement of Lemma \ref{Lem_App_GraphParam_Compare diff graph func}. It suffices to show that for any sequence of $\{\gamma_j\}_{j\geq 1}$ in $(0, 1)$, any sequence of pairs of metrics $\{(g^0_j, g^1_j)\}_{j \geq 1}$ and any sequence of pairs $\{(\Sigma^0_j, \Sigma^1_j)\}$, where $\Sigma^i_j \subset (\BB(3), g^i_j)$ is an $\MSI$ for $i =0, 1$, if \ref{item:3FirstLemGraph}-\ref{item:LastLemGraph} in the lemma hold for $g_j^i, \Sigma_j^i, \gamma_j, 1/j$ in place of $g^i, \Sigma^i, \gamma, \delta_1$, then the conclusion of the lemma holds for all sufficiently large $j$ and for some finite constant $C > 0$. Without loss of generality, we assume that $\Sigma^0_j \neq \Sigma^1_j$ for all $j$; otherwise, $\Sigma^1_j$ is a global graph over $(\Sigma^0_j, g^0_j)$ whose graphical section is identically zero.

        First notice that by Corollary \ref{Cor_Converg in all Scales}, as $j \to \infty$, up to subsequences, $\Sigma^i_j$ are $\eps_j$-$C^2$ graphs over $\mbfC^i$ for some $\eps_j\to 0$ and $\mbfC^i\in \cC_{N, n}(\Lambda)$, where $i=0, 1$.  Since by \ref{item:5FirstLemGraph}, \ref{item:LastLemGraph}, $\Sigma^1_j$ and $\Sigma^0_j$ are approaching each other in $A^{g^0}(\orig, 1, 2)$, we must have $\mbfC^0=\mbfC^1 =: \mbfC$. Then by Lemma \ref{Lem_App_GraphParam_Compare diff graph func} \ref{Item_GraphParam_Graph over each other}, for all $j$ large enough $\Sigma^1_j$ is a $\tilde\eps_j$-$C^2$ graph over $\Sigma^0_j$ in $B^{g^0_j}(\orig, 3/2)$, where $\tilde\eps_j\to 0$ as one lets $j\to\infty$.

By the definition of asymptotic rate we then have
        \[
            \cA\cR_p(u_j) \geq 1 > \gamma\,.
        \]
        In particular, by \ref{Item_MSE for g' minimal graph} of Proposition \ref{Prop_MSE_Main} and \ref{Item_MSE_SecTrans_Concl3} of Proposition \ref{Prop_MSE_Section Transp}, the same argument as the one presented in the proof of Corollary \ref{Cor_App_Uniform Growth Est for Jac Equ} applies to $u_j$ to give an $L^2_\gamma-$ estimate of $v_j := \mbfT^{\Sigma^0_j}_{\mbfC, g_\text{euc}}(u_j)$. Thus, for $K=K(\sigma)>2$ given by Lemma \ref{Lem_Ana on SMC_Growth Rate Monoton}, for every $\ell\geq 1$,
        \begin{align*}
          \|v_j\|_{L^2(\mbfA(K^{-\ell-1}, K^{-\ell}))} \leq C(\Lambda, \sigma)(\|g^1_j - g^0_j\|_{C^5(\BB(1))} + \|v_j\|_{L^2(\mbfA(K^{-2}, K^{-1}))})\cdot K^{-\ell(n/2+\gamma)} \,.
        \end{align*}
        The analogous $L^2_\gamma$ estimates for $u_j$ then follow then by combining the preceding inequality with \ref{Item_MSE_SecTrans_Concl1} of Proposition \ref{Prop_MSE_Section Transp}.
At that stage, the desired conclusion comes by putting together such a bound with the estimate provided by Corollary \ref{cor:L2toC0} (suitably recast via standard scaling arguments), appealing to Remark \ref{rem:RoleRefMetric} to bound the term involving the difference of the metrics. 
        \end{proof}
    
       More generally, we may consider the case of two $\MSI$ having distinct singular points that are separated by a sufficiently small distance. 

     \begin{corollary} \label{Cor_L^2 Growth Est}
        Given $\gamma, \sigma\in (0, 1)$ and $\Lambda>0$, let $\delta_1=\delta_1(\gamma, \sigma, \Lambda)\in (0,1)$ and $K = K(\sigma) > 2$ be as in Lemma~\ref{Lem_Growth Estimate}.  There exist constants $\delta_2 = \delta_2(\gamma, \sigma, \Lambda)\in (0, \delta_1)$, $\tilde \delta_2 = \tilde \delta_2(\gamma, \sigma, \Lambda) \in (0, \delta_2)$ and $C_2 = C_2(\gamma, \sigma, \Lambda) > 0$ such that the following holds. Let $M$ be a smooth manifold, $p\in M$, and let $g$ be a $C^7$ Riemannian metric such that $\|g\|_{C^7(M)}\leq \Lambda$, $\injrad_{M,g}(p)>4$, $\|(\exp^{g}_p)^*g-g_\eucl\|_{C^{5}(\BB(4))}\leq \tilde \delta_2$, and for $i=0,1$ let $\Sigma^i$ be an $\MSI$ in $(M,g)$ such that:
        \begin{enumerate}[label=(\roman*)]
            \item\label{item:DensityControlCorNormalGraph} for $i=0,1$, there exist $p^i\in B^g(p, \tilde\delta_2)$ such that $\Sing(\Sigma^i)\cap B^g(p^i, 3) = \{p_i\}$, and
            \[
                \theta_{|\Sigma^i|}(p^i, 2) - \theta_{|\Sigma^i|}(p^i) \leq \tilde \delta_2\,,
            \]
            as well as 
            \[
                \dist_\R(\gamma, \Gamma(\mathbf{C}_{p^i}(\Sigma^i)) \cup \{- (n - 2)/ 2\}) \geq \sigma\,;
            \] 
            in what follows let $\mbfx^i := (\exp_p^g)^{-1}(p^i)\in T_pM$; 
            \item\label{item:GraphParam_RegScaleBdBelow} for $i=0, 1$, the regularity scale function, defined in Definition \ref{def:RegularScale}, satisfies
            \[
              \mfr_{\Sigma^i, g} \geq \Lambda^{-1}\dist_g(\cdot, p^i) \,;
            \]
            \item\label{item:GraphParam_GraphRadBd} the graphical radius, defined in Definition \ref{def:GraphRadius}, satisfies
            \[
                \bfr := \bfr^{\Sigma^1}_{\Sigma^0, g}(p^0)\leq \tilde \delta_2\,;
            \]
            \item\label{item:LastCorNormalGraph} the graphical section, $u:= \mbfG^{\Sigma^1}_{\Sigma^0, g}$, defined in Definition \ref{def:GraphRadius}, satisfies
            \[
                \|u\|_{C^2(A^{g}(p^0,1/2, 5/2)\cap\Sigma^0)} \leq \tilde \delta_2\,.
            \]
        \end{enumerate}
        Then we have the following estimates,
        \begin{align}
            {\frac12}|\mbfx^1-\mbfx^0| \leq \bfr := \bfr^{\Sigma^1}_{\Sigma^0, g}(p^0) & \leq C_2\|u\|_{L^2(A^{g}(p^0, \delta_2, 2)\cap\Sigma^0)}\,, \label{Equ_App_GraphParam_|x^1-x^0| < |u|_L^2} \\
            \|u - (\mbfx^1-\mbfx^0)^{\perp_{\Sigma^0, g}}\|_{C^0 (A^{g}(p^0, s, 2s)\cap\Sigma^0)} & \leq C_2 \left(\|u\|_{L^2(A^{g}(p^0, \delta_2, 2)\cap\Sigma^0)}\cdot s^\gamma + s^{-1}|\mbfx^1-\mbfx^0|^2 \right)  \,. \label{Equ_App_GraphParam_|u - x^perp|_s < s^gamma|u|_L^2}
        \end{align}
        for every $s\in (C_2\bfr, 1/2)$. (Here and in what follows we view $\mbfx^1-\mbfx^0$ as a vector field on $B^g(p,4)\subset M$ by pushing forward using $\exp^g_p$.)
  In addition, for some $\tilde C_2(\gamma,\sigma,\Lambda) = \tilde C_2(C_2) > 0$, there holds
        \begin{equation}\label{eq:BasicEllEst}
            \|u\|_{C^0 (A^{g}(p^0, C_2 \bfr, 1)\cap\Sigma^0)} \leq \tilde C_2 \|u\|_{L^2(A^{g}(p^0, \delta_2, 2)\cap\Sigma^0)} \,.
        \end{equation}       
        Moreover, considering for $\mbfx\in T_pM$ with $|\mbfx|\leq 1$ the map $\tau_\mbfx^g$ given by
        \[
          \tau_\mbfx^g: B^g(p, 4)\to M, \quad q\mapsto \exp_p^g \left((\exp^g_p)^{-1}(q) + \mbfx \right) \,
        \] 
        and set $\tilde \Sigma^1 := (\tau^g_{\mbfx^1 - \mbfx^0})^{-1}(\Sigma^1)$ and $\tilde g := \left(\tau^g_{\mathbf{x}^1 - \mathbf{x}^0}\right)^*(g)$, in $B^g(p^0, 3)$, all  assumptions in Lemma~\ref{Lem_Growth Estimate} are satisfied with $\Sigma^0, g; \tilde \Sigma^1, \tilde g, p^0$ replacing $\Sigma^0, g^0; \Sigma^1, g^1, p$ therein.
    \end{corollary}
    
    \begin{proof}
         Throughout this proof, we denote for simplicity $\eta_{p, r}: T_pM\to M$ the map $v\mapsto \exp^g_p(rv)$; note that $\eta_{p, r}^{-1}$ is well-defined on $B^g(p, 4)$.

		\medskip
        \textbf{Claim 1.} There exists $\delta' > 0$ and $C' > 1$ such that if $\tilde\delta_2 \leq \delta'$ and $r\in (0, 1)$, then 
        \begin{equation}\label{eqn:x^i_dist}
            \frac12 |\mathbf{x}^1 - \mathbf{x}^0| \leq \bfr^{\Sigma^1}_{\Sigma^0, g}(p^0) \leq C' \|(\mathbf{x}^1 - \mathbf{x}^0)^{\perp_{\Sigma^0, g}}\|_{C^0(A^g(p^0, r, 2r)\cap\Sigma^0)}\,.
        \end{equation}
        \begin{proof}
            Since $p^i$ is a singular point of $\Sigma^i$, it is clear from the definition of graphical radius (Definition~\ref{def:GraphRadius}) that by taking $\tilde\delta_2>0$ sufficiently small, we have,
            \[
               \frac12|\mathbf{x}^1 - \mathbf{x}^0| \leq \dist_g(p^1, p^0) \leq \bfr^{\Sigma^1}_{\Sigma^0, g}(p^0)\,.
            \]
                To prove the right-hand side inequality, suppose for the sake of contradiction, that there exist sequences $\{\Sigma^0_j\}_{j\geq 1}$ and $\{\Sigma^1_j\}_{j\geq 1}$ where $\Sigma^0_j$ and $\Sigma^1_j$ are $\MSI$ in $(M, g_j)$ satisfying all assumptions in the preceding statement with $p_j, p^i_j, \mbfx^i_j$ in place of $p, p^i, \mbfx^i$ and $1/j$ in place of $\tilde \delta_2$, but
            \begin{align}
				\bfr^{\Sigma^1_j}_{\Sigma^0_j, g_j}(p^0_j) > j \|(\mathbf{x}^1_j - \mathbf{x}^0_j)^{\perp_{\Sigma^0_j, g_j}}\|_{C^0(A^{g_j}(p^0_j, s_j, 2s_j)\cap\Sigma^0_j)}\,. \label{Equ_App_GraphParam_r^0_1 > j|x^perp|} 
            \end{align}
            for some $s_j\in (0, 1)$. First note that for any $j$ sufficiently large, we must have 
            \begin{align}
              \sqrt{j} \|(\mathbf{x}^1_j - \mathbf{x}^0_j)^{\perp_{\Sigma^0_j, g_j}}\|_{C^0(A^{g_j}(p^0_j, s_j, 2s_j)\cap\Sigma^0_j)} \geq |\mbfx^1_j - \mbfx^0_j|\,.  \label{Equ_App_GraphParam_|x| < |x^perp|} 
            \end{align}
            This is because otherwise, by \ref{item:DensityControlCorNormalGraph} and Corollary~\ref{Cor_Converg in all Scales}, $\eta^{-1}_{p^0_j, s_j}(\Sigma^0_{j}\cap B^{g_j}(p_j, 3))$ would subconverge to some nontrivial regular minimal cone $\mbfC\in \cC_{N,n}(\Lambda)$, but the failure of \eqref{Equ_App_GraphParam_|x| < |x^perp|} implies that the subsequential limit $\hat\mbfx \in \RR^N$ of $(\mbfx^1_j - \mbfx^0_j)/|\mbfx^1_j - \mbfx^0_j|$ satisfies $\hat\mbfx^{\perp_{\mbfC, g_\eucl}} = 0$. Since $\hat\mbfx\neq 0$, this means $\mbfC$ splits in $\hat\mbfx$ direction, contradicting the fact that $\mbfC$ is a regular cone.
            
            Now let $\bfr_j := \bfr^{\Sigma^1_j}_{\Sigma^0_j, g_j}(p^0_j)$:  by \ref{item:DensityControlCorNormalGraph}, Corollary~\ref{Cor_Converg in all Scales}, \eqref{Equ_App_GraphParam_r^0_1 > j|x^perp|} and \eqref{Equ_App_GraphParam_|x| < |x^perp|}, $\hat{\Sigma}^i_j:= \eta^{-1}_{p^0_j, \bfr_j}(\Sigma^i_{j}\cap B^{g_j}(p_j, 3))$  subconverges to some $\mbfC\in \cC_{N,n}(\Lambda)$ for both $i=0, 1$. It then follows from Allard's regularity theorem~\cite{All72} and Lemma \ref{Lem_App_GraphParam_Compare diff graph func} \ref{Item_GraphParam_Graph over each other} that for sufficiently large $j$, 
            $\hat{\Sigma}^1_j$ is a $\delta^2_0/2$-$C^2$ graph over $\hat{\Sigma}^0_j$ in the Euclidean annulus centered at the origin and having radii $1/2$ and $2$, which contradicts the definition of $\bfr^{\Sigma^1_j}_{\Sigma^0_j, g_j}(\mbfx^0_j)$.
        \end{proof}

		In order to proceed with the next ancillary claim, recall the definitions of $\tilde\Sigma^1$, $\tilde g^1$ (given at the end of the statement); note that $\Sing(\tilde{\Sigma}^1)\cap B^{\tilde{g}}(p^0, 3) = \{p^0\}$ and let $\tilde u := \mathbf{G}^{\tilde \Sigma^1}_{\Sigma^0, g}$ be the corresponding graphical section (according to Definition \ref{def:GraphRadius}).

        \medskip

        \textbf{Claim 2.} There exists $\delta''(N, \Lambda) \in (0, \delta')$ such that if we choose $\tilde\delta_2 \leq \delta''$, then all assumptions in Lemma~\ref{Lem_Growth Estimate} are satisfied for $(\Sigma^0, \tilde \Sigma^1, g, \tilde g, p^0)$ in place of $(\Sigma^0, \Sigma^1, g^0, g^1, p)$. 
          In particular, $\tilde \Sigma^1$ is a global graph of class $C^2$ over $(\Sigma^0, g)$ with graphical section $\tilde u$ satisfying, 
        \begin{equation}\label{eqn:tilde_u_est}
            \|\tilde u\|_{C^0_\gamma(B^{g}(p^0, 1/2)\cap\Sigma^0)} 
            \leq C''(\sigma, \Lambda)\left(|\mbfx^1-\mbfx^0| + \|\tilde u\|_{L^2_{\gamma} (A^{g^0}(p^0, K^{-3}, 1)\cap\Sigma^0)} \right)\,.
        \end{equation}
    
        \begin{proof}
 
          The only non-straightforward task is
            to justify \ref{item:5FirstLemGraph} and \ref{item:LastLemGraph} in Lemma~\ref{Lem_Growth Estimate}.
            Working with 
           $\delta_0$ from Lemma \ref{Lem_App_GraphParam_Compare diff graph func} and $\delta_1$ from Lemma \ref{Lem_Growth Estimate}, we note that
         by our assumption \ref{item:GraphParam_RegScaleBdBelow}, taking $\delta''<\delta_1/4$, there exists $C(N, \Lambda)>4$ large enough that for every $q\in \Sigma^i\cap A^g(p, \delta_1/2, 3)$, after pulling back to $T_qM$ with the map $\eta_{q, \delta_1/C}$, $\Sigma^i$ is a $\delta_0$-$C^4$ graph over $T_q\Sigma^i$ for $i=0,1$.
          
            Therefore, by our assumptions \ref{item:GraphParam_GraphRadBd} and \ref{item:LastCorNormalGraph}, we can take sufficiently small $\delta''=\delta''(N, \Lambda)$ and apply Lemma \ref{Lem_App_GraphParam_Compare diff graph func} \ref{Item_GraphParam_Graph over each other} in the ball of center $q$ and radius $\delta_1/C$ for every $q\in A^g(p^0, \delta_1, 5/2)$ to verify the satisfaction of \ref{item:5FirstLemGraph} and \ref{item:LastLemGraph} in Lemma~\ref{Lem_Growth Estimate}.        \end{proof}
        
		\medskip

		\textbf{Claim 3.} There exist $\delta'''(\Lambda, \sigma) \in (0, \delta'')$, $C'''(\Lambda, \sigma)> 1$ and $C''''(\Lambda, \sigma)> 1$ such that if $\tilde\delta_2 \leq \delta'''$ and $r \in ((C''')^2\bfr^{\Sigma^1}_{\Sigma^0, g}(p^0), 1/2)$, then we have, 
		\begin{equation}\label{eqn:tilde_u_comp}
            \|u - \tilde u - (\mbfx^1-\mbfx^0)^{\perp_{\Sigma_0, g}}\|_{C^0(A^g(p^0, 2r, 3r)\cap\Sigma^0)} \leq C'''\left(\|\tilde u\|_{C^0(A^g(p^0, r, 4r)\cap\Sigma^0)} + |\mbfx^1 - \mbfx^0|\right)r^{-1}\cdot |\mbfx^1 - \mbfx^0|\,,
        \end{equation}
		\begin{equation}\label{eqn:tilde_tilde(u)_comp}
            \|u - \tilde u - (\mbfx^1-\mbfx^0)^{\perp_{\Sigma_0, g}}\|_{C^0(A^g(p^0, 2r, 3r)\cap\Sigma^0)} \leq C'''\left(\|u\|_{C^0(A^g(p^0, r, 4r)\cap\Sigma^0)} + |\mbfx^1 - \mbfx^0|\right)r^{-1}\cdot |\mbfx^1 - \mbfx^0|\,,
        \end{equation}
		and
		\begin{align}
			|\mbfx^1-\mbfx^0|\leq C'''' (\|\tilde u\|_{C^0(A^g(p^0, r, 4r)\cap\Sigma^0)} + \|u\|_{C^0(A^g(p^0, 2r, 3r)\cap\Sigma^0)}) \,. \label{eqn:x1_x0_dist_tilde_u_comp}
		\end{align}
		\begin{proof}

          In view of Remark \ref{Rem_Choice of smallness in graph rad}, by choosing sufficiently small $\delta'''$ and large $C'''$, \eqref{eqn:tilde_u_comp} and \eqref{eqn:tilde_tilde(u)_comp} follow directly by applying Lemma \ref{Lem_App_GraphParam_Compare diff graph func} near each $q\in A^g(p^0, r, 2r)$. At that stage, \eqref{eqn:x1_x0_dist_tilde_u_comp} then follows from \eqref{eqn:tilde_u_comp}, Claim 1 and the triangle inequality by taking $C'''$ much larger than $C'$.
		\end{proof}

        We are finally in the position of proving \eqref{Equ_App_GraphParam_|x^1-x^0| < |u|_L^2}, \eqref{Equ_App_GraphParam_|u - x^perp|_s < s^gamma|u|_L^2} and \eqref{eq:BasicEllEst}.
		In what follows we will always take $\tilde\delta_2 \leq \delta'''$ where the latter constant is as afforded by the previous claim. Then for any $r \in ((C''')^2\bfr^{\Sigma^1}_{\Sigma^0, g}(p^0), 1/8)$, we have 
		\begin{align}\begin{split}
            |\mathbf{x}^1 - \mathbf{x}^0| & \leq C(\|\tilde u\|_{C^0(A^g(p^0, r, 4r)\cap\Sigma^0)} + \|u\|_{C^0(A^g(p^0, 2r, 3r)\cap\Sigma^0)}) \\
                & \leq C r^{\gamma} \|\tilde u\|_{C^0_{\gamma}(A^g(p^0, r, 4r)\cap\Sigma^0)} + C\|u\|_{C^0(A^{g}(p^0, 2r, 3r)\cap\Sigma^0)} \\ 
                &\leq Cr^{\gamma} \cdot (|\mbfx^1-\mbfx^0| + \|\tilde u\|_{C^0(A^{g^0}(p^0, K^{-3}, 1)\cap\Sigma^0)}) + C \|u\|_{C^0(A^{g}(p^0,2r, 3r)\cap\Sigma^0)} \\ 
                &\leq C r^{\gamma} |\mbfx^1 - \mbfx^0| + C\|u\|_{C^0(A^{g}(p^0, \min\{2r, K^{-3}/2\}, 4/3)\cap\Sigma^0)}\,. 
        \label{Equ_App_GraphParam_Est |x^1-x^0| by |u|}
        \end{split}
        \end{align}
		Here and below, we remind the reader that the constant $C>1$ are all only depending on $\Lambda, \sigma$ (and possibly on $N$) and are changing from line to line; the first inequality follows from \eqref{eqn:x1_x0_dist_tilde_u_comp}; the second inequality follows from the definition of $C^0_\gamma$ and \ref{item:GraphParam_RegScaleBdBelow}; the third inequality follows from \eqref{eqn:tilde_u_est}; the last inequality follows from triangle inequality, \eqref{eqn:tilde_tilde(u)_comp}, Claim 1 and the allowed range for $r$. Now we fix the choice of $\delta_2(\Lambda, \sigma)$ so small that $C \delta_2^\gamma < 1/2$ and $4\delta_2<K^{-3}$.

        Recalling that, by Corollary \ref{cor:L2toC0}, we have 
        \begin{align}
           \|u\|_{C^0(A^{g}(p^0, 2r, 16/9)\cap\Sigma^0)} \leq C(r)\|u\|_{L^2(A^{g}(p^0, r, 2)\cap\Sigma^0)}\,. \label{Equ_App_GraphParam_C^0 bd by L^2} 
        \end{align}
 \eqref{Equ_App_GraphParam_|x^1-x^0| < |u|_L^2} finally follows by combining \eqref{Equ_App_GraphParam_Est |x^1-x^0| by |u|}, \eqref{Equ_App_GraphParam_C^0 bd by L^2} (with $r=\delta_2$) and \eqref{eqn:x^i_dist} with $\tilde \delta_2\in (0, \delta_2)$ small enough.
Next, to prove \eqref{Equ_App_GraphParam_|u - x^perp|_s < s^gamma|u|_L^2}, notice that for every $s\in ((C''')^2\bfr, 1/3)$ we have 
        \begin{align*}
          \|u - (\mbfx^1-\mbfx^0)^{\perp_{\Sigma^0, g}}\|_{C^0(A^g(p^0, 2s, 3s)\cap\Sigma^0)} & \leq C\left(\|\tilde u\|_{C^0(A^g(p^0, s, 4s)\cap\Sigma^0)} + s^{-1}\cdot|\mbfx^1 - \mbfx^0|^2\right) \\
          & \leq C\left(\|\tilde u\|_{C^0_\gamma(A^g(p^0, s, 4s)\cap\Sigma^0)}\cdot s^\gamma + s^{-1}\cdot|\mbfx^1 - \mbfx^0|^2\right) 
        \end{align*}
        Here, the first inequality follows again from triangle inequality, \eqref{eqn:tilde_u_comp}, Claim 1 and choice of $s$; and the second inequality follows by definition of $C^0_\gamma$ norm and \ref{item:GraphParam_RegScaleBdBelow}.
On the other hand,
        by \eqref{eqn:tilde_u_est}, \eqref{eqn:tilde_tilde(u)_comp}, \eqref{Equ_App_GraphParam_C^0 bd by L^2} and \eqref{Equ_App_GraphParam_|x^1-x^0| < |u|_L^2}, we have
        \[
          \|\tilde u\|_{C^0_\gamma(A^g(p^0, s, 4s)\cap\Sigma^0)} \leq C \left(|\mbfx^1-\mbfx^0| + \|u\|_{C^0(A^{g}(p^0, K^{-3}/2, 16/9))} \right) \leq C\|u\|_{L^2(A^{g}(p^0, \delta_2, 2)\cap \Sigma^0)}.
        \]
        We can now choose $C_2$ in terms of $C'''$ and the constants appearing in the right-hand side of the preceding two inequalities (in particular $C_2$ shall be larger than $2(C''')^2$) to obtain \eqref{Equ_App_GraphParam_|u - x^perp|_s < s^gamma|u|_L^2}, which completes the proof.
          \end{proof}
    
     We also need --- for the purpose of Lemma \ref{Lem_Tangent vectors to scL^k} --- to compare the graphical sections defining three different submanifolds. 

    \begin{corollary} \label{Cor_L^2 Growth est for pair}
        Let $\gamma, \sigma\in (0, 1)$ and $\Lambda>0$. Then there exists $\delta_3 = \delta_3(\gamma, \sigma, \Lambda)\in (0,1)$, $\tilde \delta_3 = \tilde \delta_3(\gamma, \sigma, \Lambda)\in (0, \delta_3)$ and $C_3 = C_3(\gamma, \sigma, \Lambda) > 2$ such that the following holds. Let $M$ be a smooth manifold, $p\in M$, and let $g$ be a $C^7$ Riemannian metric such that $\|g\|_{C^7(M)}\leq \Lambda$, $\injrad_{M,g}(p)>5$, $\|(\exp^{g}_p)^*g-g_\eucl\|_{C^{5}(\BB(5))}\leq \tilde \delta_3$, and for $i=0,1, 2$ let $\Sigma^i$ be an $\MSI$ in $(M,g)$ such that:
        \begin{enumerate}[label=(\roman*)]
           
            \item for $i=0,1,2$, there exists $p^i\in B^g(p, \tilde\delta_3)$ such that $p^0=p$, $\Sing(\Sigma^i)\cap B^g(p^i, 3) = \{p_i\}$, and: 
			\[
				\theta_{|\Sigma^i|}(p^i, 2) - \theta_{|\Sigma^i|}(p^i) \leq \tilde \delta_3\,, 
			\]
            as well as
            \[
                \dist_\R(\gamma, \Gamma(\mathbf{C}_{p^i}(\Sigma^i)) \cup \{- (n - 2)/ 2\}) \geq \sigma\,;
            \]
            in what follows let $\mbfx^i := (\exp_p^g)^{-1}(p^i)\in T_pM$ for $i=0,1,2$ (so that $\mbfx^0=\orig$); 
		  \item for $i=0,1,2$, the regularity scale function, defined in Definition \ref{def:RegularScale}, satisfies
			\[
					\mfr_{\Sigma^i, g} \geq \Lambda^{-1}\dist_g(\cdot, p^i)\,;
			\]
		  \item for $i=1,2$, the graphical radius satisfies
            \[
                \bfr^{\Sigma^i}_{\Sigma^0, g}(p)\leq \tilde \delta_3\,;
            \]
            \item for $i=1,2$, the graphical section $u^{(i)}:= \mbfG_{\Sigma^0, g}^{\Sigma^i}$ satisfies $\|u^{(i)}\|_{L^2(A^g(p,1/2, 2)\cap\Sigma^0)} \leq \tilde \delta_3$\,.
        \end{enumerate}
         
        Then for $i=1,2$, we have 
        \be\label{eq:LastCompareEasy}
           \frac{1}{2} |\mbfx^i| \leq  \bfr^{\Sigma^i}_{\Sigma^0, g}(p) \leq  C_3\|u^{(i)}\|_{L^2(A^{g}(p, \delta_3, 2)\cap\Sigma^0)}\,,
        \ee
        and, set $\bfr:= \max_i \bfr^{\Sigma^i}_{\Sigma^0, g}(p)$, $\mbfx:= \mbfx^2-\mbfx^1$, then for every $s\in (C_3 \bfr, 1/4)$ there holds
        \begin{align*}
            \|(u^{(1)} - u^{(2)}) & + \mbfx^{\perp_{\Sigma^0, g}}\|_{C^0(A^g(p, s, 2s)\cap\Sigma^0)} \\
            & \leq C_3 \|\mbfG^{\Sigma^2}_{\Sigma^1, g}\|_{L^2(A^{g}(p, \delta_3, 2)\cap\Sigma^1)} \cdot 
            \Big(s^\gamma + s^{-1}\left( \|u^{(2)}\|_{L^2(A^g(p, \delta_3, 2)\cap\Sigma^0)} + |\mbfx^1|+|\mbfx^2|\right) \Big) \,.
        \end{align*}
    \end{corollary}
    \begin{proof}
The first claim, i.\,e. \eqref{eq:LastCompareEasy} follows directly from the corresponding conclusion in Corollary \ref{Cor_L^2 Growth Est}. To proceed and justify the other estimate, 
 define
 $   \tilde \Sigma^2 := (\tau^g_{\mbfx})^{-1}\Sigma^2), \tilde g := \left(\tau^g_{\mathbf{x}}\right)^*(g)\,.
        $
     Let $\tilde u^{(2)}:= \mbfG_{\Sigma^0, g}^{\tilde\Sigma^2}$ be the graphical section of $\tilde\Sigma^2$ over $\Sigma^0$, and $v:= \mbfG^{\Sigma^2}_{\Sigma^1, g}$, $\tilde v:= \mbfG^{\tilde\Sigma^2}_{\Sigma^1, g}$ be the graphical section of $\Sigma^2$ and $\tilde\Sigma^2$ over $\Sigma^1$ respectively. 
        By making $\tilde\delta_3$ small enough, by the last assertion in the statement of Corollary~\ref{Cor_L^2 Growth Est}, we know that the assumptions in Lemma \ref{Lem_Growth Estimate} hold for $\Sigma^1, g, \tilde\Sigma^2, \tilde g$. Therefore, since $\tilde\Sigma^2$ is $\tilde g$-minimal and 
        $\|(\exp^{g}_p)^*{\tilde g} - (\exp^{g}_p)^*g\|_{C^5(\BB(4))} \leq C|\mbfx|$, we have for every $s\in (0, 1/6)$,
        \begin{align}
         \begin{split}
          \|\tilde v\|_{C^0(A^g(p^1, s/2, 3s)\cap\Sigma^1)} & \leq C \left( |\mbfx| + \|\tilde v\|_{C^0(A^g(p^1, K^{-4}, 1)\cap\Sigma^1)} \right)\cdot s^\gamma \\ 
          & \leq C\left( |\mbfx| + \|v\|_{C^0(A^g(p^1, K^{-5}, 3/2)\cap\Sigma^1)} \right)\cdot s^\gamma \\ 
          & \leq C\hspace{1mm}\|v\|_{L^2(A^g(p^1, \min\{\delta_2, K^{-6}\}, 2)\cap\Sigma^1)} \cdot s^\gamma \,. 
         \end{split} \label{Equ_App_GraphParam_|tilde v|< |v|s^gamma}
        \end{align}
        where the first inequality follows by Lemma \ref{Lem_Growth Estimate}, the second inequality follows from \eqref{Equ_GraphParam_|tilde u - u - x^perp| bd} and the last follows from \eqref{Equ_App_GraphParam_|x^1-x^0| < |u|_L^2} in Corollary \ref{Cor_L^2 Growth Est} to absorb $|\mbfx|$ and Corollary \ref{cor:L2toC0} if $\tilde\delta_3, \delta_3$ are taken small enough. (Here and below it is understood that $C=C(\gamma,\sigma,\Lambda)$.)

        Also, by applying \eqref{Equ_App_GraphParam_|u - x^perp|_s < s^gamma|u|_L^2} to the data $\Sigma^0, \Sigma^2, g$ and choosing $\delta_3\leq \delta_2$, $\tilde\delta_3\leq \tilde\delta_2$, we have for every $s\in (C_2\bfr, 1/2)$, 
        \begin{equation}
            \|u^{(2)}\|_{C^0(A^g(\mathbf{0}, s, 2s)\cap\Sigma^0)} \leq C \left(\|u^{(2)}\|_{L^2(A^g(\orig, \delta_3, 2)\cap\Sigma^0)} \cdot s^\gamma + |\mbfx^2| \right) \,. \label{Equ_App_GraphParam_|u^2|_A(s) < |u^2|_0 s^gamma + |x^2|}      
        \end{equation}
            On the other hand, Lemma \ref{Lem_App_GraphParam_Compare diff graph func} implies that for every $s\in (4C_2\bfr, 1/4)$, 
        \begin{align*}
            & \|(u^{(1)} - u^{(2)}) + \mbfx^{\perp_{\Sigma^0, g}}\|_{C^0(A^g(p, s, 2s)\cap\Sigma^0)} \\
             \leq\ & \|(\tilde u^{(2)} - u^{(2)}) + \mbfx^{\perp_{\Sigma^0, g}}\|_{C^0(A^g(p, s, 2s)\cap\Sigma^0)} + \|u^{(1)} - \tilde u^{(2)}\|_{C^0(A^g(p, s, 2s)\cap\Sigma^0)} \\
             \leq\ & C\left(\|u^{(2)}\|_{C^0(A^g(p, s/2, 3s)\cap\Sigma^0)} + |\mbfx|\right)s^{-1}\cdot|\mbfx| + C\|\tilde v\|_{C^0(A^g(p^1, s/2, 3s)\cap \Sigma^1)}\,.
        \end{align*}   Here, the extra $s^{-1}$ comes from rescaling $B^g(p, 2s)$ to a ball of radius $2$ as in Lemma~\ref{Lem_App_GraphParam_Compare diff graph func}. Combining this inequality with \eqref{Equ_App_GraphParam_|tilde v|< |v|s^gamma} and \eqref{Equ_App_GraphParam_|u^2|_A(s) < |u^2|_0 s^gamma + |x^2|} allows to conclude the proof.
    \end{proof}

\section{Parametrizing the space of MSI in a Riemannian manifold} \label{Sec:Count_Decomp}

    For $N > n\geq 2$ be integers, recall that $\mathcal{C}_{N,n}$ is the collection of \emph{non-trivial} regular $n$-dimensional minimal cones $\mathbf{C}$ in $\mathbb{R}^N$. For each $\mathbf{C} \in \mathcal{C}_{N,n}$ and every $x \in \mathbf{C} \setminus \{\mathbf{0}\}$, we have (essentially as a specification of Definition \ref{def:RegularScale})
    \[
    \mathfrak{r}_\mathbf{C}(x) := \sup\left\{r > 0 : \mbfC\cap \BB(x,r)=\graph_{T_x \mbfC}(u), \ r^{-1}|u|+|\mathring{\nabla} u|+r|\mathring{\nabla}^2 u|\leq 1\right\}
    \]
for $\phi:\textrm{dom}(\phi)\subset T_x\mbfC\to T^{\perp}_x\mbfC$, of class $C^2$.
Note that $r_\mathbf{C}$ is $1$-homogeneous in the radial direction. To obtain a compactness result, Lemma \ref{Lem_cC_(N,n)(Lambda) bF Cpt}, for each $\Lambda \geq 1$ we had set
    \[
        \mathcal{C}_{N,n}(\Lambda):=\{\mbfC\in \cC_{N,n}:  \inf_{\mbfC \cap \SSp^{N-1}} \mathfrak{r}_\mbfC(x) \geq \Lambda^{-1}\};
    \] clearly, there holds $\mathcal{C}_{N, n} = \bigcup^\infty_{\Lambda = 1} \mathcal{C}_{N, n}(\Lambda).$
    
\vspace{3mm}

\textbf{Multiplicity-one cone decomposition.}
    In this appendix, we adapt a result of Edelen~\cite{Edelen2021} to ``parametrize'' the class of all minimal submanifolds with strongly isolated singular points ($\MSI$) inside a given Riemannian manifold. It is important to note that in his original work, Edelen considered minimal cones and varifolds with higher multiplicities, whereas for our applications it suffices to restrict to multiplicity-one objects. On the other hand Edelen focused on hypersurfaces, where there is a natural hierarchy among the minimal graphs over a given region;  this hierarchy does not exist in the submanifold setting (namely: for general dimension and codimension), so we need to exploit partly different ideas.

    \begin{definition}[Strong-cone region]\label{Def:strong_cone_region}
        Let $g$ be a $C^2$ metric on $\mathbb{B}(a, R) \subset \mathbb{R}^N$, and $V$ be an $n$-dimensional integral varifold on $(\mathbb{B}(a, R), g)$. Given $\mathbf{C} \in \mathcal{C}_{N, n}$, $\beta\in [0, 1/4]$, $\rho \in [0, R]$, we say that $V\llcorner(\mathbb{A}(a, \rho, R), g)$ is a \textbf{$(\mathbf{C}, \beta)$-strong-cone region} if there is a $C^\infty$ normal section 
        \[
            u: (a + \mathbf{C}) \cap \mathbb{A}(a, \rho/8, R) \rightarrow \mathbf{C}^\perp
        \] so that for any $r \in [\rho, R] \cap (0, \infty)$ there holds $V \llcorner  \mathbb{A}(a, \rho/8, R) = |\graph_{a + \mathbf{C}}(u)|$ (i.\,e. $V$ coincides in that annulus with the multiplicity one varifold associated to the graph of $u$) and we have
        \begin{enumerate}[label=(\roman*)]
            \item small $C^2$ norms: $r^{-1}|u| + |\nabla u| + r|\nabla^2 u| \leq \beta$;
            \item almost constant density ratios: $\theta_\mathbf{C}(0) - \beta \leq \theta_V(a, r) \leq \theta_\mathbf{C}(0) + \beta$.
        \end{enumerate}
        In this case, for simplicity, we will also call $\mathbb{A}(a, \rho, R)$ a strong-cone region for $V$.
    \end{definition}
    \begin{remark}
        In~\cite[Definition~6.0.1,~Definition~6.0.3]{Edelen2021}, Edelen introduced notions of a \emph{weak-cone region} and a \emph{strong-cone region}, respectively. It is then shown in~\cite[Definition~6.1(3)]{Edelen2021} that, when the parameters $\beta$ and $\tau$ are sufficiently small depending on the cone $\mathbf{C}$, a weak-cone region is in fact also a strong-cone region. This result relies on a compactness theorem (Theorem 5.1 therein), which can be replaced by Lemma~\ref{Lem_cC_(N,n)(Lambda) bF Cpt} to obtain the same result in our (more general) setting. Therefore, for simplicity, in this manuscript we shall only employ the notion of strong-cone region.
    \end{remark}
    
    As discussed prior to the definition, we consider only multiplicity-one objects, so we fix the multiplicity parameter $m$ in~\cite[~Definition~6.0.3]{Edelen2021} to be $1$. This choice applies throughout this section.

    \begin{definition}[Smooth model]\label{Def:smooth_model}
        Given parameters $\Lambda, \gamma \in (0, \infty), \sigma \in (0, 1/3)$, a tuple $(S, \mbfC, \set{(\mbfC_\alpha,  \mathbb{B}(y_\alpha, r_\alpha))}_\alpha)$ is called a \textbf{$(\Lambda, \sigma, \gamma)$-smooth model} if
        \begin{itemize}
            \item $S$ is a stationary integral $n$-varifold in $(\mathbb{R}^N, g_{\eucl})$ with 
            \[
                \theta_S(\mathbf{0}, \infty) \leq \Theta(\Lambda)\,,
            \]
            \item $\mbfC, \set{\mbfC_\alpha}_\alpha \subset \mathcal{C}_{N, n}(\Lambda)$,
            \item $\set{ \mathbb{B}(y_\alpha, 2r_\alpha)}$ is a finite collection of disjoint balls in $\mathbb{B}(1 - 3\sigma)$,
        \end{itemize}
        such that the following conditions are satisfied:
        \begin{enumerate}[label=(\roman*)]
            \item $\spt S$ is a smooth, closed, embedded minimal submanifold $\mathring{S}$ in $\mathbb{R}^N\setminus \set{y_\alpha}_\alpha$ such that for every $x \in \mathring{S}$,
            \[
                \mfr_{\mathring{S}, g_{\eucl}}(x) \cdot \min\{1, \dist_{g_{\eucl}}(x, \set{y_\alpha}_\alpha)\} \geq \Lambda^{-1}\,,
            \]
            where the regularity scale $\mfr_{\mathring{S}, g_\eucl}$ is that of Definition \ref{def:RegularScale};
            \item $S\llcorner( \mathbb{A}(\mathbf{0}, 1, \infty), g_{\eucl})$ is a $(\mbfC, \gamma)$-strong-cone region;
            \item for each $\alpha$, $\spt S \cap  \mathbb{A}(y_\alpha, 0, 2r_\alpha)$ is a $(\mbfC_\alpha, \gamma)$-strong-cone region.
        \end{enumerate}
        (Note that the function $\Theta(\Lambda)$ has been defined in Remark \ref{rem:DensityFunct}.)
        If there is no ambiguity, we will refer to the smooth model as $S$ for simplicity.
    \end{definition}
    \begin{remark}
        Intuitively, when a minimal submanifold is close to a cone outside some small ball, this approximation does not necessarily extend directly into the small ball. Fortunately, due to the volume monotonicity, on a small scale, if the minimal submanifold is modelled on another cone, the density of the new cone is controlled by that of the original one. The notion of smooth region, introduced below, aims at depicting, in $\BB(1) \setminus\bigcup_{\alpha} \BB(y_\alpha, r_\alpha / 4)$, the region between strong-cone regions at two different scales.
    \end{remark}

    \begin{definition}[Smooth model scale constant]\label{Def:smooth_model_const}
        Given a $(\Lambda, \sigma, \gamma)$-smooth model $S$, we let $\epsilon_S$ be the largest number $\leq \min(1, \min_\alpha\set{r_\alpha})$ for which the graph map 
        \[
            \graph_S: T^\perp(\mathring{S}) \rightarrow \mathbb{R}^N, \quad \graph_S(x, v):= x + v\,,
        \]
        is a diffeomorphism from $\set{(x,v) \in T^\perp(\mathring{S}): x \in \BB(2)\setminus \bigcup_\alpha \BB(y_\alpha, r_\alpha/8), |v| < 2\epsilon_S}$ onto its image, and satisfies
        \[
            \left|D\graph_S\vert_{(x,v)} - \id\right| \leq \epsilon_S^{-1}|v|\,.
        \]
(Here $T^\perp(\mathring{S})$ denotes the normal bundle to $S$ in Euclidean $\R^N$.)
        
    \end{definition}

    \begin{definition}[Smooth region]
        Given a smooth model $S$, a $C^2$ metric $g$ on $ \mathbb{B}(a, R) \subset \mathbb{R}^N$, and $\beta \in (0, 1)$, we say that an integral varifold $V$ in $(\mathbb{B}(a, R), g)$ is an \textbf{$(S, \beta)$-smooth region} if there is a $C^2$ function $u: \mathring{S} \rightarrow \mathring{S}^\perp$ so that 
        \[
            \left((\eta_{a, R})_\# V\right)\llcorner( \BB(1) \setminus \bigcup_\alpha  \mathbb{B}(y_\alpha, r_\alpha /4)) = \left[\graph_{\mathring{S}}(u)\cap  \BB(1) \setminus \bigcup_\alpha  \mathbb{B}(y_\alpha, {r_\alpha /4} )\right]_{R^{-2} \circ \eta^*_{a,R}(g)}\,,
        \]
        where $\eta_{a,R}$ is the dilation of center $a$ and scale $1/R$,
        and
        \[
            |u|_{C^2(\mathring{S})} \leq \beta \epsilon_{S}\,,
        \]
        where $\epsilon_S$ is a scale constant in the previous definition. 
        
        In this case, for simplicity, we will also call $\mathbb{B}(a, R)$ a smooth region for $V$.
    \end{definition}

    \begin{definition}[Local cone decomposition]\label{Def:cone_decomposition}
        Given $\Lambda \geq 1$, $\gamma, \beta \in \mathbb{R}$, $\sigma \in (0, 1/3)$, and $N_R \in \mathbb{N}^+$, we let 
        \begin{itemize}
            \item $g$ be a $C^2$ metric on $\mathbb{B}(x, R) \subset \mathbb{R}^N$;
            \item $V$ be an integral varifold in $(\mathbb{B}(x, R), g)$;
            \item $\cS = \set{S_s}_{s}$ be a finite collection of $(\Lambda, \sigma, \gamma)$-smooth models.
        \end{itemize}
        A \textbf{$(\Lambda, \beta, \cS, N_R)$-cone decomposition} of $V$ consists of the following parameters:
        \begin{itemize}
            \item Integers $N_C$, $N_S$ satisfying $N_C + N_S \leq N_R$, where $N_C$ is the number of strong-cone regions and $N_S$ the number of smooth regions;
            \item Points $\set{x_a}_a, \set{x_b}_b \subset \mathbb{B}(x, R)$, where $\set{x_a}$ are henceforth referred to as centers of strong-cone regions and $\set{x_b}$ as centers of smooth regions;
            \item Radii $\set{R_a, \rho_a \st R_a \geq 2 \rho_a}_a, \set{R_b}_b$, corresponding to radii of annuli in the definition of strong-cone regions and of balls in the definition of smooth regions, respectively;
            \item Cones $\set{\mbfC_a}_a \subset \cC_{N,n}(\Lambda)$;
            \item Indices $\set{s_b}_b$, corresponding to the smooth model $S_{s_b}$;
        \end{itemize}
        where $a = 1, \cdots, N_C$ and $b = 1, \cdots, N_S$. Such parameters determine a covering of balls and annuli satisfying:
        \begin{enumerate}[label=(\roman*)]
            \item every $V\llcorner(\mathbb{A}(x_a, \rho_a, R_a), g)$ is a $(\mbfC_a, \beta)$-strong cone region and every $V \llcorner (\mathbb{B}(x_b, R_b), g)$ is a $(S_{s_b}, \beta)$-smooth region;
            \item there is either a strong-cone region $\mathbb{A}(x_a, \rho_a, R_a)$ for $V$ with $R_a = R$ and $x_a = x$, or a smooth region $B_{R_b}(x_b)$ for $V$ with $R_b = R$ and $x_b = x$;
            \item if $V \llcorner{(\mathbb{A}(x_a, \rho_a, R_a),g)}$ is a $(\mbfC_a, \beta)$-strong-cone region and $\rho_a > 0$, then there exists either a smooth region $\mathbb{B}(x_b, R_b)$ for $V$ with $R_b = \rho_a$, or another cone region $\mathbb{A}(x_{a'}, \rho_{a'}, R_{a'})$ for $V$ with $R_{a'} = \rho_a, x_{a'} = x_a$. If $\rho_a = 0$, then $\theta_{\mbfC_a}(\mathbf{0}) > 1$;
            \item if $V\llcorner (\mathbb{B}(x_b, R_b), g)$ is a smooth region with $(S, \mbfC, \set{\mbfC_{\hat\alpha}, \BB(y_{\hat\alpha}, r_{\hat\alpha})}_{\hat\alpha})\in \cS$, then for any ${\hat\alpha}$, there exists a point $x_{b,{\hat\alpha}}$ and a radius $R_{b,{\hat\alpha}}$ satisfying
                \[
                    |x_{b,{\hat\alpha}} - (x_b + R_b \cdot y_{\hat\alpha})| \leq \beta R_b r_{\hat\alpha}, \quad \frac{1}{2}\leq \frac{R_{b,{\hat\alpha}}}{R_b r_{\hat\alpha}} \leq 1 + \beta\,,
                \]
                and either a strong-cone region $\mathbb{A}(x_{a'}, \rho_{a'}, R_{a'})$ for $V$ with $R_{a'} = R_{b, {\hat\alpha}}, x_{a'} = x_{b, {\hat\alpha}}$, or another smooth region $\mathbb{B}(x_{b'}, R_{b'})$ with $R_{b'} = R_{b, {\hat\alpha}}$ and $x_{b'} = x_{b, {\hat\alpha}}$.
        \end{enumerate}
    \end{definition}

    \begin{remark}
        In these definitions, we do not assume the stability of minimal cones. Instead, we impose a regularity scale condition, in the definition of $\mathcal{C}_{N,n}(\Lambda)$, to have uniform control on the geometry of the link 
        (cf. Section
\ref{sub:Compact}).
    \end{remark}

    Given a stationary integral varifold $V$ in $(\BB(1), g)$ with finitely many singular points $\Sing V$, for every $x \in \Reg V$, we define
    \[
        \rho_{V}(x) := \min\{1, \inf_{p \in \Sing V}\{\dist_g(x, p)\}\}\,.
    \]
    By Lemma \ref{Lem_cC_(N,n)(Lambda) finitely many density}, for each $\Lambda > 0$, there are only finitely many possible densities for regular minimal cones in $\mathcal{C}_{N, n}(\Lambda)$. Among various consequences, we have the following important local cone decomposition theorem.
    
    \begin{theorem}[Existence of local cone decomposition]\label{Thm:cone_decomposition}
        Given parameters $I \in \mathbb{N}$, $\Lambda \in \mathbb{N}^+$, $\sigma \in (0, \frac{1}{100 (\Lambda + 1)})$, $\gamma \in (0, 1)$, there exist constants $\delta$, $N_R$ and a finite collection of $(\Lambda, \sigma, \gamma)$-smooth models $\set{S_s}_s = \mathcal{S}$, all depending only on $(\Lambda, \sigma, \gamma)$ with the following property.
        
        For any $C^3$ metric $g$ on $\BB(1)$ satisfying $|g - g_{\eucl}|_{C^3(\BB(1))} \leq \delta$, any stationary integral varifold $V$ in $(\BB(1), g)$ with $\#\Sing V \leq I$, $\mathbf{C} \in \mathcal{C}_{N, n}(\Lambda)$, if
        \begin{enumerate}[label=(\roman*)]
            \item $\dist_H(\spt V \cap \BB(1), \mathbf{C} \cap \BB(1)) \leq \delta$,
            \item $\frac{1}{2} \theta_{\mathbf{C}}(0) \leq \theta_V(0, 1/2)$ and $\theta_V(0, 1) \leq \frac{3}{2} \theta_{\mathbf{C}}(0)$,
            \item for all $x \in \Reg V$, $\mfr_{\Reg V, g}(x)  \geq \Lambda^{-1}\rho_V(x)$,
        \end{enumerate}
        where $\dist_H$ stands for the Hausdorff distance, then there exists a radius $r \in (1 - 20\sigma, 1)$ so that $V \llcorner \BB(x,r)$ admits a $(\Lambda, \gamma, \mathcal{S}, N_R)$-cone decomposition.
    \end{theorem}
    \begin{proof}
        The proof follows verbatim that of~\cite[Theorem~7.1]{Edelen2021}. Note that by Allard's regularity theorem~\cite{All72}, the regularity scale condition $\mfr_{\Reg V, g}(x) \geq \Lambda^{-1}\rho_V(x)$ and the assumption  $\#\Sing V \leq I$ ensure that the convergence is smooth and occurs with multiplicity one outside a finite set, and thus, only Case 1 ($\mathcal{I} \subset \{0\}$) from the proof of~\cite[Theorem~7.1]{Edelen2021} will arise during the induction in our setting. 
    \end{proof}

\vspace{3mm}

\textbf{Tree representations of cone decomposition.}
    In the definition of the cone decomposition of a minimal submanifold near a cone, Definition~\ref{Def:cone_decomposition}, there is a natural hierarchy among the strong-cone regions and smooth regions based on their respective scales due to the volume monotonicity. Therefore, we can use a tree structure to represent the cone decomposition.
                
    \begin{definition}[Tree representation of a local cone decomposition]
        Given a $(\Lambda, \beta, \cS, N_R)$-cone decomposition of $V\llcorner \mathbb{B}(x, R)$ as in Definition \ref{Def:cone_decomposition} with parameters:
        \begin{itemize}
            \item Integers $N_S$, $N_C$ satisfying $N_S + N_C \leq N_R$;
            \item Points $\set{x_a}_a, \set{x_b}_b \subset \mathbb{B}(x, R)$;
            \item Radii $\set{R_a, \rho_a \st R_a \geq 2 \rho_a}_a, \set{R_b}_b$;
            \item Indices $\set{s_b}_b$;
            \item Cones $\set{\mathbf{C}_a}_a \subset \cC_{N, n}(\Lambda)$,
        \end{itemize}
        where $a = 1, \cdots, N_C$ and $b = 1, \cdots, N_S$.
        The corresponding \textbf{tree representation} of the cone decomposition is a rooted tree (in the sense of e.\,g. \cite[Section~B.5]{Cormen2022}) uniquely defined by the following requirements:
        \begin{enumerate}[label=(\roman*)]
            \item there are two types of nodes: every node of \textit{type I} is labeled with $(\mathbf{C}_a, x_a, R_a, \rho_a)$, while every node of \textit{type II} with $(S_{s_b}, x_b, R_b)$;
            \item the root is labeled with either $(\mathbf{C}_a, x_a = x, R_a = R, \rho_a)$ or $(S_{s_b}, x_b = x, R_b = R)$;
            \item for any type I node $(\mathbf{C}_a, x_a, R_a, \rho_a)$, either $\rho_a = 0$, $\theta_{\mbfC_a}(\mathbf{0}) > 1$ and it is a leaf; or $\rho_\alpha > 0$ and it has a unique child of either
                  \begin{itemize}
                      \item type I $(\mathbf{C}_{a'}, x_{a'} = x_a, R_{a'} = \rho_a, \rho_{a'})$, or
                      \item type II $(S_{s_{b'}}, x_{b'} = x_a, R_{b'} = \rho_a)$;
                  \end{itemize}
            \item for any type II node $(S_{s_b}, x_b, R_b)$ where $S_{s_b} = (S, \mathbf{C}, \set{\mathbf{C}_{{\hat\alpha}}, \mathbb{B}(y_{{\hat\alpha}}, r_{{\hat\alpha}})}_{{\hat\alpha}\in I_b})$, it has $\text{card}(I_b)$ child nodes such that for each ${\hat\alpha}$, there exists $R_{b,{\hat\alpha}}$ and $x_{b,{\hat\alpha}}$ such that
                  \[
                      |x_{b,{\hat\alpha}} - (x_b + R_b \cdot y_{\hat\alpha})| \leq \beta R_b r_{\hat\alpha}, \quad \frac{1}{2}\leq \frac{R_{b,{\hat\alpha}}}{R_b r_{\hat\alpha}} \leq 1 + \beta\,,
                  \]
                  so that the corresponding child node is either 
                  \begin{itemize}
                      \item of type I $(\mathbf{C}_{a'} = \mathbf{C}_{{\hat\alpha}}, x_{a'} = x_{b, {\hat\alpha}}, R_{a'} = R_{b, {\hat\alpha}}, \rho_{a'})$, or
                      \item of type II $(S_{s_{b'}}, x_{b'} = x_{b, {\hat\alpha}}, R_{b'} = R_{b, {\hat\alpha}})$.
                  \end{itemize}
        \end{enumerate}
        
        The \textbf{coarse tree representation} of the cone decomposition is obtained by relabeling the rooted tree above, i.e., replacing the type I node $(\mathbf{C}_a, x_a, R_a, \rho_a)$ by $(\theta_{\mathbf{C}_a}(\mathbf{0}))$, and the type II node $(S_{s_b}, x_b, R_b)$ by $S_{s_b}$.
    \end{definition}

    \begin{definition}
        For $\gamma \in (0, 1/100)$, two $(\Lambda, \beta, \cS, N_R)$-tree representations of local cone decompositions with parameters
        \begin{itemize}
            \item $(N_S, N_C, \set{x_a}, \set{x_b}, \set{R_a}, \set{\rho_a}, \set{R_b}, \set{\mbfC_a} \set{s_b})$,
            \item $(N'_S, N'_C, \set{x'_a}, \set{x'_b}, \set{R'_a}, \set{\rho'_a}, \set{R'_b}, \set{\mbfC'_a}, \set{s'_b})$,
        \end{itemize}
        are said to be \textbf{$\gamma$-close} if $N'_S = N_S$, $N'_C = N_C$, they have the same coarse tree representations, and in addition:
        \begin{enumerate}[label=(\roman*)]
            \item if the corresponding two nodes are both of type I, then
                \begin{itemize}
                    \item $\dist_H(\mathbf{C}_a \cap \partial B_1,\mathbf{C}_{a'} \cap \partial B_1) \leq \gamma$,
                    \item If $\rho_a > 0$, then
                        \begin{itemize}
                            \item $|\rho_a - \rho_{a'}|\leq \gamma \min(\rho_a, \rho_{a'})$;
                            \item $|x_a - x_{a'}| \leq \gamma \min(\rho_a, \rho_{a'})$;
                            \item $|R_a - R_{a'}| \leq \gamma \min(\rho_a, \rho_{a'})$;
                        \end{itemize}
                        otherwise, if $\rho_a = 0$, then
                        \begin{itemize}
                            \item $\rho_{a'} = 0$;
                            \item $|x_a - x_{a'}| \leq \gamma \min(R_a, R_{a'})$;
                            \item $|R_a - R_{a'}| \leq \gamma \min(R_a, R_{a'})$;
                        \end{itemize}
                \end{itemize}
            \item if the corresponding two nodes are both of type II, then
                \begin{itemize}
                \item $|x_b - x_{b'}| \leq \gamma \min(R_b, R_b') \min_{{\hat\alpha} \in I_b}(r_{\hat\alpha})$,
                \item $|R_b - R_{b'}| \leq \gamma \min(R_b, R_b') \min_{{\hat\alpha} \in I_b}(r_{\hat\alpha})$.
                \end{itemize}
        \end{enumerate}
    \end{definition}

    In order to deal with minimal submanifolds in a closed Riemannian manifold directly, we also introduce a large-scale cone decomposition for an $\MSI$ in a closed Riemannian manifold $M$ of dimension $N$.

    \begin{definition}[Large-scale cone decomposition]
        Given $\Lambda, \gamma, \beta \in \mathbb{R}_+$, $\sigma \in (0, 1/3)$, and $N_R \in \mathbb{N}$, let \begin{itemize}
            \item $g_0$, $g$ be two $C^3$ metrics on $M$;
            \item $\Sigma_0$, $\Sigma$ be two $\MSI$ in $(M, g_0)$ and $(M, g)$ respectively;
            \item $\cS = \set{S_s}_{s}$ be a finite collection of $(\Lambda, \sigma, \gamma)$-smooth models.
        \end{itemize}
        A \textbf{large-scale $(\Lambda, \beta, g_0, \Sigma_0, \mathcal{S}, N_R)$-cone decomposition} of $\Sigma$ consists of:
        \begin{itemize}
            \item a collection of radii $\set{r_{\hat\alpha}}_{\hat\alpha}$ corresponding to the singular set $\Sing(\Sigma_0) = \set{p_{\hat\alpha}}_{\hat\alpha}$, such that the metric balls $\set{B^g(p_{\hat\alpha}, r_{\hat\alpha})}_{\hat\alpha}$ are pairwise disjoint;
            \item a $(\Lambda, \beta, \cS, N_R)$-cone decomposition for each $|\Sigma| \llcorner B^g(p_{\hat\alpha}, r_{\hat\alpha})$;
            \item a $C^2$ normal section $u: \Sigma_0 \setminus \bigcup_{p_{\hat\alpha} \in \Sing(\Sigma_0)} B^g(p_{\hat\alpha}, {r_{\hat\alpha}/2}) \rightarrow \Sigma_0^\perp$ so that for $r_0 = \min_{\hat\alpha}\set{r_{\hat\alpha}} > 0$,
                \[
                    r_0^{-1}|u| + |\nabla u| + r_0 |\nabla^2 u| \leq \beta\,,
                \]
                and $|\Sigma| \llcorner B^g(p_{\hat\alpha}, r_{\hat\alpha})$ coincides with $\graph_{\Sigma_0}(u) \setminus B^g(p_{\hat\alpha}, {r_{\hat\alpha}})$ understood as varifold with multiplicity one).
        \end{itemize}
    \end{definition}

    Similarly, we can define the corresponding tree representation and the notion of $\gamma$-closeness.

    \begin{definition}[Tree representation of large-scale cone decomposition]
        Given a large-scale $(\Lambda, \beta, g_0, \Sigma_0, \cS, N_R)$-cone decomposition of an $\MSI$ $\Sigma$ in $(M, g)$ with parameters:
        \begin{itemize}
            \item $\Sing(\Sigma_0) = \set{p_{\hat\alpha}}_{\hat\alpha}$;
            \item radii $\set{r_{\hat\alpha}}_{\hat\alpha}$;
            \item $(\Lambda, \beta, \cS, N_R)$-cone decompositions for each 
            $|\Sigma| \llcorner B^g(p_{\hat\alpha}, r_{\hat\alpha})$;
        \end{itemize}
        The corresponding \textbf{tree representation} of the large-scale cone decomposition is a rooted tree uniquely defined by:
        \begin{enumerate}
            \item the root node is a labeled by a tuple $(\Sigma_0, g_0, \set{p_{\hat\alpha}}, \set{r_{\hat\alpha}})$;
            \item the root node has $\# \Sing(\Sigma_0)$ children, indexed by $\hat\alpha$. The corresponding subtree rooted at the $\hat\alpha$-child is the tree representation of the $(\Lambda, \beta, \cS, N_R)$-cone decomposition for each $|\Sigma|\llcorner B^g(p_{\hat\alpha}, {r_{\hat\alpha}})$.
        \end{enumerate}
        
        Similarly, the \textbf{coarse tree representation} will be the same directed rooted tree with the subtrees above replaced by their corresponding coarse trees.
    \end{definition}

    \begin{definition}
        For $\gamma \in (0, 1/100)$, two $(\Lambda, \beta, g_0, \Sigma_0, \cS, N_R)$-tree representations are said to be \textbf{$\gamma$-close} if 
        \begin{itemize}
            \item their root nodes have the same label;
            \item their subtrees corresponding to the $\hat\alpha$-child are $\gamma$-close for each $\hat\alpha$.
        \end{itemize}
    \end{definition}

    \begin{theorem}\label{thm:delta_nbhd-decomposition}
        In a closed smooth manifold $M$ of dimension $N$, for any given $(g, \Sigma) \in \cM^{k, \alpha}_n(M)$, $\Lambda\in\cR^{k,\alpha}_n(g,\Sigma)$, $\beta \in (0, 1/100)$, and $I\in\mathbb{N}$ with
        $\#\Sing(\Sigma) \leq I$
        there exists $\delta(g, \Sigma, \beta, \Lambda, I) > 0$ satisfying the following property.
        
        Defined a $\delta$-``neighborhood'' of $(g, \Sigma)$ as
        \[
            \begin{split}
                \cM^{k, \alpha}_n(g, \Sigma; \Lambda, I,\beta) := \Big\{ (g', \Sigma') \in \cM^{k,\alpha}_n(M): \|g'\|_{C^{k, \alpha}} \leq \Lambda,  
                \mfr_{\Sigma', g'} \geq \Lambda^{-1}\rho_{\Sigma',g'}, 
                \\
                \# \Sing(\Sigma') \leq I, \|g - g'\|_{C^k} + \mathbf{F}(|\Sigma|_g, |\Sigma'|_{g'}) \leq \delta(g, \Sigma, \beta, \Lambda, I)\
               \Big\} \,
            \end{split}
        \]
        \begin{enumerate}[label=(\roman*)]
            \item there exist a finite collection of $(\Lambda, \sigma, \beta)$-smooth models $\cS$ and an integer $N_R$, so that any $(g', \Sigma') \in \cM^{k, \alpha}_n(g, \Sigma; \Lambda, I, \beta)$ admits a large-scale $(\Lambda, \beta, g, \Sigma, \cS, N_R)$-cone decomposition;
            \item there exists a countable collection $\set{(g_v, \Sigma_v)}_{v \in \mathbb{N}} \subset \cM^{k, \alpha}_n(g, \Sigma; \Lambda, I, \beta)$ with fixed large-scale $(\Lambda, \beta, g, \Sigma, \cS, N_R)$-cone decompositions with the following property: every $(g', \Sigma') \in \mathcal{M}^{k, \alpha}_n(g, \Sigma; \Lambda, I, \beta)$ admits a large-scale $(\Lambda, \beta, g, \Sigma, \cS, N_R)$-cone decomposition whose tree representation is $\beta$-close to that of some $(g_v, \Sigma_v)$.
        \end{enumerate}
    \end{theorem}
    \begin{proof}
        The proof is essentially the same as that of \cite[Theorem 9.6]{LW22}.
    \end{proof}

    To proceed, define
    \[
        \cM^{k, \alpha}_n(M; \Lambda, I) := \Big\{ (g, \Sigma) \in \cM^{k,\alpha}_n(M), \Lambda\in\cR^{k,\alpha}_n(g,\Sigma), \
        \#\Sing(\Sigma) \leq I\Big\} \,.
    \]
    For any $\beta > 0$, it follows from the corresponding compactness statement that there exists a sequence of $\{(g_i, \Sigma_i)\}_i$ of $\cM^{k, \alpha}_n(M; \Lambda, I)$ such that
    \[
        \cM^{k, \alpha}_n(M; \Lambda, I) = \bigcup^\infty_{i = 1} \cM^{k, \alpha}_n(g_i, \Sigma_i; \Lambda, I, \beta)\,.
    \]

    With reference to the second point of Theorem~\ref{thm:delta_nbhd-decomposition}, for each $v$, we can define an ``\emph{intermediate canonical neighborhood}'' by declaring
    \[
        \cL^{k, \alpha}_0(g_v, \Sigma_v; \Lambda, I ,\beta)
    \]
    to consist of every pair $(g', \Sigma') \in \cM^{k, \alpha}_n(g, \Sigma; \Lambda, I,\beta)$ that admits a large-scale $(\Lambda, \beta, g, \Sigma, \cS, N_R)$-cone decomposition whose tree representation is $\beta$-close to that of $(g_v, \Sigma_v)$. (This definition should be compared with that of canonical neighborhood, see Definition \ref{Def_injrad, canonical neighb}).

    Therefore, we have
    \[\begin{aligned}
        \cM^{k,\alpha}_n(M) &= \bigcup^\infty_{I = 0} \bigcup^\infty_{\Lambda = 1} \cM^{k, \alpha}_n(M; \Lambda, I)\\
            &= \bigcup^\infty_{I = 0} \bigcup^\infty_{\Lambda = 1} \bigcup^\infty_{i = 1} \cM^{k, \alpha}_n(g_i, \Sigma_i; \Lambda, I, \beta)\\            
            &= \bigcup^\infty_{I = 0} \bigcup^\infty_{\Lambda = 1} \bigcup^\infty_{i = 1} \bigcup^\infty_{v = 1} \cL^{k, \alpha}_0(g_{i, v}, \Sigma_{i, v}; \Lambda, I ,\beta)\,.
    \end{aligned}\]

    After relabeling the subscripts of the various parameters in play (each varying in a countable set), we obtain
    \begin{equation}\label{Eqn:Snd_Decomp}
        \cM^{k,\alpha}_n(M) = \bigcup^\infty_{i = 1} \cL^{k, \alpha}_0(g_i, \Sigma_i; \Lambda_i, I_i ,\beta)\,.
    \end{equation}

\vspace{3mm}

\textbf{Second decomposition.}
    Following the arguments in \cite[Subsection 9.2]{LW22}, we can prove that an intermediate canonical neighborhood $\mathcal{L}^{k, \alpha}_0(g, \Sigma; \Lambda, I, \beta)$ is sequentially compact, and thus we have the following result.

      \begin{proposition}[Finite covering of $\cL^{k,\alpha}_0$]\label{prop:finite_covering}
        For any $g, \Sigma, \Lambda, I, \beta$ as in (\ref{Eqn:Snd_Decomp}) (with subscripts omitted) and for any positive function $\kappa: \cT^{k,\alpha}_n(M)\to \RR_+$ (not necessarily continuous), there exists a finite set of pairs $\set{(g_{v}, \Sigma_{v})}_v \subset \cL^{k, \alpha}_0(g, \Sigma; \Lambda, I ,\beta)$ such that
        \[
            \cL^{k, \alpha}_0(g, \Sigma; \Lambda, I ,\beta) \subset \bigcup^{V}_{v = 1} \cL^{k,\alpha}(g_{v}, \Sigma_{v};\Lambda, \kappa_{v})\,,
        \]
        where $\kappa_v = \kappa(g_{v}, \Sigma_{v}; \Lambda)$.
    \end{proposition}

\

\textbf{Proof of Theorem \ref{Thm_Countable Decomp}.}
    By (\ref{Eqn:Snd_Decomp}) and Proposition~\ref{prop:finite_covering}, we have
    \begin{align*}
        \cM^{k,\alpha}(M) & = \bigcup^\infty_{i = 1} \cL^{k, \alpha}_0(g_i, \Sigma_i; \Lambda_i, I_i ,\beta)                                    \\
                        & = \bigcup^\infty_{i = 1} \bigcup^{V_i}_{v = 1} \cL^{k,\alpha}(g_{i, v}, \Sigma_{i, v};\Lambda_i, \kappa_{i, v})\,,
    \end{align*}
    where $\kappa_{i, v} = \kappa(g_{i, v}, \Sigma_{i, v}; \Lambda_i)$.

    Hence, Theorem \ref{Thm_Countable Decomp} follows from relabeling the indices $\set{i, v}$.

\

\section*{Acknowledgements}
This project has received funding from the European Research Council (ERC) under the European Union’s Horizon 2020 research and innovation programme (grant agreement No. 947923). Parts of this article were finalized while A. C. was visiting the University of Chicago and then during various visits of the authors at the Simons Laufer Mathematical Sciences Institute (SLMath): the support
and excellent working conditions of both institutions are gratefully acknowledged. Y. L. was partially supported by an AMS-Simons travel grant.

\bibliography{biblio}

    \end{document}